\documentclass[12pt,reqno]{amsart}
\usepackage{amsfonts}

\usepackage{amssymb,amsmath,graphicx,amsfonts,euscript}
\usepackage{color}

\newtheorem{thm}{Theorem}[section]

\newtheorem{Pros}{Proposition}[section]
\newtheorem{lemma}{Lemma}[section]

\theoremstyle{definition}

\theoremstyle{remark}
\newtheorem{rem}{Remark}[section]

\allowdisplaybreaks \voffset=-0.2in \numberwithin{equation}{section}

\begin{document}
\bigskip

\centerline{\Large\bf  Global regularity and time decay for the SQG equation }

\smallskip
\centerline{\Large\bf with anisotropic fractional dissipation}

\smallskip

\centerline{\Large\bf   }

\medskip

\centerline{Zhuan Ye }

\medskip

\centerline{Department of Mathematics and Statistics, Jiangsu Normal University,
}
\medskip

\centerline{101 Shanghai Road, Xuzhou 221116, Jiangsu, PR China}

\medskip

\centerline{E-mail: \texttt{yezhuan815@126.com
}}

\bigskip

{\bf Abstract:}~~%
In this paper, we focus on the two-dimensional surface quasi-geostrophic equation with fractional horizontal dissipation and fractional vertical thermal diffusion. On the one hand, when the dissipation powers are restricted to a suitable range, the global regularity of the surface quasi-geostrophic equation is obtained by some anisotropic embedding and interpolation inequalities involving fractional derivatives. One the other hand, we obtain the optimal large time decay estimates for global weak solutions by an anisotropic interpolation inequality. Moreover, based on the argument adopted in establishing the global $\dot{H}^1$-norm of the solution, we obtain the optimal large time decay estimates for the above obtained global smooth solutions. Finally, the decay estimates for the difference between the full solution and the solution to the corresponding linear part are also derived.

{\vskip 1mm
 {\bf AMS Subject Classification 2010:}\quad 35B40; 35Q35; 35Q86; 76D03; 35R11.

 }

\vskip .3in
\section{Introduction}
This paper is devoted to investigating the following two-dimensional (2D) surface quasi-geostrophic (abbr. SQG) equation with fractional horizontal dissipation and fractional vertical thermal diffusion
\begin{equation}\label{SQG}
\left\{\aligned
&\partial_{t}\theta+(u \cdot \nabla)\theta+\Lambda_{x_{1}}^{2\alpha}\theta+\Lambda_{x_{2}}^{2\beta}\theta=0, \\
&\theta(x, 0)=\theta_{0}(x),
\endaligned\right.
\end{equation}
where $x=(x_{1},x_{2})\in \mathbb{R}^2, \, t>0$ and $\alpha\in (0,\,1],\beta\in (0,\,1]$ are real constants. The velocity $u$ is determined by the scalar $\theta$ via the formula
$$ {u}=(u_{1},\,u_{2})=\left(-\frac{\partial_{x_{2}}}{ \Lambda}\theta,\,\,\frac{\partial_{x_{1}}}{\Lambda}\,\theta\right)
=(-\mathcal
{R}_{2}\theta,\,\,\mathcal {R}_{1}\theta)\triangleq\mathcal {R}^{\perp}\theta,$$
where $\mathcal{R}_{1}, \mathcal
{R}_{2}$ are the standard 2D Riesz transforms.
The fractional operators $\Lambda_{x_{1}}\triangleq \sqrt{-\partial_{x_{1}}^{2}}$, $\Lambda_{x_{2}}\triangleq \sqrt{-\partial_{x_{2}}^{2}}$ and $\Lambda^{\eta}\triangleq(-\Delta)^{\frac{\eta}{2}}$ are defined through the Fourier transform, namely
$$
\widehat{\Lambda_{x_{1}}^{2\alpha}
f}(\xi)=|\xi_{1}|^{2\alpha}\hat{f}(\xi),\qquad \widehat{\Lambda_{x_{2}}^{2\beta}
f}(\xi)=|\xi_{2}|^{2\beta}\hat{f}(\xi),\qquad \widehat{\Lambda^{\eta}f}(\xi)=|\xi|^{\eta} \widehat{f}(\xi),
$$
where
$$\hat{f}(\xi)=\frac{1}{(2\pi)^{2}}\int_{\mathbb{{R}}^{2}}{e^{-ix\cdot\xi}f(x)\,dx}.$$
We make the convention that by $\alpha=0$ we mean
that there is no horizontal dissipation in \eqref{SQG}, and similarly $\beta=0$ means that there is no
vertical thermal diffusion in \eqref{SQG}.

\vskip .1in
In the case when $\alpha=\beta$, the dissipative term $\Lambda_{x_{1}}^{2\alpha}\theta+\Lambda_{x_{2}}^{2\beta}\theta$ can be regarded as the standard fractional Laplacian term $\Lambda^{2\alpha}\theta$ in form due to
$|\xi_{1}|^{2\alpha}+|\xi_{2}|^{2\alpha}\thickapprox |\xi|^{2\alpha}$. In this paper, without special emphasis, by ${\Lambda_{x_{1}}^{2\alpha}}\theta
+ {\Lambda_{x_{2}}^{2\alpha}}\theta$ we mean that $ {\Lambda^{2\alpha}}\theta$. As a result, when $\alpha=\beta$, \eqref{SQG} becomes the following classical fractional dissipative SQG equation
\begin{equation}\label{classSQG}
\left\{\aligned
&\partial_{t}\theta+(u \cdot \nabla)\theta+\Lambda^{2\alpha}\theta=0, \\
&{u}=\mathcal {R}^{\perp}\theta,\\
&\theta(x, 0)=\theta_{0}(x).
\endaligned\right.
\end{equation}
The SQG equation is an important model in geophysical fluid dynamics, which describes the evolution of a surface temperature field in a rotating and stratified fluid (see e.g.\cite{PG,CMT}). Besides the physical interpretation of the SQG equation, it also serves as a simplified model for the 3D Navier--Stokes equations (see \cite{CMT} for details). After the first mathematical
investigation initiated by Constantin-Majda-Tabak in \cite{CMT}, the SQG equation \eqref{classSQG} has been extensively studied over the past three decades, especially, its global regularity and large time behavior. Before recalling the results on the large time behavior, let us first mention some global regularity results about \eqref{classSQG}. In view of the underlying scaling invariance associated with \eqref{SQG} and the available maximum principle, it is customary to refer to the case $\alpha>\frac{1}{2}$ as subcritical, the case $\alpha=\frac{1}{2}$ as critical and the case $\alpha<\frac{1}{2}$ as supercritical. Actually, in the subcritical case the global wellposedness result for initial data in certain Sobolev spaces is well-known (e.g., \cite{CW3111,CADcmp,Resnick,JUNING}). In the critical case, Constantin, C$\rm\acute{o}$rdoba and Wu \cite{CCWiumj} gave a construction of the unique global regular solutions under a smallness assumption of $L^{\infty}$-norm of the initial data. However, the global regularity in the critical case for the general initial data turned out to be a more delicate problem, which was independently settled slightly more than a decade ago by Kiselev-Nazarov-Volberg \cite{KNV} via developing an original method called the "nonlocal maximum principle" and Caffarelli-Vasseur \cite{CV1} via exploiting the De Giorgi's iteration method and a novel extension. Two other elegant proofs resolving the critical problem have been obtained independently by Kiselev--Nazarov \cite{KN2} using the duality method, and Constantin--Vicol \cite{CV} using the "nonlinear maximum principle" method. However, in contrast with the subcritical and the critical case, the global regularity problem in the supercritical case remains an outstanding open problem (see \cite{Chaolicmp,CMZcmp,Dongl,JUNING,Miurah,Wuj001,Wuj002,WangZ} for small data global existence results and see \cite{Dabgga,CZelativ,Kis2011,Sil} for the eventual regularity of the global weak solutions). Very recently, it is worthwhile to point out that Buckmaster, Shkoller and Vicol \cite{Buccpam} have been able to make an important step toward this supercritical problem by showing the nonuniqueness for a class of solutions having negative Sobolev regularity. For some further works on the regularity results of the SQG equation \eqref{classSQG}, please refer to \cite{AbidiH,CCCGW,CCwind,CIWind,ConstantiW08,Lazarxue,MX1,SilVicolaz}, just name a few. As for the large time behavior, the SQG equation \eqref{classSQG} also has been studied in many works. The subcritical case was resolved by several works (see for example \cite{CW3111,YZhou,Schonbek2siam,DLarma}). For the critical case and supercritical case, the time decay estimates of weak solutions are available in \cite{CCWiumj,Resnick,Dongd,Nichesc11,Carrillo,Nichesc,Ferreirann}. We refer to the paper \cite{CTVCMP} concerning the critical case on the existence of a compact global attractor with a time-independent force, and see also references therein. We also refer to the recent papers \cite{Iwabuchi} on the analyticity and large time behavior for the critical Burgers equation and the
quasi-geostrophic equation.
However, among the existing literature regarding \eqref{SQG} with $\alpha\neq\beta$, very little has addressed the case $\alpha\neq\beta$.
Recently, the author with collaborators in \cite{wxyjmfm} proved the global regularity result of \eqref{SQG} with $\alpha=1,\,\beta=0$ or $\alpha=0,\,\beta=1$. Subsequently, the author \cite{Yenonli} established the global regularity when the dissipation powers are restricted to a suitable range, namely \eqref{sdf2334} below.
This paper is the attempt to deal with the global regularity and makes a detailed study on the large-time behavior of both the weak solution and the smooth solution to (\ref{SQG}).

\vskip .1in
The first goal of this paper is to establish the global regularity and the large time behavior of the SQG equation \eqref{SQG}. More precisely, our main result can be stated as follows.
\begin{thm}\label{OKTh1}
Let $\theta_{0}\in H^{\rho}(\mathbb{R}^{2})$ with $\rho\geq2$. If $\alpha,\,\beta\in (0,\,1)$ satisfy
\begin{equation}\label{sdf2334}
\beta>\left\{\aligned
&\frac{1}{2\alpha+1},\qquad 0<\alpha\leq \frac{1}{2},\\
&\frac{1-\alpha}{2\alpha},\ \qquad \frac{1}{2}<\alpha<1,
\endaligned\right.
\end{equation}
then (\ref{SQG}) has a unique global solution $\theta$ such that for any given $T>0$,
$$\theta\in C([0, T]; H^{\rho}(\mathbb{R}^{2})),\quad \Lambda_{x_{1}}^{\alpha}\theta,\ \Lambda_{x_{2}}^{\beta}\theta \in L^{2}([0, T]; H^{\rho}(\mathbb{R}^{2})).
$$
Additionally, if $\theta_{0}\in  L^{p}(\mathbb{R}^{2})$ with $p\in [1,\,2)$, then the solution as stated above admits the decay estimate for any $s\in [0,\rho]$
\begin{align}\label{fvbnm023}
 \|\Lambda^{s} \theta(t)\|_{L^2}\leq C_{0}(1+t)^{-\frac{(\alpha+\beta)(2-p)+2\min\{\alpha,\,\beta\}sp}{4\alpha\beta p}},
\end{align}
where the positive constant $C_{0}$ depends only on $\alpha,\,\beta,\,s,\,p$ and $\theta_{0}$.
\end{thm}

\begin{rem}
Actually, for the case $s=0$, \eqref{fvbnm023} is valid for all    $\alpha,\,\beta\in (0,\,1]$, without restriction \eqref{sdf2334} (see Proposition \ref{Prosty666} below for details). In the case when $\alpha=\beta$, our sharp decay rate \eqref{fvbnm023} coincides with the previous results in literature (see for example \cite{CW3111,Nichesc11,Nichesc}).
\end{rem}

\begin{rem}
Although the special critical case $\alpha=\beta=\frac{1}{2}$ (that is \eqref{classSQG} with $\alpha=\frac{1}{2}$) considered by
\cite{CV1,KNV,KN2,CV} is not covered by Theorem \ref{OKTh1}, yet it should be pointed out that
\begin{equation*}
 \left\{\aligned
&\frac{1}{2\alpha+1}+\alpha<1,\qquad 0<\alpha< \frac{1}{2},\\
&\frac{1-\alpha}{2\alpha}+\alpha<1,\ \qquad \frac{1}{2}<\alpha<1.
\endaligned\right.
\end{equation*}
This indicates that one of the indices is allowed to go below $\frac{1}{2}$, that
comes at the price of the other one being above it, but in fact the sum of the two
exponents is always smaller than $1$. This result may be of independent interest, because at this stage we note that to date the global regularity for the supercritical SQG equation (namely, \eqref{classSQG} with $\alpha<\frac{1}{2}$) with arbitrarily large initial data remains open.
\end{rem}

\begin{rem}
It seems that the strategies developed in \cite{CV1,KNV,KN2,CV} do not work for \eqref{SQG} with $\Lambda_{x_{1}} \theta
+ \Lambda_{x_{2}}\theta$. As stated above, at present we are not able to show the global regularity for \eqref{SQG} with $\Lambda_{x_{1}} \theta
+ \Lambda_{x_{2}}\theta$. However, if one assumes that the $L^{\infty}$-norm of the initial data $\theta_{0}$ are suitable small, then the corresponding \eqref{SQG} indeed admits a unique global classical solution. Moreover, the decay
 estimate also holds true. More details are provided in Appendix \ref{appdexA}.
\end{rem}

\begin{rem}
We remark that the global regularity part of Theorem \ref{OKTh1} was derived by \cite{Yenonli} via establishing a new refined logarithmic-type Gronwall inequality (see \cite[Lemma 2.1]{Yenonli} for details), which helps us to bypass one of the main difficulties that Riesz transform does not map continuously from $L^{\infty}$ space to $L^{\infty}$ space. More precisely, it plays an important role in deriving the global $\dot{H}^1$-bound of the solution. Here, we provide an alternative proof of the global $\dot{H}^1$-bound by some anisotropic embedding and interpolation inequalities involving fractional derivatives, without the use of the above mentioned logarithmic-type Gronwall inequality. We also point out that this new approach is both more direct and more explicitly quantitative than \cite{Yenonli}, which allows us to further investigate the decay estimate of the solutions. Actually, the global smooth solutions in the aforementioned reference \cite{Yenonli} may grow in time due to the application of Gronwall type inequalities.
\end{rem}

\begin{rem}
Let us outline the main ideas and difficulties in the proof of Theorem \ref{OKTh1}. By some sharp anisotropic embedding and interpolation inequalities involving fractional derivatives, we first derive the global ${\dot{H}}^1$-bound of the solutions, which allows us to show the global ${\dot{H}}^2$-bound, and thus show the global ${\dot{H}}^s$-bound. This thus finishes the proof of the global regularity part of Theorem \ref{OKTh1}. To prove the time decay part of Theorem \ref{OKTh1}, we first derive the basic $L^2$-decay rate via the use of the direct interpolation inequality (see Proposition \ref{Prosty666} for details). Next, by the appropriately modified Fourier Splitting method and the argument adopted in establishing the global $\dot{H}^1$-norm of the solution, we are able to show the sharp decay estimate of the solution in $\dot{H}^{1}$-norm (see Proposition \ref{fggyvby7} for details), which enables us to derive the uniform-in-time of the $\dot{H}^2$-bound of $\theta$. However, the estimates obtained at present are not sufficient to show show the decay estimate of the solution in $\dot{H}^{s}$-norm for $s>1$ because we consider all the case $p\in[1,2)$ other than $p\in[1,p_{0})$ for some $p_{0}\in[1,2)$. In order to achieve this goal, we turn to prove the sharp decay estimate of the solution in $\dot{H}^{2}$-norm by using the preliminary decay rate of the $\dot{H}^2$-norm and the sharp decay estimate of the solution in $\dot{H}^{1}$-norm (see Proposition \ref{fxcdfn9} for details).
Finally, with the help of the sharp decay estimate of the solution in $\dot{H}^{2}$-norm, the sharp decay estimate of the solution in $\dot{H}^{s}$-norm for any $s>2$ follows by the modified Fourier Splitting method (see Proposition \ref{dfdsfp37} for details). Therefore, this completes the proof of the time decay part of Theorem \ref{OKTh1}.
\end{rem}

\begin{rem}
 The corresponding linear part of \eqref{SQG}, namely,
\begin{equation}\label{QGssd}
\left\{\aligned
&\partial_{t}\widetilde{\theta}+\Lambda_{x_{1}}^{2\alpha}\widetilde{\theta}
+\Lambda_{x_{2}}^{2\beta}\widetilde{\theta}=0, \\
&\widetilde{\theta}(x, 0)=\theta_{0}(x)
\endaligned\right.
\end{equation}
 admits the decay estimate
\begin{align}\label{linaddtq1}
\|\Lambda^{s}\widetilde{\theta}(t)\|_{L^2}\leq C_{0}(1+t)^{-\frac{(\alpha+\beta)(2-p)+2\min\{\alpha,\,\beta\}sp}{4\alpha\beta p}}
\end{align}
for $\alpha\in (0,1]$, $\beta\in (0,1]$ and $\theta_{0}\in \dot{H}^{s}(\mathbb{R}^{2})\cap L^{p}(\mathbb{R}^{2})$ with $s\geq0$ and $p\in [1,\,2)$. Moreover, if $\theta_{0}$ satisfies additionally that for some positive constant $c$
\begin{align}\label{linaddtq2}
\lim_{\rho\rightarrow0^{+}}\rho^{-\frac{(\alpha+\beta)(2-p)+2\min\{\alpha,\,\beta\}sp}{2\alpha\beta p}}\int_{E(\rho)} |\xi|^{2s}|\widehat{\theta_{0}}(\xi)|^{2}\;d\xi=c,
\end{align}
where $E(\rho)$ is given by
$$E(\rho)=\{\xi\in \mathbb{R}^2: \  |\xi_{1}|^{2\alpha}+|\xi_{2}|^{2\beta}\leq \rho\},$$
then the solution $\widetilde{\theta}$ admits the same order for both lower bound and upper bound, namely,
\begin{align}\label{linaddtq3}
C_{1}(1+t)^{-\frac{(\alpha+\beta)(2-p)+2\min\{\alpha,\,\beta\}sp}{4\alpha\beta p}}\leq\|\Lambda^{s}\widetilde{\theta}(t)\|_{L^2}\leq C_{2}(1+t)^{-\frac{(\alpha+\beta)(2-p)+2\min\{\alpha,\,\beta\}sp}{4\alpha\beta p}}.
\end{align}
For the sake of completeness, the proof are provided in Appendix \ref{appdexB}.
Consequently, the decay rate obtained in Theorem \ref{OKTh1} is in line with this
for the solution of the linear equation \eqref{QGssd}. This implies that the decay rate obtained in Theorem \ref{OKTh1} is optimal.
\end{rem}

If the condition $\theta_{0}\in L^{p}(\mathbb{R}^{2})$ with $p\in [1,\,2)$ is dropped from Theorem \ref{OKTh1}, then we have the following theorem.
\begin{thm}\label{Lderq1}
Let $\theta_{0}\in H^{\rho}(\mathbb{R}^{2})$ with $\rho\geq2$ and $\alpha,\,\beta\in (0,\,1)$ satisfy \eqref{sdf2334}, then the global solution derived in Theorem \ref{OKTh1} admits the decay estimate for any $s\in [0,\rho]$
\begin{align}\label{adssseq11}
 \|\Lambda^{s} \theta(t)\|_{L^2}\leq C_{0}(1+t)^{-\frac{ \min\{\alpha,\,\beta\}s }{2\alpha\beta }},
\end{align}
where the positive constant $C_{0}$ depends only on $\alpha,\,\beta,\,s$ and $\theta_{0}$.
\end{thm}

The third result concerns the remainder case: $$\digamma\triangleq\{(\alpha,\beta):\,(0,1]\times (0,1]\setminus \Upsilon\},$$
where $\Upsilon\triangleq\{(\alpha,\beta):\,\alpha\in (0,1], \beta\in (0,1] \,\,\mbox{satisfy}\,\, \eqref{sdf2334}\}$. In this case, we prove the small data
global existence and large time behavior. The precise statement is provided in the following theorem.
\begin{thm}\label{okspcrca11}
Let $(\alpha,\,\beta)\in \digamma$ and $\theta_{0}\in {H}^{s}(\mathbb{R}^{2}) \cap \dot{H}^{2-\frac{4\alpha\beta}{\alpha+\beta}}(\mathbb{R}^{2})$ with $s\geq0$. If there exists a small constant $\epsilon>0$ such that
\begin{align}\label{okssfgt01}
\|\Lambda^{2-\frac{4\alpha\beta}{\alpha+\beta}}\theta_{0}
\|_{L^{2}}\leq\epsilon,
\end{align}
then (\ref{SQG}) has a unique global solution $\theta$ such that
$$\theta\in C([0, \infty); H^{s}(\mathbb{R}^{2}))\cap C([0, \infty); \dot{H}^{2-\frac{4\alpha\beta}{\alpha+\beta}}(\mathbb{R}^{2})),$$ $$\Lambda_{x_{1}}^{\alpha}\theta,\ \Lambda_{x_{2}}^{\beta}\theta \in L^{2}([0, \infty); H^{s}(\mathbb{R}^{2}))\cap L^{2}([0, \infty); \dot{H}^{2-\frac{4\alpha\beta}{\alpha+\beta}}(\mathbb{R}^{2})).
$$
Additionally, if $\theta_{0}\in L^{p}(\mathbb{R}^{2})$ with $p\in [1,\,2]$, then the solution as stated above admits the decay estimate
\begin{align}\label{oksffpp11}
  \|\Lambda^{s} \theta(t)\|_{L^2}\leq C_{0}(1+t)^{-\frac{(\alpha+\beta)(2-p)+2\min\{\alpha,\,\beta\}sp}{4\alpha\beta p}},\quad \forall s\geq0,
\end{align}
where the positive constant $C_{0}$ depends only on $\alpha,\,\beta,\,s,\,p$ and $\theta_{0}$.
\end{thm}
\begin{rem}
The small data global well-posedness of Theorem \ref{okspcrca11} for the case $\alpha=\beta<\frac{1}{2}$ corresponds to the
result of Ju \cite{JUNING} as well as Miura \cite{Miurah} and improves the
result of Chae and Lee \cite{Chaolicmp} in the critical Besov space $\dot{B}_{2,1}^{2-2\alpha}$. We point out that the smallness assumption \eqref{okssfgt01} is made only on the homogeneous norm of the initial data.
\end{rem}

\begin{rem}
We point out that for the case $s=0$, \eqref{oksffpp11} holds true without the restriction \eqref{okssfgt01} (see Proposition \ref{Prosty666} below for details).
\end{rem}

\vskip .1in
The following result concerns the decay of the difference $\theta(t)-\widetilde{\theta}(t)$, where $\theta(t)$ is the solution to (\ref{SQG}) and $\widetilde{\theta}(t)$ is the solution to (\ref{QGssd}).
More precisely, we have the following difference decay result.
\begin{thm}\label{Th3} Let $\alpha\in (0,1]$, $\beta\in (0,1]$ and $\theta_{0}\in L^{2}(\mathbb{R}^{2})\cap L^{p}(\mathbb{R}^{2})$ with $p\in [1,\,2)$, then the difference between \eqref{SQG} and \eqref{QGssd} satisfies
\begin{align*}
 &\|\theta(t)-\widetilde{\theta}(t)\|_{L^2}
\nonumber \\ &\leq
  \left\{
  \begin{aligned}
   &C_{0}(1+t)^{-\frac{\min\{\alpha+(p+1)\beta,\, (p+1)\alpha+\beta\}}{4\alpha\beta p}},  \qquad  \qquad \qquad \quad \quad \ \  \mbox{\rm if}\ \ p=\frac{2(\alpha+\beta)}{2\alpha\beta+\alpha+\beta}, \\
   &C_{0}(1+t)^{-\frac{\min\{\min\{\alpha+3\beta,\, 3\alpha+\beta\}p,\  \mathcal{A}\}}{4\alpha\beta p}},\qquad \qquad   \qquad \ \ \, \quad \ \mbox{\rm if} \   \ p<\frac{2(\alpha+\beta)}{2\alpha\beta+\alpha+\beta},\\
   &C_{0}(1+t)^{-\frac{\min\{\min\{\alpha+3\beta,\, 3\alpha+\beta\}p+2(\alpha+\beta)(2-p)-4\alpha\beta p,\ \mathcal{A}\}}{4\alpha\beta p}}, \ \qquad \textrm{\rm if}\   \ p>\frac{2(\alpha+\beta)}{2\alpha\beta+\alpha+\beta},
  \end{aligned}
  \right.
\end{align*}
where $\mathcal{A}>0$ is given by
\begin{equation}\label{sdfg68}\mathcal{A}=\min\{\alpha+(p+1)\beta,\, (p+1)\alpha+\beta\}+2(\alpha+\beta)-(\alpha+\beta+2\alpha\beta)p,\end{equation}
and the positive constant $C_{0}$ depends only on $\alpha,\,\beta,\,p$ and $\theta_{0}$.
\end{thm}

\begin{rem}
In the case when $\alpha=\beta$ and $p=1$, our rate $(1+t)^{-\frac{\min\{4,\,5-4\alpha\}}{4\alpha}}$ decays faster than $(1+t)^{-\frac{2-\alpha}{2\alpha}}$ derived by \cite[Theorem 4.3]{CW3111} and \cite[Theorem 3.2(2)]{Nichesc}.
Moreover, via the tedious computations, we can check by comparing the rates in Theorem \ref{Th3} and Theorem \ref{OKTh1} with $s=0$ that $\theta-\widetilde{\theta}$ decays faster than $\theta$ does.
\end{rem}

\begin{rem}
By direct calculation, one may check that
$$\frac{(\alpha+\beta)(2-p)}{4\alpha\beta p}<\frac{\min\{\alpha+(p+1)\beta,\, (p+1)\alpha+\beta\}}{4\alpha\beta p},$$
$$\frac{(\alpha+\beta)(2-p)}{4\alpha\beta p}<\frac{\min\{\min\{\alpha+3\beta,\, 3\alpha+\beta\}p,\  \mathcal{A}\}}{4\alpha\beta p},$$
$$\frac{(\alpha+\beta)(2-p)}{4\alpha\beta p}<\frac{\min\{\min\{\alpha+3\beta,\, 3\alpha+\beta\}p+2(\alpha+\beta)(2-p)-4\alpha\beta p,\ \mathcal{A}\}}{4\alpha\beta p}.$$
Consequently, according to the reverse triangle inequality
$$\|\theta(t)\|_{L^2}\geq \|\widetilde{\theta}(t)\|_{L^2}-\|\theta(t)-\widetilde{\theta}(t)\|_{L^2}$$
and the left hand side of the estimate \eqref{linaddtq3} with $s=0$ as well as Theorem \ref{Th3}, it leads to a lower bound
$$\|\theta(t)\|_{L^2}\geq C_{3}(1+t)^{-\frac{(\alpha+\beta)(2-p)}{4\alpha\beta p}}.$$
This along with \eqref{fvbnm023} with $s=0$ implies that the solution $\theta$ of (\ref{SQG}) admits the same order for both lower bound and upper bound, namely,
\begin{align}
C_{3}(1+t)^{-\frac{(\alpha+\beta)(2-p)}{4\alpha\beta p}}\leq\|\theta(t)\|_{L^2}\leq C_{4}(1+t)^{-\frac{(\alpha+\beta)(2-p)}{4\alpha\beta p}}.\nonumber
\end{align}
\end{rem}

\vskip .1in
Finally, we state the following results concerning the decay of the difference $\theta(t)-\widetilde{\theta}(t)$ with only $\theta_{0}\in L^{2}(\mathbb{R}^{2})$.
\begin{thm}\label{sdfthe61}
Let $\min\{\alpha+3\beta,\,3\alpha+\beta\}>4\alpha\beta$ with $\alpha>0,\,\beta>0$ and $\theta_{0}\in L^{2}(\mathbb{R}^{2})$, then the difference between \eqref{SQG} and \eqref{QGssd} satisfies
\begin{align}\label{dfcnmza1}
 \|\theta(t)-\widetilde{\theta}(t)\|_{L^2}\leq C_{0}(1+t)^{-\frac{ \min\{\alpha+3\beta,\, 3\alpha+\beta\}-4\alpha\beta}{8\alpha\beta}},
\end{align}
where the positive constant $C_{0}$ depends only on $\alpha,\,\beta$ and $\theta_{0}$.
\end{thm}

\begin{rem}
Obviously, comparing the result in Proposition \ref{mkjqq1} below and Theorem \ref{sdfthe61}, we see that $\theta-\widetilde{\theta}$ decays faster than $\theta$ does.
\end{rem}

\begin{rem}
We mention that the above theorems \ref{OKTh1}, \ref{Lderq1}, \ref{okspcrca11}, \ref{Th3} and \ref{sdfthe61} also hold true for a more general case: $u=\mathbf{T}[\theta]$, where $\mathbf{T}$ is a divergence
free zero order operator. For example, Theorems \ref{OKTh1}, \ref{Lderq1}, \ref{okspcrca11} and  \ref{Th3} are true for the following 2D incompressible porous
medium equation
\begin{equation*}
\left\{\aligned
&\partial_{t}\theta+( {u}\cdot\nabla)\theta+
\Lambda_{x_{1}}^{2\alpha}\theta+ \Lambda_{x_{2}}^{2\beta}\theta=0,\\
& {u}=-\nabla p-\theta e_{2},\\
&\nabla\cdot {u}=0,\\
&\theta(x, 0)=\theta_{0}(x).
\endaligned\right.
\end{equation*}
In fact, ${u}=-\nabla p-\theta e_{2}$ and $\nabla\cdot {u}=0$ imply
$$ {u}=(-\mathcal{R}_{1}\mathcal{R}_{2}\theta,\,\,\mathcal
{R}_{1}\mathcal {R}_{1}\theta).$$
With the above relation  in hand, following the same lines in proving the above theorems, one may obtain the same results.
\end{rem}

 \vskip .2in
The rest of this paper is organized as follows. Section \ref{sec2} is devoted to the proof of the global regularity part of Theorem \ref{OKTh1}. Section \ref{sec3} proves the time decay part of Theorem \ref{OKTh1}. Next, we give the proof of Theorem \ref{Lderq1} in Section \ref{dfcse11}.
The proof of Theorem \ref{okspcrca11} and Theorem \ref{Th3} is presented in Section \ref{secadd1} and Section \ref{sec4}, respectively.
We sketch the proof of Theorem \ref{sdfthe61} in Section \ref{yetee1}.
In Appendix \ref{appdexA}, we show the small data global existence and the decay estimate for the case \eqref{SQG} with $\alpha=\beta=\frac{1}{2}$. The decay estimates of linear part are also provided in Appendix \ref{appdexB}. Throughout this paper, we use the following convention regarding constants: $C$ shall denote a positive,
sufficiently large constant, whose value may change from line to line; to emphasize the dependence of a constant on some certain quantities $\gamma_{1},\gamma_{2},\cdot\cdot\cdot,\gamma_{n}$, we write $C(\gamma_{1},\gamma_{2},\cdot\cdot\cdot,\gamma_{n})$.

\vskip .3in
\section{The proof of Theorem \ref{OKTh1}: Global regularity }\setcounter{equation}{0}
\label{sec2}

The main contribution of this section is to give a transparent and new proof of global regularity for the SQG equation \eqref{SQG} under the assumption \eqref{sdf2334}, which relies solely on some anisotropic embedding and interpolation inequalities involving fractional derivatives. The argument of the proof of the global regularity plays a crucial role in deriving the decay estimates.\\

To start, we present several inequalities to be used in the subsequent sections. The first one is the classical commutator estimates (see, e.g., \cite{KPV2}).
\begin{lemma} If $s>0$, $p, p_1, p_3\in (1, \infty)$ and $p_2, p_4\in [1,\infty]$ satisfy
$$
\frac1p =\frac1{p_1} + \frac1{p_2} = \frac1{p_3} + \frac1{p_4},
$$
then there exists some constants $C$ such that
\begin{eqnarray}\label{yzz1}
  \|[\Lambda^s, f]g\|_{L^p}\leq C \left(\|\Lambda^s f\|_{L^{p_1}}\, \|g\|_{L^{p_2}} + \|\Lambda^{s-1} g\|_{L^{p_3}}\,\|\nabla f\|_{L^{p_4}}\right).
\end{eqnarray}
If further $s\in(0,1)$, then it holds true
\begin{eqnarray}\label{tvghgf}
\|\Lambda^s (f\,g) - f\, \Lambda^s g- g\,\Lambda^s f\|_{L^p}
\leq C \|\Lambda^s f\|_{L^{p_{1}}} \|g\|_{L^{p_{2}}}.
\end{eqnarray}
In particular, \eqref{tvghgf} implies
\begin{eqnarray}\label{yzz3}
\|[\Lambda^s, f]g\|_{L^p}
\leq C\|\Lambda^s f\|_{L^{p_{1}}} \|g\|_{L^{p_{2}}},
\end{eqnarray}
where the brackets $[,]$ represent the commutator.
\end{lemma}

\vskip .1in
The following interpolation inequality
will be frequently used later (see \cite[Lemma 2.2]{Yenonli}).
\begin{lemma}
If $0\leq\gamma\leq\varrho$, then it hold true for $i=1,\,2$
\begin{eqnarray}\label{avbnuy8911}
\|\Lambda_{x_{i}}^{\gamma}f\|_{L^{2}}\leq C\|f\|_{L^{2}}^{1-\frac{\gamma}{\varrho}}\|\Lambda_{x_{i}}
^{\varrho}f\|_{L^{2}}^{\frac{\gamma}{\varrho}}.
_{}\end{eqnarray}
\end{lemma}

\vskip .1in
In order to obtain the higher regularity, we need to bound the triple product in terms of the Lebesgue norms of the functions and their directional derivatives (see \cite[Lemma 2.4]{Yenonli}).
\begin{lemma}
Let $f\in L_{x_{2}}^{q}L_{x_{1}}^{p}(\mathbb{R}^2)$ for $p,\,q\in[2,\,\infty]$. If $g, h, \Lambda_{x_{1}}^{\gamma_{1}}g,  \Lambda_{x_{2}}^{\gamma_{2}}h \in L^2(\mathbb{R}^2)$ for any $\gamma_{1}\in (\frac{1}{p},\,1]$ and $\gamma_{2}\in (\frac{1}{q},\,1]$, then it holds true
\begin{equation}\label{qtri}
 \int_{\mathbb{R}}\int_{\mathbb{R}}| f \, g\, h|  \;dx_{1}dx_{2}  \le C \, \|f\|_{L_{x_{2}}^{q}L_{x_{1}}^{p}} \, \|g\|_{L^2}^{1-\frac{1}{\gamma_{1}p}}\|\Lambda_{x_{1}}^{\gamma_{1}}g\|_{L^2}^{ \frac{1}{\gamma_{1}p}} \,\|h\|_{L^2}^{1-\frac{1}{\gamma_{2}q}}\|\Lambda_{x_{2}}^{\gamma_{2}}h\|_{L^2}^{ \frac{1}{\gamma_{2}q}},
\end{equation}
where here and in sequel, we use the notation
$$\|h\|_{L_{x_{2}}^{q}L_{x_{1}}^p}\triangleq\Big( \int_{\mathbb{R}}\|h(.,x_{2})\|_{L_{x_{1}}^p}^{q} \,dx_{2}\Big)^{\frac{1}{q}}.$$
In particular, if $f, g,  h\in L^2(\mathbb{R}^2)$ and $\Lambda_{x_{1}}^{\gamma_{1}}g, \Lambda_{x_{2}}^{\gamma_{2}}h \in L^2(\mathbb{R}^2)$ for any $\gamma_{1},\ \gamma_{2}\in (\frac{1}{2},\,1]$, then it holds
\begin{equation}
 \int_{\mathbb{R}}\int_{\mathbb{R}}| f \, g\, h|  \;dx_{1}dx_{2}  \le C \, \|f\|_{L^2} \, \|g\|_{L^2}^{1-\frac{1}{2\gamma_{1}}}\|\Lambda_{x_{1}}^{\gamma_{1}}g\|_{L^2}^{ \frac{1}{2\gamma_{1}}} \,\|h\|_{L^2}^{1-\frac{1}{2\gamma_{2}}}\|\Lambda_{x_{2}}^{\gamma_{2}}h\|_{L^2}^{ \frac{1}{2\gamma_{2}}},\nonumber
\end{equation}
where $C$ is a constant depending on $\gamma_{1}$ and $\gamma_{2}$ only.
\end{lemma}

 \vskip .1in
Next, we will show the anisotropic interpolation inequality.
\begin{lemma}\label{gfhgj}
Let $\delta_{i}>0, \varepsilon_{i}\geq0, \gamma_{i}\geq0$ ($i=1,2$) satisfy
$$\mu\triangleq\frac{\delta_{2}\gamma_{1}-\delta_{1}\gamma_{2}}{\varepsilon_{2}\gamma_{1}-
\varepsilon_{1}\gamma_{2}}\geq 0,\ \ \lambda\triangleq\frac{\varepsilon_{2}\delta_{1}-\varepsilon_{1}\delta_{2}}{\varepsilon_{2}\gamma_{1}-
\varepsilon_{1}\gamma_{2}}\geq 0,\ \  \mu+\lambda\leq1.$$
Then the following anisotropic interpolation inequality is true
\begin{eqnarray}\label{t2326001}
 \|\Lambda_{x_{1}}^{\delta_{1}}\Lambda_{x_{2}}^{\delta_{2}}f\|_{L^{2}}\leq
C\|f\|_{L^{2}}^{1-\mu-\lambda}
 \|\Lambda_{x_{1}}^{\varepsilon_{1}}\Lambda_{x_{2}}^{\varepsilon_{2}}f\|_{L^{2}}^{\mu}
 \|\Lambda_{x_{1}}^{\gamma_{1}}\Lambda_{x_{2}}^{\gamma_{2}}f\|_{L^{2}}^{\lambda}.
\end{eqnarray}
A special consequence of \eqref{t2326001} is that if $\varepsilon_{1}=\gamma_{2}=0$ and $\delta_{1}>0, \delta_{2}>0, \varepsilon_{2}>0, \gamma_{1}>0$ satisfy $\frac{\delta_{2} }{\varepsilon_{2} }+\frac{\delta_{1} }{\gamma_{1} }=1$, then it holds true
\begin{eqnarray}\label{xcvyadi1}
 \|\Lambda_{x_{1}}^{\delta_{1}}\Lambda_{x_{2}}^{\delta_{2}}f\|_{L^{2}}\leq
C
 \| \Lambda_{x_{2}}^{\varepsilon_{2}}f\|_{L^{2}}^{\frac{\delta_{2} }{\varepsilon_{2} }}
 \|\Lambda_{x_{1}}^{\gamma_{1}} f\|_{L^{2}}^{\frac{\delta_{1} }{\gamma_{1} }}.
\end{eqnarray}
\end{lemma}
\begin{proof}
By Plancherel's theorem and the H$\rm \ddot{o}$lder inequality, we can check that
\begin{align}
&\|\Lambda_{x_{1}}^{\delta_{1}}\Lambda_{x_{2}}^{\delta_{2}}f(x_{1},x_{2})
\|_{L^{2}}
 \nonumber\\&=\Big(\int_{\mathbb{R}^{2}}{|\xi_{1}|^{2\delta_{1}}|\xi_{2}|^{2\delta_{2}}
|\widehat{f}(\xi)|^{2}\,d\xi}\Big)^{\frac{1}{2}}\nonumber\\
& =\Big(\int_{\mathbb{R}^{2}}{\big(|\xi_{1}|^{2\varepsilon_{1}\mu}
 |\xi_{2}|^{2\varepsilon_{2}\mu}
|\widehat{f}(\xi)|^{2\mu}\big)\big(|\xi_{1}|^{2\gamma_{1}\lambda}|\xi_{2}|^{
2\gamma_{2}\lambda}
|\widehat{f}(\xi)|^{2\lambda}\big)} |\widehat{f}(\xi)|^{2(1-\mu-\lambda)}\,d\xi\Big)^{\frac{1}{2}}
\nonumber\\
 &\leq C\Big(\int_{\mathbb{R}^{2}}{|\xi_{1}|^{2\varepsilon_{1}}
 |\xi_{2}|^{2\varepsilon_{2} }
|\widehat{f}(\xi)|^{2}\,d\xi}\Big)^{\frac{\mu}{2}}
\Big(\int_{\mathbb{R}^{2}}{|\xi_{1}|^{2\gamma_{1} }|\xi_{2}|^{
2\gamma_{2} }
|\widehat{f}(\xi)|^{2}\,d\xi}\Big)^{\frac{\lambda}{2}}\nonumber\\& \quad
\times
\Big(\int_{\mathbb{R}^{2}}{
|\widehat{f}(\xi)|^{2}\,d\xi}\Big)^{\frac{1}{2}(1-\mu-\lambda)}
\nonumber\\
& =C\|f\|_{L^{2}}^{1-\mu-\lambda}
 \|\Lambda_{x_{1}}^{\varepsilon_{1}}\Lambda_{x_{2}}^{\varepsilon_{2}}f\|_{L^{2}}^{\mu}
 \|\Lambda_{x_{1}}^{\gamma_{1}}\Lambda_{x_{2}}^{\gamma_{2}}f\|_{L^{2}}^{\lambda}.\nonumber
\end{align}
This completes the proof of Lemma \ref{gfhgj}.
\end{proof}

\vskip .1in
The following sharp anisotropic
interpolation inequality will be also used later.
\begin{lemma} \label{dfdasfwqew33}
Let $0\leq\sigma<\delta$, $1\leq p<r<q\leq\infty$ satisfy $$\frac{1}{r}=\frac{1}{q}\left(1-\frac{\sigma}{\delta}\right)+\frac{1}{p}
\frac{\sigma}{\delta}.$$
Then the following anisotropic interpolation inequality holds true
\begin{eqnarray}\label{ghpoiu908}
 \|\Lambda_{x_{i}}^{\sigma}f\|_{L^{r}}\leq
C\|f\|_{L^{q}}^{1-\frac{\sigma}{\delta}}
\|\Lambda_{x_{i}}^{\delta}f\|_{L^{p}}^{\frac{\sigma}{\delta}},\quad i=1,2.
\end{eqnarray}
\end{lemma}
\begin{proof}
It suffices to prove (\ref{ghpoiu908}) for $i=1$.
We recall the following 1D Sobolev interpolation inequality
\begin{eqnarray}
\|\Lambda_{x_{1}}^{\sigma}f\|_{L_{x_{1}}^{r}(\mathbb{R})}\leq C\|f\|_{L_{x_{1}}^{q}(\mathbb{R})}^{1-\frac{\sigma}{\delta}} \|\Lambda_{x_{1}}^{\delta}f\|_{L_{x_{1}}^{p}(\mathbb{R})}^{ \frac{\sigma}{\delta}},\nonumber
\end{eqnarray}
where we have used the sub-index $x_{1}$ with the Lebesgue spaces to emphasize that the norms are taken in 1D Lebesgue spaces with respect to $x_{1}$.
This along with the Young inequality implies
\begin{align}
\|\Lambda_{x_{1}}^{\sigma}f\|_{L^{r}}^{r}=&\int_{\mathbb{R}}{
\|\Lambda_{x_{1}}^{\sigma}f(x_{1},\,x_{2})\|_{L_{x_{1}}^{r}}^{r}\,
dx_{2}}\nonumber\\ \leq& C\int_{\mathbb{R}}{ \|f(x_{1},\,x_{2})\|_{L_{x_{1}}^{q}}^{r(1-\frac{\sigma}{\delta})} \|\Lambda_{x_{1}}^{\delta}f(x_{1},\,x_{2})\|_{L_{x_{1}}^{p}}^{ \frac{r\sigma}{\delta}}\,
 dx_{2}}
 \nonumber\\ \leq& C\left(\int_{\mathbb{R}}{ \|f(x_{1},\,x_{2})\|_{L_{x_{1}}^{q}}^{q} \,
 dx_{2}}\right)^{\frac{r(1-\frac{\sigma}{\delta})}{q}}
 \left(\int_{\mathbb{R}}{ \|\Lambda_{x_{1}}^{\delta}f(x_{1},\,x_{2})\|_{L_{x_{1}}^{p}}^{p}\,
 dx_{2}}\right)^{\frac{\frac{r\sigma}{\delta}}{p}}
 \nonumber\\ \leq& C\|f(x_{1},\,x_{2})\|_{L_{x_{1}x_{2}}^{q}}^{r(1-\frac{\sigma}{\delta})}
 \|\Lambda_{x_{1}}^{\delta}f(x_{1},\,x_{2})\|_{L_{x_{1}x_{2}}^{p}}
 ^{\frac{r\sigma}{\delta}}
 \nonumber\\ =&
 C\|f\|_{L^{q}}^{r(1-\frac{\sigma}{\delta})}
\|\Lambda_{x_{1}}^{\delta}f\|_{L^{p}}^{\frac{r\sigma}{\delta}}, \nonumber
\end{align}
which yields that
$$\|\Lambda_{x_{1}}^{\sigma}f\|_{L^{r}}\leq
C\|f\|_{L^{q}}^{1-\frac{\sigma}{\delta}}
\|\Lambda_{x_{1}}^{\delta}f\|_{L^{p}}^{\frac{\sigma}{\delta}}.$$
Thus, this ends the proof of Lemma \ref{dfdasfwqew33}.
\end{proof}

\vskip .1in
With the above lemmas in hand, we are ready to show the global regularity part of Theorem \ref{OKTh1}. We point out that the existence and uniqueness of local smooth solutions can be established via a standard procedure (see for example \cite[Proposition 5.1]{Yenonli}). Consequently, it suffices to establish {\it a priori} estimates that hold for any fixed $T>0$. We begin with the following standard energy estimates.
\begin{Pros}\label{Spr1}
Let $\theta$ be the solution of (\ref{SQG}). Then, for any $t>0$,
$$\|\theta(t)\|_{L^{2}}^{2}+2\int_{0}^{t}{(\|\Lambda_{x_{1}}^{\alpha}\theta(
\tau)\|_{L^{2}}^{2}+\|\Lambda_{x_{2}}^{\beta}\theta(
\tau)\|_{L^{2}}^{2})\,d\tau}\leq \|\theta_{0}\|_{L^{2}}^{2},$$
$$
\|\theta(t)\|_{L^p}\leq \|\theta_0\|_{L^p}, \qquad 1\leq p\leq\infty.
$$
\end{Pros}
\begin{proof}
The proof of Proposition \ref{Spr1} was derived by \cite[Proposition 3.1]{Yenonli}. Here for the convenience of the reader, we restate it here.
Multiplying $(\ref{SQG})$ by $\theta$ and using the divergence-free condition give rise to
 $$\frac{1}{2}\frac{d}{dt}\|\theta(t)\|_{L^{2}}^{2}+
 \|\Lambda_{x_{1}}^{\alpha}\theta\|_{L^{2}}^{2}+
 \|\Lambda_{x_{2}}^{\beta}\theta\|_{L^{2}}^{2}= 0.$$
Integrating with respect to time yields
\begin{eqnarray}
\|\theta(t)\|_{L^{2}}^{2}+2\int_{0}^{t}{(\|\Lambda_{x_{1}}^{\alpha}\theta(
\tau)\|_{L^{2}}^{2}+\|\Lambda_{x_{2}}^{\beta}\theta(
\tau)\|_{L^{2}}^{2})\,d\tau}\leq \|\theta_{0}\|_{L^{2}}^{2}.\nonumber
\end{eqnarray}
We multiply $(\ref{SQG})$ by $|\theta|^{p-2}\theta$ and use the divergence-free condition to obtain
$$\frac{1}{p}\frac{d}{dt}\|\theta(t)\|_{L^{p}}^{p}+
\int_{\mathbb{R}^{2}}{\Lambda_{x_{1}}^{2\alpha}\theta(|\theta|^{p-2}\theta)\,dx}+
\int_{\mathbb{R}^{2}}{\Lambda_{x_{2}}^{2\beta}\theta(|\theta|^{p-2}\theta)\,dx}=0.$$
Thanks to the lower bounds
\begin{align}
\int_{\mathbb{R}^{2}}{\Lambda_{x_{1}}^{2\alpha}\theta(|\theta|^{p-2}\theta)
\,dx} =&\int_{\mathbb{R}}\int_{\mathbb{R}}{\Lambda_{x_{1}}^{2\alpha}
\theta(x_{1},x_{2})
(|\theta(x_{1},x_{2})|^{p-2}\theta(x_{1},x_{2}))
\,dx_{1}dx_{2}}\nonumber\\ \geq&C\int_{\mathbb{R}}\int_{\mathbb{R}}
{\big(\Lambda_{x_{1}}^{\alpha}
|\theta(x_{1},x_{2})|^{\frac{p}{2}}\big)^{2}
\,dx_{1}dx_{2}}\nonumber\\ \geq&0,\nonumber
\end{align}
\begin{align}
\int_{\mathbb{R}^{2}}{\Lambda_{x_{2}}^{2\beta}\theta(|\theta|^{p-2}\theta)
\,dx} =&\int_{\mathbb{R}}\int_{\mathbb{R}}{\Lambda_{x_{2}}^{2\beta}\theta(x_{1},x_{2})
(|\theta(x_{1},x_{2})|^{p-2}\theta(x_{1},x_{2}))
\,dx_{1}dx_{2}}\nonumber\\ \geq&C\int_{\mathbb{R}}\int_{\mathbb{R}}
{\big(\Lambda_{x_{2}}^{\beta}
|\theta(x_{1},x_{2})|^{\frac{p}{2}}\big)^{2}
\,dx_{1}dx_{2}}\nonumber\\ \geq&0,\nonumber
\end{align}
we infer that
$$
\|\theta(t)\|_{L^p}\leq \|\theta_0\|_{L^p}, \qquad 1\le p\le \infty.
$$
This finishes the proof of the proposition.
\end{proof}

 \vskip .1in
We now establish the global $\dot{H}^1$-bound in the case when $\beta>\frac{1}{2\alpha+1}$ with $0<\alpha\leq\frac{1}{2}$.

\begin{Pros}\label{Lpp302}
Let $\theta$ be the solution of (\ref{SQG}). If $\alpha$ and $\beta$ satisfy
$$\beta> \frac{1}{2\alpha+1}, \ \ 0<\alpha\leq\frac{1}{2},$$
then for any $t>0$
\begin{eqnarray}\label{asddfgdfht302}
\|\nabla\theta(t)\|_{L^{2}}^{2}+ \int_{0}^{t}{(\|\Lambda_{x_{1}}^{\alpha}\nabla\theta(
\tau)\|_{L^{2}}^{2}+\|\Lambda_{x_{2}}^{\beta}\nabla\theta(
\tau)\|_{L^{2}}^{2})\,d\tau}\leq
C(t,\,\theta_{0}),
\end{eqnarray}
where $C(t, \,\theta_{0})$ is a constant depending on $t$ and the initial norm  $\|\theta_0\|_{H^{1}}$. Here and in what follows, we should mention that the constant $C(t, \,\theta_{0})$ could grow when $t$ goes to infinity, but it must
not blow-up in finite time. This can be checked throughout the proof.
\end{Pros}

\begin{proof}
Taking the inner product of $(\ref{SQG})$ with $\Delta \theta$ and using the divergence-free condition $\partial_{x_{1}} u_{1}+\partial_{x_{2}} u_{2}=0$, we infer that
\begin{align}\label{t3326t002}
\frac{1}{2}\frac{d}{dt}\|\nabla \theta(t)\|_{L^{2}}^{2}+
 \|\Lambda_{x_{1}}^{\alpha}\nabla\theta\|_{L^{2}}^{2}+
 \|\Lambda_{x_{2}}^{\beta}\nabla\theta\|_{L^{2}}^{2} =&\int_{\mathbb{R}^{2}}{(u \cdot \nabla)\theta\Delta\theta\,dx}\nonumber\\
 =& \mathcal{J}_{1}+\mathcal{J}_{2}+\mathcal{J}_{3}+\mathcal{J}_{4},
\end{align}
where
$$
\mathcal{J}_{1}=-\int_{\mathbb{R}^{2}}{\partial_{x_{1}}u_{1} \partial_{x_{1}}\theta\partial_{x_{1}}\theta\,dx},\quad \mathcal{J}_{2}=-\int_{\mathbb{R}^{2}}{\partial_{x_{1}}u_{2} \partial_{x_{2}}\theta\partial_{x_{1}}\theta\,dx},
$$
$$
\mathcal{J}_{3}=-\int_{\mathbb{R}^{2}}{\partial_{x_{2}}u_{1} \partial_{x_{1}}\theta\partial_{x_{2}}\theta\,dx},\quad \mathcal{J}_{4}=-\int_{\mathbb{R}^{2}}{\partial_{x_{2}}u_{2} \partial_{x_{2}}\theta\partial_{x_{2}}\theta\,dx}.
$$
Using $\partial_{x_{1}} u_{1}+\partial_{x_{2}} u_{2}=0$ and \eqref{yzz3}, we derive
\begin{align}\label{t3326t003}
\mathcal{J}_{1} =&\int_{\mathbb{R}^{2}}{\partial_{x_{2}}u_{2} \partial_{x_{1}}\theta\partial_{x_{1}}\theta\,dx}\nonumber\\
 =&-2\int_{\mathbb{R}^{2}}{u_{2} \partial_{x_{1}}\theta \partial_{x_{2}x_{1}}\theta\,dx}
\nonumber\\
 =&-2\int_{\mathbb{R}^{2}}{\Lambda_{x_{2}}^{1-\delta}(u_{2} \partial_{x_{1}}\theta) \Lambda_{x_{2}}^{\delta}\Lambda_{x_{2}}^{-1}\partial_{x_{2}}\partial_{x_{1}}\theta\,dx}
\nonumber\\
 \leq&C\|\Lambda_{x_{2}}^{\delta}\partial_{x_{1}}\theta\|_{L^{\frac{2q}{q-2}}} \|\Lambda_{x_{2}}^{1-\delta}(u_{2} \partial_{x_{1}}\theta)\|_{L^{\frac{2q}{q+2}}}
\nonumber\\
 \leq&C\|\Lambda_{x_{2}}^{\delta}\partial_{x_{1}}\theta\|_{L^{\frac{2q}{q-2}}} (\|\Lambda_{x_{2}}^{1-\delta}(u_{2} \partial_{x_{1}}\theta)-\Lambda_{x_{2}}^{1-\delta}u_{2} \partial_{x_{1}}\theta\|_{L^{\frac{2q}{q+2}}} +\|\Lambda_{x_{2}}^{1-\delta}u_{2} \partial_{x_{1}}\theta\|_{L^{\frac{2q}{q+2}}} )
\nonumber\\
 \leq&C\|\Lambda^{\frac{2}{q}}\Lambda_{x_{2}}^{\delta}\partial_{x_{1}}\theta\|_{L^{2}} ( \|u_{2}\|_{L^{q}} \|\Lambda_{x_{2}}^{1-\delta}\partial_{x_{1}}\theta\|_{L^{2}}+ \|\Lambda_{x_{2}}^{1-\delta}u_{2}\|_{L^{\frac{2(1+\alpha)q}{\alpha q+2}}} \| \partial_{x_{1}}\theta\|_{L^{\frac{2(1+\alpha)q}{ q+2\alpha}}}),
\end{align}
where $q\in (2,\infty)$ will be fixed hereafter. Here we have appealed to the Sobolev embedding $\dot{H}^{\frac{2}{q}}(\mathbb{R}^{2})\hookrightarrow L^{\frac{2q}{q-2}}(\mathbb{R}^{2})$.
If we choose
\begin{align}\label{udfr1}
\delta=\frac{1-\alpha\beta}{\alpha +1},
\end{align}
then we conclude by \eqref{ghpoiu908} that
\begin{align}\label{udfr2}
\|\Lambda_{x_{2}}^{1-\delta}u_{2}\|_{L^{\frac{2(1+\alpha)q}{\alpha q+2}}}&\leq C\|u_{2}\|_{L^{q}}^{\frac{1}{\alpha+1}}\|\Lambda_{x_{2}}^{\beta+1} u_{2}\|_{L^{2}}^{\frac{\alpha}{\alpha+1}}\nonumber\\&\leq C\|\theta\|_{L^{q}}^{\frac{1}{\alpha+1}}\|\Lambda_{x_{2}}^{\beta+1} \theta\|_{L^{2}}^{\frac{\alpha}{\alpha+1}}\nonumber\\&\leq C\|\theta\|_{L^{q}}^{\frac{1}{\alpha+1}}\|\Lambda_{x_{2}}^{\beta} \nabla\theta\|_{L^{2}}^{\frac{\alpha}{\alpha+1}},
\end{align}
\begin{align}\label{udfr3}
\| \partial_{x_{1}}\theta\|_{L^{\frac{2(1+\alpha)q}{ q+2\alpha}}}&\leq C\|\theta\|_{L^{q}}^{\frac{\alpha}{\alpha+1}}\|\Lambda_{x_{1}}^{\alpha+1} \theta\|_{L^{2}}^{\frac{1}{\alpha+1}}\nonumber\\&\leq
C\|\theta\|_{L^{q}}^{\frac{\alpha}{\alpha+1}}\|\Lambda_{x_{1}}^{\alpha} \nabla \theta\|_{L^{2}}^{\frac{1}{\alpha+1}}.
\end{align}
If we further let $\delta$ of \eqref{udfr1} satisfy $\delta\in [1-\beta,\, \beta]$, then \eqref{t2326001} allows us to show
\begin{align}\label{udfr4}
\|\Lambda_{x_{2}}^{1-\delta}\partial_{x_{1}}\theta\|_{L^{2}}\leq C\|\Lambda_{x_{2}}^{\beta+1} \theta\|_{L^{2}}^{\frac{\alpha}{\alpha+1}}\|\Lambda_{x_{1}}^{\alpha+1} \theta\|_{L^{2}}^{\frac{1}{\alpha+1}}\leq C\|\Lambda_{x_{2}}^{\beta} \nabla\theta\|_{L^{2}}^{\frac{\alpha}{\alpha+1}}\|\Lambda_{x_{1}}^{\alpha} \nabla\theta\|_{L^{2}}^{\frac{1}{\alpha+1}}.
\end{align}
Putting \eqref{udfr2}, \eqref{udfr3} and \eqref{udfr4} into \eqref{t3326t003} yields
\begin{align}\label{udfr5}
\mathcal{J}_{1}  &\leq C\|\Lambda^{\frac{2}{q}}\Lambda_{x_{2}}^{\delta}\partial_{x_{1}}\theta\|_{L^{2}} \|\theta\|_{L^{q}} \|\Lambda_{x_{2}}^{\beta} \nabla\theta\|_{L^{2}}^{\frac{\alpha}{\alpha+1}}\|\Lambda_{x_{1}}^{\alpha} \nabla\theta\|_{L^{2}}^{\frac{1}{\alpha+1}}.
\end{align}
It is easy to see that
\begin{align}\label{ubnnmii91} \|\Lambda^{\frac{2}{q}}\Lambda_{x_{2}}^{\delta}\partial_{x_{1}}\theta\|_{L^{2}} \leq C \|\Lambda_{x_{1}}^{1+\frac{2}{q}}\Lambda_{x_{2}}^{\delta} \theta\|_{L^{2}} +C\|\Lambda_{x_{1}}\Lambda_{x_{2}}^{\delta+\frac{2}{q}} \theta\|_{L^{2}}.
 \end{align}
By \eqref{t2326001}, the first term at the right hand side of \eqref{ubnnmii91} can be bounded by
\begin{align}\label{udfr7}
\|\Lambda_{x_{1}}^{1+\frac{2}{q}}\Lambda_{x_{2}}^{\delta} \theta\|_{L^{2}}&\leq C
 \|\theta\|_{L^{2}}^{1-\frac{\delta}{\beta}-\frac{1+\frac{2}{q}
 -\frac{\delta}{\beta}}{\alpha+1}} \|\Lambda_{x_{2}}^{\beta}\Lambda_{x_{1}}\theta\|_{L^{2}}^{\frac{\delta}{\beta}}
 \|\Lambda_{x_{1}}^{\alpha+1} \theta\|_{L^{2}}^{\frac{1+\frac{2}{q}
 -\frac{\delta}{\beta}}{\alpha+1}}\nonumber\\
 &\leq C
 \|\theta\|_{L^{2}}^{\frac{\epsilon\delta}{\beta}} \|\Lambda_{x_{2}}^{\beta}\nabla\theta\|_{L^{2}}^{\frac{\delta}{\beta}}
 \|\Lambda_{x_{1}}^{\alpha}\nabla\theta\|_{L^{2}}
 ^{\frac{\beta-(1+\epsilon)\delta}{\beta}}.
\end{align}
To bound the last term at the right hand side of \eqref{ubnnmii91} effectively, we fix $q\in (2,\infty)$ as
\begin{align}\label{udfr6}
q=\frac{2\beta}{\alpha\beta+\delta-(\alpha+1)(1+\epsilon)\delta},
\end{align}
where $\delta<\frac{\alpha\beta}{\alpha+(\alpha+1)\epsilon}$ with small $\epsilon>0$ to be fixed after.
As the $q$ of \eqref{udfr6}, it follows from \eqref{t2326001} that
\begin{align}\label{udfr8}
\|\Lambda_{x_{1}}\Lambda_{x_{2}}^{\delta+\frac{2}{q}} \theta\|_{L^{2}}&\leq C
 \|\theta\|_{L^{2}}^{1-\frac{\delta}{\beta}-\frac{1+\frac{2}{q}
 -\frac{\delta}{\beta}}{\alpha+1}} \|\Lambda_{x_{2}}^{\beta+\frac{2(1+\epsilon)}{q}}
 \Lambda_{x_{1}}^{1-\frac{2(1+\epsilon)}{q}}\theta\|_{L^{2}}^{\frac{\delta}{\beta}}
 \|\Lambda_{x_{1}}^{\alpha+1-\frac{2}{q}}\Lambda_{x_{2}}^{\frac{2}{q}}\theta\|_{L^{2}}^{\frac{1+\frac{2}{q}
 -\frac{\delta}{\beta}}{\alpha+1}}\nonumber\\
 &\leq C
 \|\theta\|_{L^{2}}^{\frac{\epsilon\delta}{\beta}} \|\Lambda_{x_{2}}^{\beta}\nabla\theta\|_{L^{2}}^{\frac{\delta}{\beta}}
 \|\Lambda_{x_{1}}^{\alpha}\nabla\theta\|_{L^{2}}
 ^{\frac{\beta-(1+\epsilon)\delta}{\beta}}.
\end{align}
We remark that the following facts have been applied in establishing \eqref{udfr8} \begin{align}\label{dffgtyu889}
\|\Lambda_{x_{i}}^{k+l}\Lambda_{x_{j}}^{1-l}\theta\|_{L^{2}}\leq C \|\Lambda_{x_{i}}^{k}\nabla\theta\|_{L^{2}},\qquad \|\Lambda_{x_{i}}^{k-l}\Lambda_{x_{j}}^{l}\theta\|_{L^{2}}\leq C \|\Lambda_{x_{i}}^{k-1}\nabla\theta\|_{L^{2}},
\end{align}
where $0\leq l\leq 1$ and $i=1,j=2$ or $i=2,j=1$. Actually, due to Plancherel's theorem, we have
$$\|\Lambda_{x_{i}}^{k+l}\Lambda_{x_{j}}^{1-l}\theta\|_{L^{2}}=
\||\xi_{i}|^{k+l}|\xi_{j}|^{1-l}\widehat{\theta}\|_{L^{2}}\leq C\||\xi_{i}|^{k}|\xi|\widehat{\theta}\|_{L^{2}}=C \|\Lambda_{x_{i}}^{k}\nabla\theta\|_{L^{2}},$$
$$\|\Lambda_{x_{i}}^{k-l}\Lambda_{x_{j}}^{l}\theta\|_{L^{2}}
=\||\xi_{i}|^{k-l}|\xi_{j}|^{l}\widehat{\theta}\|_{L^{2}}\leq C
\||\xi_{i}|^{k-1}|\xi|\widehat{\theta}\|_{L^{2}}=C \|\Lambda_{x_{i}}^{k-1}\nabla\theta\|_{L^{2}},$$
where we have used the simple facts
$$|\xi_{i}|^{k+l}|\xi_{j}|^{1-l}\leq C|\xi_{i}|^{k}|\xi|,\qquad |\xi_{i}|^{k-l}|\xi_{j}|^{l}\leq C|\xi_{i}|^{k-1}|\xi|.$$
Inserting \eqref{udfr7} and \eqref{udfr8} into \eqref{udfr5} shows
\begin{align}\label{udfr9}
\mathcal{J}_{1}  &\leq C\|\theta\|_{L^{2}}^{\frac{\epsilon\delta}{\beta}} \|\Lambda_{x_{2}}^{\beta}\nabla\theta\|_{L^{2}}^{\frac{\delta}{\beta}}
 \|\Lambda_{x_{1}}^{\alpha}\nabla\theta\|_{L^{2}}
 ^{\frac{\beta-(1+\epsilon)\delta}{\beta}} \|\theta\|_{L^{q}} \|\Lambda_{x_{2}}^{\beta} \nabla\theta\|_{L^{2}}^{\frac{\alpha}{\alpha+1}}\|\Lambda_{x_{1}}^{\alpha} \nabla\theta\|_{L^{2}}^{\frac{1}{\alpha+1}}\nonumber\\&\leq\eta\|\Lambda_{x_{1}}^{\alpha}\nabla\theta\|_{L^{2}}^{2}+
 \eta\|\Lambda_{x_{2}}^{\beta}\nabla\theta\|_{L^{2}}^{2}
 +C_{\eta}\|\theta\|_{L^{q}}^{\frac{2\beta}{\epsilon\delta}}\|\theta\|_{L^{2}}^{2}
 \nonumber\\&=\eta\|\Lambda_{x_{1}}^{\alpha}\nabla\theta\|_{L^{2}}^{2}+
 \eta\|\Lambda_{x_{2}}^{\beta}\nabla\theta\|_{L^{2}}^{2}
 +C_{\eta}\|\theta\|_{L^{q}}^{\frac{2(\alpha+1)\beta}
 {\epsilon(1-\alpha\beta)}}\|\theta\|_{L^{2}}^{2}.
\end{align}
Keeping in mind $\delta=\frac{1-\alpha\beta}{\alpha +1}\in [1-\beta,\beta]$ and $\delta<\frac{\alpha\beta}{\alpha+(\alpha+1)\epsilon}$, we can check that
$$0<\epsilon<\frac{\alpha[(2\alpha +1)\beta-1]}{(1-\alpha\beta)(\alpha +1)},$$
which requires the condition
$$\beta>\frac{1}{2\alpha+1}.$$
By the same argument adopted in handing $\mathcal{J}_{1}$, one gets
\begin{align}\label{t3326t005}
\mathcal{J}_{2} =&\int_{\mathbb{R}^{2}}{ \theta\partial_{x_{2}x_{1}}u_{2}\partial_{x_{1}}\theta\,dx}
+\int_{\mathbb{R}^{2}}{ \theta\partial_{x_{1}}u_{2}\partial_{x_{2}x_{1}}\theta\,dx}\nonumber\\
 \leq&C\|\Lambda_{x_{2}}^{\delta}\partial_{x_{1}}u_{2}\|_{L^{\frac{2q}{q-2}}} \|\Lambda_{x_{2}}^{1-\delta}(\theta \partial_{x_{1}}\theta)\|_{L^{\frac{2q}{q+2}}}+
C\|\Lambda_{x_{2}}^{\delta}\partial_{x_{1}}\theta\|_{L^{\frac{2q}{q-2}}} \|\Lambda_{x_{2}}^{1-\delta}(\theta \partial_{x_{1}}u_{2})\|_{L^{\frac{2q}{q+2}}}
\nonumber\\ \leq&C\|\Lambda_{x_{2}}^{\delta}\partial_{x_{1}}\mathcal {R}_{1}\theta\|_{L^{\frac{2q}{q-2}}}( \|\theta\|_{L^{q}} \|\Lambda_{x_{2}}^{1-\delta}\partial_{x_{1}}\theta\|_{L^{2}}+ \|\Lambda_{x_{2}}^{1-\delta}\theta \partial_{x_{1}}\theta\|_{L^{\frac{2q}{q+2}}})\nonumber\\&
+C\|\Lambda_{x_{2}}^{\delta}\partial_{x_{1}} \theta\|_{L^{2}}( \|\theta\|_{L^{q}} \|\Lambda_{x_{2}}^{1-\delta}\partial_{x_{1}}u_{2}\|_{L^{2}}+ \|\Lambda_{x_{2}}^{1-\delta}\theta \partial_{x_{1}}u_{2}\|_{L^{\frac{2q}{q+2}}})\nonumber\\ \leq&\eta\|\Lambda_{x_{1}}^{\alpha}\nabla\theta\|_{L^{2}}^{2}+
 \eta\|\Lambda_{x_{2}}^{\beta}\nabla\theta\|_{L^{2}}^{2}
 +C_{\eta}\|\theta\|_{L^{q}}^{\frac{2(\alpha+1)\beta}
 {\epsilon(1-\alpha\beta)}}\|\theta\|_{L^{2}}^{2},
\end{align}
where we have used the fact ${u}=(-\mathcal{R}_{2}\theta,\,\,\mathcal {R}_{1}\theta)$ several times.
For the third term $\mathcal{J}_{3}$, it follows from integration by parts that
\begin{align}
\mathcal{J}_{3} =&\int_{\mathbb{R}^{2}}{ u_{1} \partial_{x_{2}x_{1}}\theta\partial_{x_{2}}\theta\,dx}
+\int_{\mathbb{R}^{2}}{ u_{1} \partial_{x_{1}}\theta\partial_{x_{2}x_{2}}\theta\,dx}\nonumber\\
=&\frac{1}{2}\int_{\mathbb{R}^{2}}{ u_{1} \partial_{x_{1}}(\partial_{x_{2}}\theta\partial_{x_{2}}\theta)\,dx}
+\int_{\mathbb{R}^{2}}{ u_{1} \partial_{x_{1}}\theta\partial_{x_{2}x_{2}}\theta\,dx}
\nonumber\\
=&-\frac{1}{2}\int_{\mathbb{R}^{2}}{ \partial_{x_{1}}u_{1} \partial_{x_{2}}\theta\partial_{x_{2}}\theta\,dx}
+\int_{\mathbb{R}^{2}}{ u_{1} \partial_{x_{1}}\theta\partial_{x_{2}x_{2}}\theta\,dx}
\nonumber\\
=&\frac{1}{2}\int_{\mathbb{R}^{2}}{ \partial_{x_{2}}u_{2} \partial_{x_{2}}\theta\partial_{x_{2}}\theta\,dx}
+\int_{\mathbb{R}^{2}}{ u_{1} \partial_{x_{1}}\theta\partial_{x_{2}x_{2}}\theta\,dx}
\nonumber\\
=&-\frac{1}{2}\mathcal{J}_{4}
+\int_{\mathbb{R}^{2}}{ u_{1} \partial_{x_{1}}\theta\partial_{x_{2}x_{2}}\theta\,dx}
\nonumber\\
= :&-\frac{1}{2}\mathcal{J}_{4}+\widetilde{\mathcal{J}_{3}}.\nonumber
\end{align}
Just as for $\mathcal{J}_{1}$, we may show
\begin{align}\label{hhh3e456}
\widetilde{\mathcal{J}_{3}}
 \leq&C\|\Lambda_{x_{2}}^{\delta}\partial_{x_{2}}\theta
 \|_{L^{\frac{2\widetilde{q}}{\widetilde{q}-2}}} \|\Lambda_{x_{2}}^{1-\delta}(u_{1} \partial_{x_{1}}\theta)\|_{L^{\frac{2\widetilde{q}}{\widetilde{q}+2}}}
\nonumber\\
 \leq&C\|\Lambda_{x_{2}}^{\delta}\partial_{x_{2}}\theta
 \|_{L^{\frac{2\widetilde{q}}{\widetilde{q}-2}}} (\|\Lambda_{x_{2}}^{1-\delta}(u_{1} \partial_{x_{1}}\theta)-\Lambda_{x_{2}}^{1-\delta}u_{1} \partial_{x_{1}}\theta\|_{L^{\frac{2\widetilde{q}}{\widetilde{q}+2}}} +\|\Lambda_{x_{2}}^{1-\delta}u_{1} \partial_{x_{1}}\theta\|_{L^{\frac{2\widetilde{q}}{\widetilde{q}+2}}} )
\nonumber\\
 \leq&C\|\Lambda_{x_{2}}^{\delta}\partial_{x_{2}}\theta
 \|_{L^{\frac{2\widetilde{q}}{\widetilde{q}-2}}}  ( \|u_{1}\|_{L^{\widetilde{q}}} \|\Lambda_{x_{2}}^{1-\delta}\partial_{x_{1}}\theta\|_{L^{2}}+ \|\Lambda_{x_{2}}^{1-\delta}u_{1}\|_{L^{\frac{2(1+\alpha)\widetilde{q}}{\alpha \widetilde{q}+2}}} \| \partial_{x_{1}}\theta\|_{L^{\frac{2(1+\alpha)\widetilde{q}}{ \widetilde{q}+2\alpha}}})
 \nonumber\\
 \leq&C\|\Lambda^{\frac{2}{\widetilde{q}}}\Lambda_{x_{2}}^{\delta}\partial_{x_{2}}
\theta\|_{L^{2}}
 \|\theta\|_{L^{\widetilde{q}}} \|\Lambda_{x_{2}}^{\beta} \nabla\theta\|_{L^{2}}^{\frac{\alpha}{\alpha+1}}\|\Lambda_{x_{1}}^{\alpha} \nabla\theta\|_{L^{2}}^{\frac{1}{\alpha+1}},
\end{align}
where $\widetilde{q}\in (2,\infty)$ will be fixed hereafter. We furthermore obtain from \eqref{t2326001} and \eqref{avbnuy8911} that
\begin{align}\label{udfhy01y} \|\Lambda^{\frac{2}{\widetilde{q}}}\Lambda_{x_{2}}^{\delta}\partial_{x_{2}}
\theta\|_{L^{2}} &\leq C \|\Lambda_{x_{1}}^{\frac{2}{\widetilde{q}}}\Lambda_{x_{2}}^{\delta+1} \theta\|_{L^{2}} +C\|\Lambda_{x_{2}}^{1+\delta+\frac{2}{\widetilde{q}}} \theta\|_{L^{2}}
\nonumber\\&\leq
C\|\theta\|_{L^{2}}^{1-\frac{\delta+1}{\beta+1}-\frac{\frac{2}{\widetilde{q}}}
{\alpha+1}} \|\Lambda_{x_{2}}^{\beta+1}\theta\|_{L^{2}}^{\frac{\delta+1}{\beta+1}}
 \|\Lambda_{x_{1}}^{\alpha+1}\theta\|_{L^{2}}
 ^{\frac{\frac{2}{\widetilde{q}}}{\alpha+1}}\nonumber\\&\quad+
 C\|\theta\|_{L^{2}}^{1-\frac{\delta+1+\frac{2}{\widetilde{q}}}{\beta+1}} \|\Lambda_{x_{2}}^{\beta+1}\theta\|_{L^{2}}^{
 \frac{\delta+1+\frac{2}{\widetilde{q}}}{\beta+1}}\nonumber\\&\leq
C\|\theta\|_{L^{2}}^{1-\frac{\delta+1}{\beta+1}-\frac{\frac{2}{\widetilde{q}}}
{\alpha+1}} \|\Lambda_{x_{2}}^{\beta}\nabla\theta\|_{L^{2}}^{\frac{\delta+1}{\beta+1}}
 \|\Lambda_{x_{1}}^{\alpha}\nabla\theta\|_{L^{2}}
 ^{\frac{\frac{2}{\widetilde{q}}}{\alpha+1}}\nonumber\\&\quad+
 C\|\theta\|_{L^{2}}^{1-\frac{\delta+1+\frac{2}{\widetilde{q}}}{\beta+1}} \|\Lambda_{x_{2}}^{\beta}\nabla\theta\|_{L^{2}}^{
 \frac{\delta+1+\frac{2}{\widetilde{q}}}{\beta+1}},
\end{align}
where $\delta<\beta-\frac{\beta+1}{\alpha+1}\frac{2}{\widetilde{q}}$. This along with \eqref{udfr1} gives
$$\widetilde{q}>\frac{2(\beta+1)}{(2\alpha+1)\beta-1}.$$
Putting \eqref{udfhy01y} into \eqref{hhh3e456} and using the H$\rm\ddot{o}$lder inequality, we arrive at
\begin{align}\label{hcvby71}
\widetilde{\mathcal{J}_{3}}
 &\leq \eta\|\Lambda_{x_{2}}^{\beta} \nabla\theta\|_{L^{2}}^{2}+\eta\|\Lambda_{x_{1}}^{\alpha} \nabla\theta\|_{L^{2}}^{2}+C_{\eta}
 \|\theta\|_{L^{\widetilde{q}}}^{\frac{2}{1-\frac{\delta+1}
 {\beta+1}-\frac{\frac{2}{\widetilde{q}}}
{\alpha+1}}}\|\theta\|_{L^{2}}^{2}+C_{\eta}
 \|\theta\|_{L^{\widetilde{q}}}^{\frac{2}
 {1-\frac{\delta+1+\frac{2}{\widetilde{q}}}{\beta+1}}}\|\theta\|_{L^{2}}^{2}
 \nonumber\\&= \eta\|\Lambda_{x_{2}}^{\beta} \nabla\theta\|_{L^{2}}^{2}+\eta\|\Lambda_{x_{1}}^{\alpha} \nabla\theta\|_{L^{2}}^{2}+C_{\eta}\|\theta\|_{L^{\widetilde{q}}}
 ^{\frac{2(\alpha+1)(\beta+1)\widetilde{q}}{[(2\alpha+1)\beta-1]
 \widetilde{q}-2(\beta+1)}}\|\theta\|_{L^{2}}^{2}
  \nonumber\\&\quad+C_{\eta}\|\theta\|_{L^{\widetilde{q}}}
 ^{\frac{2(\alpha+1)(\beta+1)\widetilde{q}}{[(2\alpha+1)\beta-1]
 \widetilde{q}-2(\alpha+1)}}\|\theta\|_{L^{2}}^{2}.
\end{align}
Finally, the term $\mathcal{J}_{4}$ can be bounded by
\begin{align}\label{t3nhy71qw}
\mathcal{J}_{4} =&2\int_{\mathbb{R}^{2}}{ u_{2} \partial_{x_{2}x_{2}}\theta\partial_{x_{2}}\theta\,dx}
\nonumber\\
 \leq&C\|\Lambda_{x_{2}}^{1-\beta}\partial_{x_{2}}\theta
 \|_{L^{\frac{2\widehat{q}}{\widehat{q}-2}}} \|\Lambda_{x_{2}}^{\beta}(u_{2} \partial_{x_{2}}\theta)\|_{L^{\frac{2\widehat{q}}{\widehat{q}+2}}} \nonumber\\
 \leq&C\|\Lambda_{x_{2}}^{1-\beta}\partial_{x_{2}}\theta
 \|_{L^{\frac{2\widehat{q}}{\widehat{q}-2}}} (\|\Lambda_{x_{2}}^{\beta}(u_{2} \partial_{x_{2}}\theta)-\Lambda_{x_{2}}^{\beta}u_{2} \partial_{x_{2}}\theta\|_{L^{\frac{2\widehat{q}}{\widehat{q}+2}}} +\|\Lambda_{x_{2}}^{\beta}u_{2} \partial_{x_{2}}\theta\|_{L^{\frac{2\widehat{q}}{\widehat{q}+2}}} )
\nonumber\\
 \leq&C\|\Lambda_{x_{2}}^{1-\beta}\partial_{x_{2}}\theta
 \|_{L^{\frac{2\widehat{q}}{\widehat{q}-2}}}  ( \|u_{2}\|_{L^{\widehat{q}}} \|\Lambda_{x_{2}}^{\beta}\partial_{x_{2}}\theta\|_{L^{2}}+ \|\Lambda_{x_{2}}^{\beta}u_{2}\|_{L^{\frac{2(1+\beta)\widehat{q}}{\beta \widehat{q}+2}}} \| \partial_{x_{2}}\theta\|_{L^{\frac{2(1+\beta)\widehat{q}}{ \widehat{q}+2\beta}}}) \nonumber\\
 \leq&C\|\Lambda_{x_{2}}^{1-\beta}\partial_{x_{2}}\theta
 \|_{L^{\frac{2\widehat{q}}{\widehat{q}-2}}}  ( \|u_{2}\|_{L^{\widehat{q}}} \|\Lambda_{x_{2}}^{\beta}\partial_{x_{2}}\theta\|_{L^{2}}+\|u_{2}\|_{L^{ \widehat{q}}}^{\frac{1}{\beta+1}} \|\Lambda_{x_{2}}^{\beta+1}u_{2}\|_{L^{2}}^{\frac{\beta}{\beta+1}} \|\theta\|_{L^{ \widehat{q}}}^{\frac{\beta}{\beta+1}} \|\Lambda_{x_{2}}^{\beta+1}\theta\|_{L^{2}}^{\frac{1}{\beta+1}})
 \nonumber\\
 \leq&C\|\Lambda^{\frac{2}{\widehat{q}}}\Lambda_{x_{2}}^{1-\beta}\partial_{x_{2}}\theta
 \|_{L^{2}}   \|\theta\|_{L^{\widehat{q}}} \|\Lambda_{x_{2}}^{\beta}\nabla\theta\|_{L^{2}},
\end{align}
where $\widehat{q}\in (2,\infty)$ will be fixed after. Keeping in mind \eqref{udfhy01y}, one gets
\begin{align}\label{uwvty7y1}
\|\Lambda^{\frac{2}{\widehat{q}}}\Lambda_{x_{2}}^{1-\beta}\partial_{x_{2}}\theta
 \|_{L^{2}} &\leq
C\|\theta\|_{L^{2}}^{1-\frac{2-\beta}{\beta+1}-\frac{\frac{2}{\widehat{q}}}
{\alpha+1}} \|\Lambda_{x_{2}}^{\beta}\nabla\theta\|_{L^{2}}^{\frac{2-\beta}{\beta+1}}
 \|\Lambda_{x_{1}}^{\alpha}\nabla\theta\|_{L^{2}}
 ^{\frac{\frac{2}{\widehat{q}}}{\alpha+1}}\nonumber\\&\quad+
 C\|\theta\|_{L^{2}}^{1-\frac{2-\beta+\frac{2}{\widehat{q}}}{\beta+1}} \|\Lambda_{x_{2}}^{\beta}\nabla\theta\|_{L^{2}}^{
 \frac{2-\beta+\frac{2}{\widehat{q}}}{\beta+1}},
\end{align}
where we need the restriction $1-\beta<\beta-\frac{\beta+1}{\alpha+1}\frac{2}{\widehat{q}}$. It thus gives rise to
$$\widehat{q}>\frac{2(\beta+1)}{(\alpha+1)(2\beta-1)}.$$
Inserting \eqref{uwvty7y1} into \eqref{t3nhy71qw} and using the H$\rm\ddot{o}$lder inequality, we observe that
\begin{align}\label{hccvgty21er}
{\mathcal{J}_{4}}
 &\leq \eta\|\Lambda_{x_{2}}^{\beta} \nabla\theta\|_{L^{2}}^{2}+\eta\|\Lambda_{x_{1}}^{\alpha} \nabla\theta\|_{L^{2}}^{2}+ C_{\eta}\|\theta\|_{L^{\widehat{q}}}^{\frac{2(\alpha+1)(\beta+1)\widehat{q}}
 {(\alpha+1)(2\beta-1)\widehat{q}-2(\beta+1)}}\|\theta\|_{L^{2}}^{2}\nonumber\\&\quad
 + C_{\eta}\|\theta\|_{L^{\widehat{q}}}^{\frac{2(\beta+1)\widehat{q}}
 {(2\beta-1)\widehat{q}-2}}\|\theta\|_{L^{2}}^{2}.
\end{align}
Putting (\ref{udfr9}), (\ref{t3326t005}), (\ref{hcvby71}) and (\ref{hccvgty21er}) into \eqref{t3326t002}, we get by selecting small $\eta>0$ that
\begin{eqnarray}\label{t3326t009}
 \frac{d}{dt}\|\nabla \theta(t)\|_{L^{2}}^{2}+
 \|\Lambda_{x_{1}}^{\alpha}\nabla\theta\|_{L^{2}}^{2}+
 \|\Lambda_{x_{2}}^{\beta}\nabla\theta\|_{L^{2}}^{2} \leq F(t)\| \theta\|_{L^{2}}^{2},
\end{eqnarray}
where
\begin{align}F(t)=&
C\Big(\|\theta\|_{L^{q}}^{\frac{2(\alpha+1)\beta}
 {\epsilon(1-\alpha\beta)}}+\|\theta\|_{L^{\widetilde{q}}}
 ^{\frac{2(\alpha+1)(\beta+1)\widetilde{q}}{[(2\alpha+1)\beta-1]
 \widetilde{q}-2(\beta+1)}}+\|\theta\|_{L^{\widetilde{q}}}
 ^{\frac{2(\alpha+1)(\beta+1)\widetilde{q}}{[(2\alpha+1)\beta-1]
 \widetilde{q}-2(\alpha+1)}}+\|\theta\|_{L^{\widehat{q}}}^{\frac{2(\alpha+1)(\beta+1)\widehat{q}}
 {(\alpha+1)(2\beta-1)\widehat{q}-2(\beta+1)}}\nonumber\\&+\|\theta\|_{L^{\widehat{q}}}^{\frac{2(\beta+1)\widehat{q}}
 {(2\beta-1)\widehat{q}-2}}\nonumber
 \Big)\end{align}
with $q=\frac{2(\alpha+1)\beta}{\alpha(\alpha+1)\beta-\epsilon(1-\alpha\beta)}$ for any $0<\epsilon<\frac{\alpha[(2\alpha +1)\beta-1]}{(1-\alpha\beta)(\alpha +1)}$, $\frac{2(\beta+1)}{(2\alpha+1)\beta-1}<\widetilde{q}<\infty$ and $\frac{2(\beta+1)}{(\alpha+1)(2\beta-1)}<\widehat{q}<\infty$.
Noticing the estimates of Proposition \ref{Spr1} and integrating \eqref{t3326t009} in time yield the desired estimate \eqref{asddfgdfht302}.
 Consequently, we complete the proof of Proposition \ref{Lpp302}.
\end{proof}

\vskip .1in
Next we will prove the global
$\dot{H}^1$-bound in the case when $\beta>\frac{1-\alpha}{2\alpha}$ and $\frac{1}{2}<\alpha<1$.
\begin{Pros}\label{Lpp30we3}
Let $\theta$ be the solution of (\ref{SQG}). If $\alpha$ and $\beta$ satisfy
$$\beta>\frac{1-\alpha}{2\alpha}, \ \ \frac{1}{2}<\alpha<1,$$
then for any $t>0$
\begin{eqnarray}\label{cvgyt25y027}
\|\nabla\theta(t)\|_{L^{2}}^{2}+ \int_{0}^{t}{(\|\Lambda_{x_{1}}^{\alpha}\nabla\theta(
\tau)\|_{L^{2}}^{2}+\|\Lambda_{x_{2}}^{\beta}\nabla\theta(
\tau)\|_{L^{2}}^{2})\,d\tau}\leq
C(t,\,\theta_{0}),
\end{eqnarray}
where $C(t, \,\theta_{0})$ is a constant depending on $t$ and the initial norm $\|\theta_0\|_{H^{1}}$.
\end{Pros}

\begin{proof}
We recall from (\ref{t3326t002}) that
\begin{eqnarray}\label{t36t0ghd}
\frac{1}{2}\frac{d}{dt}\|\nabla \theta(t)\|_{L^{2}}^{2}+
 \|\Lambda_{x_{1}}^{\alpha}\nabla\theta\|_{L^{2}}^{2}+
 \|\Lambda_{x_{2}}^{\beta}\nabla\theta\|_{L^{2}}^{2}= \mathcal{J}_{1}+\mathcal{J}_{2}+\mathcal{J}_{3}+\mathcal{J}_{4}.
\end{eqnarray}
Noting $\alpha>\frac{1}{2}$, we derive by \eqref{yzz3} and \eqref{ghpoiu908} that
\begin{align} \label{ghtyde229}
\mathcal{J}_{1} =&-\int_{\mathbb{R}^{2}}{\partial_{x_{1}}u_{1} \partial_{x_{1}}\theta\partial_{x_{1}}\theta\,dx}\nonumber\\
 =&2\int_{\mathbb{R}^{2}}{u_{1} \partial_{x_{1}}\theta \partial_{x_{1}x_{1}}\theta\,dx}
\nonumber\\
 \leq&C\|\Lambda_{x_{1}}^{1-\alpha}\partial_{x_{1}}\theta\|_{L^{\frac{2r}{r-2}}} \|\Lambda_{x_{1}}^{\alpha}(u_{1} \partial_{x_{1}}\theta)\|_{L^{\frac{2r}{r+2}}}
 \nonumber\\
 \leq&C\|\Lambda_{x_{1}}^{1-\alpha}\partial_{x_{1}}\theta\|_{L^{\frac{2r}{r-2}}} (\|\Lambda_{x_{1}}^{\alpha}(u_{1}  \partial_{x_{1}}\theta)-\Lambda_{x_{1}}^{\alpha}u_{1} \partial_{x_{1}}\theta\|_{L^{\frac{2r}{r+2}}}+\|\Lambda_{x_{1}}^{\alpha}u_{1} \partial_{x_{1}}\theta\|_{L^{\frac{2r}{r+2}}})
\nonumber\\
 \leq&C\|\Lambda_{x_{1}}^{1-\alpha}\partial_{x_{1}}\theta\|_{L^{\frac{2r}{r-2}}}( \|u_{1}\|_{L^{r}} \|\Lambda_{x_{1}}^{\alpha}\partial_{x_{1}}\theta\|_{L^{2}}+ \|\Lambda_{x_{1}}^{\alpha}u_{1}\|_{L^{\frac{2(\alpha+1)r}{\alpha r+2}}} \| \partial_{x_{1}}\theta\|_{L^{\frac{2(\alpha+1)r}{2\alpha+r}}})
\nonumber\\
 \leq&C\|\Lambda_{x_{1}}^{1-\alpha}\partial_{x_{1}}\theta\|_{L^{\frac{2r}{r-2}}}( \|u_{1}\|_{L^{r}} \|\Lambda_{x_{1}}^{\alpha}\partial_{x_{1}}\theta\|_{L^{2}}+ \|u_{1}\|_{L^{r}}^{\frac{1}{\alpha+1}}\|\Lambda_{x_{1}}^{\alpha+1}u_{1}\|_{L^{r}}
 ^{\frac{\alpha}{\alpha+1}}
 \|\theta\|_{L^{r}}^{\frac{\alpha}{\alpha+1}}\|\Lambda_{x_{1}}^{\alpha+1}\theta
 \|_{L^{r}}^{\frac{1}{\alpha+1}})
\nonumber\\
 \leq&C\|\Lambda_{x_{1}}^{1-\alpha}\partial_{x_{1}}\theta\|_{L^{\frac{2r}{r-2}}} \|\theta\|_{L^{r}} \|\Lambda_{x_{1}}^{\alpha}\partial_{x_{1}}\theta\|_{L^{2}}
 \nonumber\\
 \leq&C\|\Lambda^{\frac{2}{r}}\Lambda_{x_{1}}^{1-\alpha}\partial_{x_{1}}
\theta\|_{L^{2}} \|\theta\|_{L^{r}} \|\Lambda_{x_{1}}^{\alpha}\partial_{x_{1}}\theta\|_{L^{2}}.
\end{align}
According to \eqref{avbnuy8911} and \eqref{t2326001}, we obtain
\begin{align}\label{ghtyde230}
\|\Lambda^{\frac{2}{r}}\Lambda_{x_{1}}^{1-\alpha}\partial_{x_{1}}
\theta\|_{L^{2}} &\leq C \|\Lambda_{x_{1}}^{2-\alpha+\frac{2}{r}} \theta\|_{L^{2}} +C\|\Lambda_{x_{1}}^{2-\alpha}\Lambda_{x_{2}}^{\frac{2}{r}} \theta\|_{L^{2}}
\nonumber\\&\leq C\|\theta\|_{L^{2}}^{1-\frac{2-\alpha+\frac{2}{r}}{\alpha+1}} \|\Lambda_{x_{1}}^{\alpha+1}\theta\|_{L^{2}}^{
 \frac{2-\alpha+\frac{2}{r}}{\alpha+1}}
\nonumber\\&\quad+
C\|\theta\|_{L^{2}}^{1-\frac{2-\alpha}{\alpha+1}-\frac{\frac{2}{r}}
{\beta+1}} \|\Lambda_{x_{2}}^{\beta+1}\theta\|_{L^{2}}^{\frac{\frac{2}{r}}{\beta+1}}
 \|\Lambda_{x_{1}}^{\alpha+1}\theta\|_{L^{2}}
 ^{\frac{2-\alpha}{\alpha+1}}\nonumber\\&\leq C\|\theta\|_{L^{2}}^{1-\frac{2-\alpha+\frac{2}{r}}{\alpha+1}} \|\Lambda_{x_{1}}^{\alpha}\nabla\theta\|_{L^{2}}^{
 \frac{2-\alpha+\frac{2}{r}}{\alpha+1}}
\nonumber\\&\quad+
C\|\theta\|_{L^{2}}^{1-\frac{2-\alpha}{\alpha+1}-\frac{\frac{2}{r}}
{\beta+1}} \|\Lambda_{x_{2}}^{\beta}\nabla\theta\|_{L^{2}}^{\frac{\frac{2}{r}}{\beta+1}}
 \|\Lambda_{x_{1}}^{\alpha}\nabla\theta\|_{L^{2}}
 ^{\frac{2-\alpha}{\alpha+1}},
\end{align}
where
$$r> \frac{2(\alpha+1)}{\min\{\alpha+1,\,\beta+1\}(2\alpha-1)}.$$
Inserting \eqref{ghtyde230} into \eqref{ghtyde229} and using the H$\rm\ddot{o}$lder inequality, it gives
\begin{align} \label{ghtyde231}
\mathcal{J}_{1}
 \leq& \eta\|\Lambda_{x_{2}}^{\beta} \nabla\theta\|_{L^{2}}^{2}+\eta\|\Lambda_{x_{1}}^{\alpha} \nabla\theta\|_{L^{2}}^{2}+ C_{\eta}\|\theta\|_{L^{r}}^{\frac{2}{1-\frac{2-\alpha+\frac{2}{r}}
 {\alpha+1}}}\|\theta\|_{L^{2}}^{2}+ C_{\eta}\|\theta\|_{L^{r}}^{\frac{2}{1-\frac{2-\alpha}{\alpha+1}-\frac{\frac{2}{r}}
{\beta+1}}}\|\theta\|_{L^{2}}^{2}\nonumber\\ =& \eta\|\Lambda_{x_{2}}^{\beta} \nabla\theta\|_{L^{2}}^{2}+\eta\|\Lambda_{x_{1}}^{\alpha} \nabla\theta\|_{L^{2}}^{2} + C_{\eta}\|\theta\|_{L^{r}}^{\frac{2(\alpha+1)r}{(2\alpha-1)r-2}}\|\theta\|_{L^{2}}^{2}
\nonumber\\&+ C_{\eta}\|\theta\|_{L^{r}}^{\frac{2(\alpha+1)(\beta+1)r}
{(2\alpha-1)(\beta+1)r-2(\alpha+1)}}\|\theta\|_{L^{2}}^{2}.
\end{align}
Integration by parts, $\mathcal{J}_{2}$ can be rewritten as follows
\begin{align}
\mathcal{J}_{2}&=\int_{\mathbb{R}^{2}}{ u_{2} \partial_{x_{1}x_{2}}\theta\partial_{x_{1}}\theta\,dx}
+\int_{\mathbb{R}^{2}}{ u_{2} \partial_{x_{2}}\theta\partial_{x_{1}x_{1}}\theta\,dx}
\nonumber\\
&=\frac{1}{2}\int_{\mathbb{R}^{2}}{\partial_{x_{1}} u_{1} \partial_{x_{1}}\theta\partial_{x_{1}}\theta\,dx}
+\int_{\mathbb{R}^{2}}{ u_{2} \partial_{x_{2}}\theta\partial_{x_{1}x_{1}}\theta\,dx}
\nonumber\\
&\triangleq-\frac{1}{2}\mathcal{J}_{1}+\widetilde{\mathcal{J}_{2}}.\nonumber
\end{align}
Taking $\widetilde{\delta}=\frac{(\alpha+1)\beta}{\beta+1}$ and using \eqref{yzz3} as well as \eqref{ghpoiu908}, we may conclude that
\begin{align}\label{ghtyde232}
\widetilde{\mathcal{J}_{2}} \leq&C\|\Lambda_{x_{1}}^{1-\widetilde{\delta}}\partial_{x_{1}}
\theta\|_{L^{\frac{2\widetilde{r}}{\widetilde{r}-2}}} \|\Lambda_{x_{1}}^{\widetilde{\delta}}(u_{2} \partial_{x_{2}}\theta)\|_{L^{\frac{2\widetilde{r}}{\widetilde{r}+2}}}
\nonumber\\
\leq&C\|\Lambda_{x_{1}}^{1-\widetilde{\delta}}\partial_{x_{1}}
\theta\|_{L^{\frac{2\widetilde{r}}{\widetilde{r}-2}}} (\|\Lambda_{x_{1}}^{\widetilde{\delta}}(u_{2} \partial_{x_{2}}\theta)-\Lambda_{x_{1}}^{\widetilde{\delta}}u_{2} \partial_{x_{2}}\theta\|_{L^{\frac{2\widetilde{r}}{\widetilde{r}+2}}}
+\|\Lambda_{x_{1}}^{\widetilde{\delta}}u_{2} \partial_{x_{2}}\theta\|_{L^{\frac{2\widetilde{r}}{\widetilde{r}+2}}})
\nonumber\\
 \leq& C\|\Lambda_{x_{1}}^{1-\widetilde{\delta}}\partial_{x_{1}}
 \theta\|_{L^{\frac{2\widetilde{r}}{\widetilde{r}-2}}}  (\|u_{2}\|_{L^{\widetilde{r}}}\|\Lambda_{x_{1}}^{\widetilde{\delta}}
 \partial_{x_{2}}\theta\|_{L^{2}}
+\|\Lambda_{x_{1}}^{\widetilde{\delta}}u_{2}
\|_{L^{\frac{2(\beta+1)\widetilde{r}}{\beta\widetilde{r}+2}}}
\|\partial_{x_{2}}\theta\|_{L^{\frac{2(\beta+1)\widetilde{r}}
{2\beta+\widetilde{r}}}})
\nonumber\\
 \leq& C\|\Lambda^{\frac{2}{\widetilde{r}}}\Lambda_{x_{1}}^{1-\widetilde{\delta}}\partial_{x_{1}}
 \theta\|_{L^{2}}  (\|\theta\|_{L^{\widetilde{r}}}\|\Lambda_{x_{1}}^{\alpha+1}\theta
\|_{L^{2}}^{\frac{\beta}{\beta+1}} \|\Lambda_{x_{2}}^{\beta+1}\theta
\|_{L^{2}}^{\frac{1}{\beta+1}}
\nonumber\\&+\|u_{2}\|_{L^{r}}^{\frac{1}{\beta+1}}\|\Lambda_{x_{1}}^{\alpha+1}u_{2}
\|_{L^{2}}^{\frac{\beta}{\beta+1}}
\|\theta\|_{L^{r}}^{\frac{\beta}{\beta+1}}\|\Lambda_{x_{2}}^{\beta+1}\theta
\|_{L^{2}}^{\frac{1}{\beta+1}})\nonumber\\
 \leq& C\|\Lambda^{\frac{2}{\widetilde{r}}}
 \Lambda_{x_{1}}^{1-\widetilde{\delta}}\partial_{x_{1}}
 \theta\|_{L^{2}}  \|\theta\|_{L^{\widetilde{r}}}\|\Lambda_{x_{1}}^{\alpha+1}\theta
\|_{L^{2}}^{\frac{\beta}{\beta+1}} \|\Lambda_{x_{2}}^{\beta+1}\theta
\|_{L^{2}}^{\frac{1}{\beta+1}}.
\end{align}
Again, we resort to \eqref{avbnuy8911} and \eqref{t2326001} to show
\begin{align}\label{ghtyde233}
\|\Lambda^{\frac{2}{\widetilde{r}}}
 \Lambda_{x_{1}}^{1-\widetilde{\delta}}\partial_{x_{1}}
 \theta\|_{L^{2}}  &\leq C \|\Lambda_{x_{1}}^{2-\widetilde{\delta}+\frac{2}{\widetilde{r}}} \theta\|_{L^{2}} +C\|\Lambda_{x_{1}}^{2-\widetilde{\delta}}\Lambda_{x_{2}}^{\frac{2}{\widetilde{r}}} \theta\|_{L^{2}}
\nonumber\\&\leq C\|\theta\|_{L^{2}}^{1-\frac{2-\widetilde{\delta}+\frac{2}{\widetilde{r}}}{\alpha+1}} \|\Lambda_{x_{1}}^{\alpha+1}\theta\|_{L^{2}}^{
 \frac{2-\widetilde{\delta}+\frac{2}{\widetilde{r}}}{\alpha+1}}
\nonumber\\&\quad+
C\|\theta\|_{L^{2}}^{1-\frac{2-\widetilde{\delta}}{\alpha+1}-\frac{\frac{2}
{\widetilde{r}}}
{\beta+1}} \|\Lambda_{x_{2}}^{\beta+1}\theta\|_{L^{2}}^{\frac{\frac{2}{\widetilde{r}}}{\beta+1}}
 \|\Lambda_{x_{1}}^{\alpha+1}\theta\|_{L^{2}}
 ^{\frac{2-\widetilde{\delta}}{\alpha+1}}\nonumber\\&\leq C\|\theta\|_{L^{2}}^{1-\frac{2-\widetilde{\delta}+\frac{2}{\widetilde{r}}}{\alpha+1}} \|\Lambda_{x_{1}}^{\alpha}\nabla\theta\|_{L^{2}}^{
 \frac{2-\widetilde{\delta}+\frac{2}{\widetilde{r}}}{\alpha+1}}
\nonumber\\&\quad+
C\|\theta\|_{L^{2}}^{1-\frac{2-\widetilde{\delta}}{\alpha+1}-\frac{\frac{2}
{\widetilde{r}}}
{\beta+1}} \|\Lambda_{x_{2}}^{\beta}\nabla\theta\|_{L^{2}}^{\frac{\frac{2}{\widetilde{r}}}{\beta+1}}
 \|\Lambda_{x_{1}}^{\alpha}\nabla\theta\|_{L^{2}}
 ^{\frac{2-\widetilde{\delta}}{\alpha+1}},
\end{align}
where
$$\widetilde{r}>\frac{2\max\{\alpha+1,\,\beta+1\}}{2\alpha\beta+\alpha-1}.$$
This requires the condition
$$\beta>\frac{1-\alpha}{2\alpha}.$$
Keeping in mind $\widetilde{\delta}=\frac{(\alpha+1)\beta}{\beta+1}$, we deduce by inserting \eqref{ghtyde233} into \eqref{ghtyde232} and using the H$\rm\ddot{o}$lder inequality that
\begin{align} \label{ghtyde234}
\widetilde{\mathcal{J}_{2}}
 \leq& \eta\|\Lambda_{x_{2}}^{\beta} \nabla\theta\|_{L^{2}}^{2}+\eta\|\Lambda_{x_{1}}^{\alpha} \nabla\theta\|_{L^{2}}^{2}+ C_{\eta}\|\theta\|_{L^{\widetilde{r}}}^{\frac{2}{1-\frac{2-\widetilde{\delta}
 +\frac{2}{\widetilde{r}}}{\alpha+1}}}\theta\|_{L^{2}}^{2}+ C_{\eta}\|\theta\|_{L^{\widetilde{r}}}^{\frac{2}{1-\frac{2-\widetilde{\delta}}{\alpha+1}-\frac{\frac{2}
{\widetilde{r}}}
{\beta+1}}}\theta\|_{L^{2}}^{2}\nonumber\\ =& \eta\|\Lambda_{x_{2}}^{\beta} \nabla\theta\|_{L^{2}}^{2}+\eta\|\Lambda_{x_{1}}^{\alpha} \nabla\theta\|_{L^{2}}^{2} + C_{\eta}\|\theta\|_{L^{\widetilde{r}}}^{\frac{2(\alpha+1)(\beta+1)\widetilde{r}}{
(2\alpha\beta+\alpha-1)\widetilde{r}-2(\beta+1)}}\|\theta\|_{L^{2}}^{2}
\nonumber\\&+ C_{\eta}\|\theta\|_{L^{\widetilde{r}}}^{\frac{2(\alpha+1)(\beta+1)\widetilde{r}}{
(2\alpha\beta+\alpha-1)\widetilde{r}-2(\alpha+1)}}\|\theta\|_{L^{2}}^{2}.
\end{align}
The third term $\mathcal{J}_{3} $ should be handled differently. Firstly, one obtains
\begin{align}\label{ghtyde235}
\mathcal{J}_{3} =&\int_{\mathbb{R}^{2}}{\theta\partial_{x_{1}x_{2}}u_{1}\partial_{x_{2}}\theta\,dx}
+\int_{\mathbb{R}^{2}}{\theta\partial_{x_{2}}u_{1}\partial_{x_{1}x_{2}}\theta\,dx}
\nonumber\\
 \leq&C\|\Lambda_{x_{1}}^{1-\widetilde{\delta}}\partial_{x_{2}}u_{1}
 \|_{L^{\frac{2\widehat{r}}{\widehat{r}-2}}} \|\Lambda_{x_{1}}^{\widetilde{\delta}}(\theta\partial_{x_{2}}\theta)
 \|_{L^{\frac{2\widehat{r}}{\widehat{r}+2}}}
+C\|\Lambda_{x_{1}}^{1-\widetilde{\delta}}\partial_{x_{2}}\theta
\|_{L^{\frac{2\widehat{r}}{\widehat{r}-2}}} \|\Lambda_{x_{1}}^{\widetilde{\delta}}(\theta\partial_{x_{2}}u_{1})
\|_{L^{\frac{2\widehat{r}}{\widehat{r}+2}}}
\nonumber\\
 \leq& C\|\Lambda^{\frac{2}{\widehat{r}}}
 \Lambda_{x_{1}}^{1-\widetilde{\delta}}\partial_{x_{2}}
 \theta\|_{L^{2}}  \|\theta\|_{L^{\widehat{r}}}\|\Lambda_{x_{1}}^{\alpha+1}\theta
\|_{L^{2}}^{\frac{\beta}{\beta+1}} \|\Lambda_{x_{2}}^{\beta+1}\theta
\|_{L^{2}}^{\frac{1}{\beta+1}}.
\end{align}
Choosing $\widehat{r}\in (2,\infty)$ as
\begin{align}
\widehat{r}=\frac{2\alpha}{\alpha\beta+1-\widetilde{\delta}
-(\beta+1)(1+\widetilde{\epsilon})(1-\widetilde{\delta})},\nonumber
\end{align}
with $\widetilde{\delta}=\frac{(\alpha+1)\beta}{\beta+1}>\frac{(\beta+1)\widetilde{\epsilon}+(1-\alpha)
\beta}{(\beta+1)\widetilde{\epsilon}+\beta}$ for positive $\widetilde{\epsilon}<\frac{(2\alpha\beta+\alpha-1)\beta}{(\beta+1)(1-\alpha\beta)}$, we thus deduce from \eqref{t2326001} and \eqref{dffgtyu889} that
\begin{align}\label{ghtyde236}
\|\Lambda^{\frac{2}{\widehat{r}}}
 \Lambda_{x_{1}}^{1-\widetilde{\delta}}\partial_{x_{2}}
 \theta\|_{L^{2}}  &\leq C \|\Lambda_{x_{1}}^{1-\widetilde{\delta}+\frac{2}{\widehat{r}}}\Lambda_{x_{2}}\theta\|_{L^{2}} +C\|\Lambda_{x_{1}}^{1-\widetilde{\delta}}\Lambda_{x_{2}}^{1+\frac{2}{\widehat{r}}} \theta\|_{L^{2}}
\nonumber\\&\leq C\|\theta\|_{L^{2}}^{1-\frac{1-\widetilde{\delta}}{\alpha}-\frac{1+\frac{2}
{\widehat{r}}-\frac{1-\widetilde{\delta}}{\alpha}}
{\beta+1}} \|\Lambda_{x_{2}}^{\beta+1-\frac{2}{\widehat{r}}}
\Lambda_{x_{1}}^{\frac{2}{\widehat{r}}}\theta
\|_{L^{2}}^{\frac{1+\frac{2}{\widehat{r}}
-\frac{1-\widetilde{\delta}}{\alpha}}{\beta+1}}
 \|\Lambda_{x_{1}}^{\alpha+\frac{2(1+\widetilde{\epsilon})}{r}}
 \Lambda_{x_{2}}^{1-\frac{2(1+\widetilde{\epsilon})}{r}} \theta\|_{L^{2}}
 ^{\frac{1-\widetilde{\delta}}{\alpha}}
\nonumber\\&\quad+
C\|\theta\|_{L^{2}}^{1-\frac{1-\widetilde{\delta}}{\alpha}-\frac{1+\frac{2}
{\widehat{r}}-\frac{1-\widetilde{\delta}}{\alpha}}
{\beta+1}} \|\Lambda_{x_{2}}^{\beta+1}\theta\|_{L^{2}}^{\frac{1+\frac{2}{\widehat{r}}
-\frac{1-\widetilde{\delta}}{\alpha}}{\beta+1}}
 \|\Lambda_{x_{1}}^{\alpha}\Lambda_{x_{2}}\theta\|_{L^{2}}
 ^{\frac{1-\widetilde{\delta}}{\alpha}}
 \nonumber\\&\leq
C\|\theta\|_{L^{2}}^{1-\frac{1-\widetilde{\delta}}{\alpha}-\frac{1+\frac{2}
{\widehat{r}}-\frac{1-\widetilde{\delta}}{\alpha}}
{\beta+1}} \|\Lambda_{x_{2}}^{\beta}\nabla\theta\|_{L^{2}}^{\frac{1+\frac{2}{\widehat{r}}
-\frac{1-\widetilde{\delta}}{\alpha}}{\beta+1}}
 \|\Lambda_{x_{1}}^{\alpha}\nabla\theta\|_{L^{2}}
 ^{\frac{1-\widetilde{\delta}}{\alpha}}
  \nonumber\\&=
C\|\theta\|_{L^{2}}^{\frac{\widetilde{\epsilon}(1-\alpha\beta)}{\alpha(\beta+1)}} \|\Lambda_{x_{2}}^{\beta}\nabla\theta\|_{L^{2}}^{1-\frac{(1+\widetilde{\epsilon})
(1-\alpha\beta)}{\alpha(\beta+1)}}
 \|\Lambda_{x_{1}}^{\alpha}\nabla\theta\|_{L^{2}}
 ^{\frac{1-\alpha\beta}{\alpha(\beta+1)}}.
\end{align}
Inserting \eqref{ghtyde236} into \eqref{ghtyde235} and using the H$\rm\ddot{o}$lder inequality imply
\begin{align} \label{ghtyde237}
 {\mathcal{J}_{3}}
 \leq& \eta\|\Lambda_{x_{2}}^{\beta} \nabla\theta\|_{L^{2}}^{2}+\eta\|\Lambda_{x_{1}}^{\alpha} \nabla\theta\|_{L^{2}}^{2}+ C_{\eta}\|\theta\|_{L^{\widehat{r}}}^{\frac{2\alpha(\beta+1)}
 {\widetilde{\epsilon}(1-\alpha\beta)}}\|\theta\|_{L^{2}}^{2}.
\end{align}
Using the same argument adopted in dealing with (\ref{ghtyde237}) yields
\begin{align}\label{ghtyde238}
\mathcal{J}_{4} =&\int_{\mathbb{R}^{2}}{\partial_{x_{1}}u_{1} \partial_{x_{2}}\theta\partial_{x_{2}}\theta\,dx}
\nonumber\\
 =&-2
\int_{\mathbb{R}^{2}}{ u_{1} \partial_{x_{2}}\theta\partial_{x_{1}x_{2}}\theta\,dx}
\nonumber\\
 \leq&
C\|\Lambda_{x_{1}}^{1-\widetilde{\delta}}\partial_{x_{2}}
\theta\|_{L^{\frac{2\widehat{r}}{\widehat{r}-2}}} \|\Lambda_{x_{1}}^{\widetilde{\delta}}(u_{1}\partial_{x_{2}}
\theta)\|_{L^{\frac{2\widehat{r}}{\widehat{r}+2}}}
\nonumber\\
 \leq&\eta\|\Lambda_{x_{2}}^{\beta} \nabla\theta\|_{L^{2}}^{2}+\eta\|\Lambda_{x_{1}}^{\alpha} \nabla\theta\|_{L^{2}}^{2}+ C_{\eta}\|\theta\|_{L^{\widehat{r}}}^{\frac{2\alpha(\beta+1)}
 {\widetilde{\epsilon}(1-\alpha\beta)}}\|\theta\|_{L^{2}}^{2}.
\end{align}
Putting \eqref{ghtyde231}, \eqref{ghtyde234}, \eqref{ghtyde237} and \eqref{ghtyde238} into \eqref{t36t0ghd} and taking $\eta$ sufficiently small, we have
\begin{eqnarray}\label{ghtyde239}
 \frac{d}{dt}\|\nabla \theta(t)\|_{L^{2}}^{2}+
 \|\Lambda_{x_{1}}^{\alpha}\nabla\theta\|_{L^{2}}^{2}+
 \|\Lambda_{x_{2}}^{\beta}\nabla\theta\|_{L^{2}}^{2} \leq D(t)\| \theta\|_{L^{2}}^{2},
\end{eqnarray}
where
\begin{align}D(t)&=
C\Big(\|\theta\|_{L^{r}}^{\frac{2(\alpha+1)r}{(2\alpha-1)r-2}}+
\|\theta\|_{L^{r}}^{\frac{2(\alpha+1)(\beta+1)r}
{(2\alpha-1)(\beta+1)r-2(\alpha+1)}}+\|\theta\|_{L^{\widetilde{r}}}^{\frac{2(\alpha+1)(\beta+1)\widetilde{r}}{
(2\alpha\beta+\alpha-1)\widetilde{r}-2(\beta+1)}}\nonumber\\&\quad+
\|\theta\|_{L^{\widetilde{r}}}^{\frac{2(\alpha+1)(\beta+1)\widetilde{r}}{
(2\alpha\beta+\alpha-1)\widetilde{r}-2(\alpha+1)}}+\|\theta\|_{L^{\widehat{r}}}^{\frac{2\alpha(\beta+1)}
 {\widetilde{\epsilon}(1-\alpha\beta)}}\Big)\nonumber
 \end{align}
with $\frac{2(\alpha+1)}{\min\{\alpha+1,\,\beta+1\}(2\alpha-1)}<r<\infty$, $\frac{2\max\{\alpha+1,\,\beta+1\}}{2\alpha\beta+\alpha-1}<\widetilde{r}<\infty$
and $\widehat{r}=\frac{2(\alpha+1)\beta}{\alpha(\alpha+1)\beta
-\widetilde{\epsilon}(1-\alpha\beta)}$ for any $0<\widetilde{\epsilon}<\frac{(2\alpha\beta+\alpha-1)\beta}{(\beta+1)(1-\alpha\beta)}$.
By the estimates of Proposition \ref{Spr1}, we obtain \eqref{cvgyt25y027} by integrating \eqref{ghtyde239} in time. This completes the proof of Proposition \ref{Lpp30we3}.
\end{proof}

\vskip .1in
With the global
${\dot{H}}^1$-bound of $\theta$ in hand, we are ready to show the global
${\dot{H}}^2$-bound.
\begin{Pros}\label{Lvfgt}
Let $\theta$ be the solution of (\ref{SQG}). If $\alpha$ and $\beta$ satisfy (\ref{sdf2334}),
then for any $t>0$
\begin{eqnarray}\label{ghtyde240}
\|\Delta\theta(t)\|_{L^{2}}^{2}+ \int_{0}^{t}{(\|\Lambda_{x_{1}}^{\alpha}\Delta\theta(
\tau)\|_{L^{2}}^{2}+\|\Lambda_{x_{2}}^{\beta}\Delta\theta(
\tau)\|_{L^{2}}^{2})\,d\tau}\leq
C(t,\,\theta_{0}),
\end{eqnarray}
where $C(t, \,\theta_{0})$ is a constant depending on $t$ and the initial norm $\|\theta_0\|_{H^{2}}$.
\end{Pros}

\begin{proof}
It follows from the proof of \cite[Proposition 3.4]{Yenonli} that
\begin{eqnarray*}
 \frac{d}{dt}\|\Delta\theta(t)\|_{L^{2}}^{2}+\|\Lambda_{x_{1}}^{\alpha}\Delta\theta\|_{L^{2}}^{2}+
 \|\Lambda_{x_{2}}^{\beta}\Delta\theta\|_{L^{2}}^{2}
\leq C
\|\nabla\theta\|_{L^2}^{\frac{2(2\alpha-1)}{4\alpha-1}}\|\Lambda_{x_{1}}^{\alpha}\nabla\theta\|_{L^2}^{ \frac{2}{4\alpha-1}}\,\|\Delta \theta\|_{L^2}^{2},
\end{eqnarray*}
for $\alpha>\frac{1}{2}$, while for $\beta>\frac{1}{2}$
\begin{eqnarray*}
 \frac{d}{dt}\|\Delta\theta(t)\|_{L^{2}}^{2}+\|\Lambda_{x_{1}}^{\alpha}\Delta\theta\|_{L^{2}}^{2}+
 \|\Lambda_{x_{2}}^{\beta}\Delta\theta\|_{L^{2}}^{2}
\leq C
\|\nabla\theta\|_{L^2}^{\frac{2(2\beta-1)}{4\beta-1}}
\|\Lambda_{x_{2}}^{\beta}\nabla\theta\|_{L^2}^{\frac{2}{4\beta-1}}\,\|\Delta \theta\|_{L^2}^{2}.
\end{eqnarray*}
The classical Gronwall inequality along with (\ref{asddfgdfht302}) and (\ref{cvgyt25y027}) allow us to derive \eqref{ghtyde240}.
Therefore, we end the proof of Proposition \ref{Lvfgt}.
\end{proof}

\vskip .1in
With the above obtained
$H^2$-bound in hand, we are now ready to establish the global $H^\rho$-estimate.
\begin{proof}[Proof of the global $H^\rho$-estimate]
By (3.38) and (3.39) of \cite{Yenonli}, we deduce from the obtained $H^\rho$-bound that \begin{eqnarray*}  \int_{0}^{t}(\|\nabla \theta(\tau)\|_{L^{\infty}}+\|\nabla u(\tau)\|_{L^{\infty}})\,d\tau\leq
C(t,\,\theta_{0}).
\end{eqnarray*}
According to the basic $H^{\rho}$-estimate of (\ref{SQG}), one has
\begin{eqnarray*}
 \frac{d}{dt}
\|\theta(t)\|_{H^{\rho}}^{2} +\|\Lambda_{x_{1}}^{\alpha}\theta\|_{H^{\rho}}^{2}
+\|\Lambda_{x_{2}}^{\beta}\theta\|_{H^{\rho}}^{2}
 \leq  C(\|\nabla u\|_{L^{\infty}}+\|\nabla \theta\|_{L^{\infty}}) \|\theta\|_{H^{\rho}}^{2}.
\end{eqnarray*}
The Gronwall inequality implies that
$$
\|\theta(t)\|_{H^{\rho}}^{2}+\int_{0}^{t}(\|\Lambda_{x_{1}}^{\alpha}
\theta(\tau)\|_{H^{\rho}}^{2}
+\|\Lambda_{x_{2}}^{\beta}\theta(\tau)\|_{H^{\rho}}^{2})\,d\tau\leq
C(t,\,\theta_{0}).$$
With the above estimate in hand, the uniqueness is not difficult to derive (see P.95--P.97 of \cite{Yenonli} for more details). This ends the proof of the global regularity part of Theorem \ref{OKTh1}.
\end{proof}

\vskip .1in
\section{The proof of Theorem \ref{OKTh1}: Time decay}\setcounter{equation}{0}
\label{sec3}

This section is devoted to the proof of the time decay as stated in Theorem \ref{OKTh1}. Here we are aiming at developing effective approaches to understand the large-time behavior of the system(s) with anisotropic fractional dissipation.
We begin with the following simple lemma, which is a direct consequence of the classical ODE argument.
\begin{lemma}\label{lem111}
If the non-negative function $X(t)$ satisfies
\begin{eqnarray}
\frac{d X(t)}{dt}+\nu X^{k}(t)\leq 0,\quad X(0)\geq 0,\nonumber
\end{eqnarray}
where the constants $\nu>0$ and $k>1$, then it holds
$$X(t)\leq \left[X(0)^{1-k}+(k-1)\nu t\right]^{-\frac{1}{k-1}}.$$
\end{lemma}

\vskip .1in
With Lemma \ref{lem111} and Proposition \ref{Spr1} in hand, we are able to show the basic $L^2$-decay rate for any weak solution of (\ref{SQG}).
\begin{Pros}\label{Prosty666} Let $\alpha \in (0,1],\, \beta\in (0,1]$ and $\theta_{0}\in L^{2}(\mathbb{R}^{2})\cap L^p(\mathbb{R}^{2})$ with $p\in [1,2)$, then any weak solution of (\ref{SQG}) satisfies
\begin{equation}\label{Proshsh78}
\|\theta(t)\|_{L^2}\leq C_{0}(1+t)^{-\frac{(\alpha+\beta)(2-p)}{4\alpha\beta p}},\end{equation}
 where the constant $C_{0}>0$ depends only on $\alpha,\beta,p$ and $\theta_{0}$.
\end{Pros}

\begin{rem}
We remark that the argument adopted in establishing the decay estimate of the previous works relies heavily on the classical Fourier Splitting method introduced by Schonbek \cite{Schonbek}. However, in the proof of Proposition \ref{Prosty666}, we chose not to apply Fourier Splitting techniques but to use the direct interpolation inequality.
\end{rem}
\begin{proof}
By the simple energy estimate, it gives
$$
\frac{1}{2}\frac{d}{dt}\|\theta(t)\|_{L^{2}}^{2}+
\int_{\mathbb{R}^{2}}{\Lambda_{x_{1}}^{2\alpha}\theta \theta\,dx}+
\int_{\mathbb{R}^{2}}{\Lambda_{x_{2}}^{2\beta}\theta\theta\,dx}=0,
$$
which equals to
\begin{align}\label{fgfgcvf3}
\frac{1}{2}\frac{d}{dt}\|\theta(t)\|_{L^{2}}^{2}+
 \|\Lambda_{x_{1}}^{\alpha}\theta\|_{L^{2}}^{2}+
 \|\Lambda_{x_{2}}^{\beta}\theta\|_{L^{2}}^{2}= 0.
\end{align}
Now we claim that the following anisotropic interpolation inequality holds for $p\in [1,\,2)$
\begin{align}\label{fggh89p1}
\|\theta\|_{L^{2}}\leq C\|\theta\|_{L^{p}}^{1-\frac{(\alpha+\beta)(\frac{1}{p}-\frac{1}{2})}{(\alpha+\beta)(\frac{1}{p}-\frac{1}{2})+\alpha\beta}}
\|\Lambda_{x_{1}}^{\alpha}\theta\|_{L^{2}}
^{\frac{\beta(\frac{1}{p}-\frac{1}{2})}{(\alpha+\beta)(\frac{1}{p}-\frac{1}{2})+\alpha\beta}}\|\Lambda_{x_{2}}^{\beta}\theta\|_{L^{2}}
^{\frac{\alpha(\frac{1}{p}-\frac{1}{2})}{(\alpha+\beta)(\frac{1}{p}-\frac{1}{2})+\alpha\beta}}.
\end{align}
Making use of several interpolation inequalities in one-dimension, we can show
\begin{align}\label{fgfgcvf4}
\|\theta\|_{L^{2}}&=\left\|\|\theta(x_{1},x_{2})\|_{L_{x_{1}}^{2}}\right\|_{L_{x_{2}}^{2}}\nonumber\\
&\leq C\left\|\|\theta(x_{1},x_{2})\|_{L_{x_{1}}^{p}}^{1-\mu}
\|\Lambda_{x_{1}}^{k_{1}}\theta(x_{1},x_{2})\|_{L_{x_{1}}^{2}}^{\mu}\right\|_{L_{x_{2}}^{2}}
\nonumber\\
&\leq C\left\|\|\theta(x_{1},x_{2})\|_{L_{x_{1}}^{p}}^{1-\mu}
\right\|_{L_{x_{2}}^{\frac{2r}{r-2}}}\left\|
\|\Lambda_{x_{1}}^{k_{1}}\theta(x_{1},x_{2})\|_{L_{x_{1}}^{2}}^{\mu}\right\|_{L_{x_{2}}^{r}}
\nonumber\\
&=C\left\|\|\theta(x_{1},x_{2})\|_{L_{x_{1}}^{p}}
\right\|_{L_{x_{2}}^{\frac{2(1-\mu)r}{r-2}}}^{1-\mu}\left\|
\|\Lambda_{x_{1}}^{k_{1}}\theta(x_{1},x_{2})\|_{L_{x_{1}}^{2}}\right\|_{L_{x_{2}}^{\mu r}}^{\mu}
\nonumber\\
&=C\left\|\|\theta(x_{1},x_{2})\|_{L_{x_{1}}^{p}}
\right\|_{L_{x_{2}}^{p}}^{1-\mu}\left\|
\|\Lambda_{x_{1}}^{k_{1}}\theta(x_{1},x_{2})\|_{L_{x_{1}}^{2}}\right\|_{L_{x_{2}}^{\mu r}}^{\mu}
\nonumber\\
&\leq C\left\|\|\theta(x_{1},x_{2})\|_{L_{x_{1}}^{p}}
\right\|_{L_{x_{2}}^{p}}^{1-\mu}\left\|
\|\Lambda_{x_{2}}^{k_{2}}\Lambda_{x_{1}}^{k_{1}}\theta(x_{1},x_{2})\|_{L_{x_{1}}^{2}}\right\|_{L_{x_{2}}^{2}}^{\mu}
\nonumber\\
&=C\left\|\theta
\right\|_{L^{p}}^{1-\mu}\left\|
 \Lambda_{x_{2}}^{k_{2}}\Lambda_{x_{1}}^{k_{1}}\theta(x_{1},x_{2})\right\|_{L^{2}}^{\mu},
\end{align}
where $\mu\in (0,1),k_{1}>0,k_{2}>0,r> 2,p\in[1,2)$ should satisfy
\begin{align}\label{fgfgcvf5}
\frac{1}{2}=\frac{1-\mu}{p}+\left(\frac{1}{2}-k_{1}\right)\mu,\quad \frac{2(1-\mu)r}{r-2}=p,\quad  \frac{1}{2}-k_{2}=\frac{1}{\mu r}.
\end{align}
Now we further choose
\begin{align}\label{fgfgcvf6}
\frac{k_{1}}{\alpha}+\frac{k_{2}}{\beta}=1,
\end{align}
then it implies
\begin{align}\label{fgfgcvf7}
\|\Lambda_{x_{2}}^{k_{2}}\Lambda_{x_{1}}^{k_{1}}\theta(x_{1},x_{2})
\|_{L^{2}}
 =&\Big(\int_{\mathbb{R}^{2}}{|\xi_{2}|^{2k_{2}}|\xi_{1}|^{2k_{1}}
|\widehat{\theta}(\xi)|^{2}\,d\xi}\Big)^{\frac{1}{2}}\nonumber\\
 =&\Big(\int_{\mathbb{R}^{2}}{\big(|\xi_{2}|^{2k_{2}}
|\widehat{h}(\xi)|^{\frac{2k_{2}}{\beta}}\big)
\big(|\xi_{1}|^{2k_{1}}|\widehat{\theta}(\xi)|
^{\frac{2k_{1}}{\alpha}}}\big)
\,d\xi\Big)^{\frac{1}{2}}
\nonumber\\
 \leq&C\Big(\int_{\mathbb{R}^{2}}{|\xi_{2}|^{2\beta}
|\widehat{\theta}(\xi)|^{2}\,d\xi}\Big)^{\frac{k_{2}}{2\beta}}
\Big(\int_{\mathbb{R}^{2}}{|\xi_{1}|^{2\alpha}
|\widehat{\theta}(\xi)|^{2}\,d\xi}\Big)^{\frac{k_{1}}{2\alpha}}
\nonumber\\
 =&C
\|\Lambda_{x_{2}}^{\beta}\theta\|_{L^{2}}
^{\frac{k_{2}}{\beta}}
\|\Lambda_{x_{1}}^{\alpha}\theta\|_{L^{2}}
^{\frac{k_{1}}{\alpha}}.
\end{align}
Collecting \eqref{fgfgcvf5} and \eqref{fgfgcvf6}, we finally derive that
\begin{align}\label{fgfgcvf9}
k_{1}=k_{2}=\frac{\alpha\beta}{\alpha+\beta},\ \ \ \mu=\frac{\frac{1}{p}-\frac{1}{2}}{\frac{1}{p}+\frac{\alpha\beta}{\alpha+\beta}-\frac{1}{2}}, \ \
r=\frac{\frac{1}{p}+\frac{\alpha\beta}{\alpha+\beta}-\frac{1}{2}}{(\frac{1}{2}-\frac{\alpha\beta}{\alpha+\beta})
(\frac{1}{p}-\frac{1}{2})}.\end{align}
Putting \eqref{fgfgcvf7} into \eqref{fgfgcvf4} and using \eqref{fgfgcvf9} yields the desired \eqref{fggh89p1}.
Recalling the fact $\|\theta(t)\|_{L^p}\leq \|\theta_0\|_{L^p}$, it follows from \eqref{fggh89p1} that
\begin{align}
\|\theta\|_{L^{2}}&\leq C
\|\Lambda_{x_{1}}^{\alpha}\theta\|_{L^{2}}
^{\frac{\beta(\frac{1}{p}-\frac{1}{2})}{(\alpha+\beta)(\frac{1}{p}-\frac{1}{2})+\alpha\beta}}\|\Lambda_{x_{2}}^{\beta}\theta\|_{L^{2}}
^{\frac{\alpha(\frac{1}{p}-\frac{1}{2})}{(\alpha+\beta)(\frac{1}{p}-\frac{1}{2})+\alpha\beta}}\nonumber\\
&=C
(\|\Lambda_{x_{1}}^{\alpha}\theta\|_{L^{2}}^{2})
^{\frac{\beta(\frac{1}{p}-\frac{1}{2})}{2(\alpha+\beta)(\frac{1}{p}-\frac{1}{2})+2\alpha\beta}}
(\|\Lambda_{x_{2}}^{\beta}\theta\|_{L^{2}}^{2})
^{\frac{\alpha(\frac{1}{p}-\frac{1}{2})}{2(\alpha+\beta)(\frac{1}{p}-\frac{1}{2})+2\alpha\beta}}\nonumber\\
&\leq C
(\|\Lambda_{x_{1}}^{\alpha}\theta\|_{L^{2}}^{2}+\|\Lambda_{x_{2}}^{\beta}\theta\|_{L^{2}}^{2})
^{\frac{(\alpha+\beta)(\frac{1}{p}-\frac{1}{2})}{2(\alpha+\beta)(\frac{1}{p}-\frac{1}{2})+2\alpha\beta}},\nonumber
\end{align}
which yields
\begin{align}
\|\Lambda_{x_{1}}^{\alpha}\theta\|_{L^{2}}^{2}+\|\Lambda_{x_{2}}^{\beta}\theta\|_{L^{2}}^{2}\geq C\|\theta\|_{L^{2}}^{\frac{2(\alpha+\beta)(\frac{1}{p}-\frac{1}{2})+2\alpha\beta}{(\alpha+\beta)
(\frac{1}{p}-\frac{1}{2})}}.\nonumber
\end{align}
This together with \eqref{fgfgcvf3} shows
\begin{align}\label{fgfgcvf10}
\frac{d}{dt}\|\theta(t)\|_{L^{2}}^{2}+C(\|\theta\|_{L^{2}}^{2})^{\frac{(\alpha+\beta)(\frac{1}{p}-\frac{1}{2})+\alpha\beta}{(\alpha+\beta)
(\frac{1}{p}-\frac{1}{2})}}\leq 0.
\end{align}
Applying Lemma \ref{lem111} to \eqref{fgfgcvf10} gives
$$\|\theta(t)\|_{L^2}\leq C_{0}(1+t)^{-\frac{(\alpha+\beta)(2-p)}{4\alpha\beta p}}.$$
Therefore, we end the proof of Proposition \ref{Prosty666}.
\end{proof}

We remark that if we only have $\theta_{0}\in L^{2}(\mathbb{R}^{2})$, then we can show the following non-uniform decay result.
\begin{Pros}\label{mkjqq1}
Let $\min\{\alpha+3\beta,\,3\alpha+\beta\}\geq4\alpha\beta$ with $\alpha>0,\,\beta>0$ and $\theta_{0}\in L^{2}(\mathbb{R}^{2})$, then any weak solution of (\ref{SQG}) satisfies
\begin{align}\label{cvbmklp1}
\lim_{t\rightarrow\infty}\|\theta(t)\|_{L^2}=0.
\end{align}
\end{Pros}

\begin{rem}
The result of Proposition \ref{mkjqq1} for the case $0<\alpha=\beta\leq1$ corresponds to the result of Niche and Schonbek \cite[Theorem 1.1]{Nichesc11}.
\end{rem}

\begin{proof}
The proof is based on the Fourier splitting method. We first claim
\begin{align}\label{mkttye2}
|\widehat{\theta}(\xi,t)|\leq |\widehat{\theta_{0}}(\xi)|e^{-(|\xi_{1}|^{2\alpha}+|\xi_{2}|^{2\beta})t}+Ct|\xi|.
\end{align}
Taking the Fourier transformation to \eqref{SQG}, it yields
 \begin{align}
 \partial_{t}\widehat{\theta}(\xi,  t)
  +(|\xi_{1}|^{2\alpha}+|\xi_{2}|^{2\beta})\widehat{\theta}(\xi,  t)
 =-\widehat{(u\cdot\nabla)\theta}(\xi, t),\nonumber
\end{align}
which leads to
 \begin{align}\label{mkttye3}
 \widehat{\theta}(\xi,t)&=\widehat{\theta_{0}}(\xi)e^{-(|\xi_{1}|^{2\alpha}
 +|\xi_{2}|^{2\beta})t}-\int_{0}^{t}
e^{-(|\xi_{1}|^{2\alpha}+|\xi_{2}|^{2\beta})(t-\tau)}
\widehat{(u\cdot\nabla)\theta}(\xi,\tau)\,d\tau.
\end{align}
Using the fact that the Fourier transform is a bounded map from $L^{1}$ into $L^{\infty}$, we derive
\begin{align}
  |\widehat{(u\cdot\nabla)\theta}(\xi,  t)|
 = |\widehat{\nabla\cdot(u\theta)}(\xi,  t)|
 \leq |\xi||\widehat{(u\theta)}(\xi,  t)|
 \leq |\xi|\|u\theta\|_{L^{1}} \leq  |\xi|\,
           \|u\|_{L^{2}}\|\theta\|_{L^{2}}.\nonumber
\end{align}
Then we deduce from \eqref{mkttye3} that
 \begin{align}
 |\widehat{\theta}(\xi,t)|&\leq |\widehat{\theta_{0}}(\xi)e^{-(|\xi_{1}|^{2\alpha}
 +|\xi_{2}|^{2\beta})t}|+C|\xi|\int_{0}^{t}
e^{-(|\xi_{1}|^{2\alpha}+|\xi_{2}|^{2\beta})(t-\tau)}
\|u(\tau)\|_{L^{2}}\|\theta(\tau)\|_{L^{2}}\,d\tau\nonumber\\
&\leq |\widehat{\theta_{0}}(\xi)e^{-(|\xi_{1}|^{2\alpha}
 +|\xi_{2}|^{2\beta})t}|+C|\xi|\int_{0}^{t}
\|\theta(\tau)\|_{L^{2}}^{2}\,d\tau\nonumber\\
&\leq |\widehat{\theta_{0}}(\xi)|e^{-(|\xi_{1}|^{2\alpha}
 +|\xi_{2}|^{2\beta})t}+Ct|\xi|,\nonumber
\end{align}
which yields \eqref{mkttye2}.
Recalling \eqref{fgfgcvf3}, we have
\begin{align}\label{mkttye13}
\frac{1}{2}\frac{d}{dt}\|\theta(t)\|_{L^{2}}^{2}+
 \|\Lambda_{x_{1}}^{\alpha}\theta\|_{L^{2}}^{2}+
 \|\Lambda_{x_{2}}^{\beta}\theta\|_{L^{2}}^{2}= 0.
\end{align}
We denote
$$E(t)=\{\xi\in \mathbb{R}^2: \  |\xi_{1}|^{2\alpha}+|\xi_{2}|^{2\beta}\leq f(t)\},$$
where $f(t)\leq C(1+t)^{-1}$ is a smooth decreasing function to be fixed later on.
It is easy to check that
\begin{align}
\|\Lambda_{x_{1}}^{\alpha} \theta\|_{L^{2}}^{2}+
 \|\Lambda_{x_{2}}^{\beta} \theta\|_{L^{2}}^{2}
 &=\int_{\mathbb{R}^{2}}|\xi_{1}|^{2\alpha}|\widehat{ \theta}(t,\xi)|^{2}\,d\xi +\int_{\mathbb{R}^{2}}|\xi_{2}|^{2\beta}|\widehat{ \theta}(t,\xi)|^{2}\,d\xi\nonumber\\
 &=\int_{\mathbb{R}^{2}}(|\xi_{1}|^{2\alpha}+|\xi_{2}|^{2\beta})|\widehat{ \theta}(t,\xi)|^{2}\,d\xi
 \nonumber\\
 &\geq \int_{E(t)^{c}}(|\xi_{1}|^{2\alpha}+|\xi_{2}|^{2\beta})|\widehat{ \theta}(t,\xi)|^{2}\,d\xi
 \nonumber\\
 &\geq f(t)\int_{E(t)^{c}}|\widehat{ \theta}(t,\xi)|^{2}\,d\xi\nonumber\\
 &= f(t)\int_{\mathbb{R}^{2}}|\widehat{ \theta}(t,\xi)|^{2}\,d\xi
 -f(t)\int_{E(t)}|\widehat{ \theta}(t,\xi)|^{2}\,d\xi,\nonumber
\end{align}
which together with \eqref{mkttye13} implies
\begin{align} \label{mkttye14}
\frac{d}{dt}\| {\theta}(t)\|_{L^{2}}^{2}+f(t)\| {\theta}\|_{L^{2}}^{2}\leq& f(t)\int_{E(t)}|\widehat{ {\theta}}(t,\xi)|^{2}\,d\xi.
\end{align}
Making use of \eqref{mkttye2}, one obtains
\begin{align}
\int_{E(t)}|\widehat{ {\theta}}(t,\xi)|^{2}\,d\xi
&\leq
\int_{E(t)}|\widehat{\theta_{0}}(\xi)e^{-(|\xi_{1}|^{2\alpha}+|\xi_{2}|^{2\beta})t}|^{2}\, d\xi
+Ct^{2}\int_{E(t)}|\xi|^{2}\, d\xi\nonumber\\
&\leq
\int_{E(t)}|\widehat{\theta_{0}}(\xi)e^{-(|\xi_{1}|^{2\alpha}+|\xi_{2}|^{2\beta})t}|^{2}\, d\xi
+Ct^{2}\left(f(t)^{\frac{\alpha+3\beta}{2\alpha\beta}}+f(t)^{\frac{3\alpha+\beta}
{2\alpha\beta}}\right),\nonumber
\end{align}
where we have used \eqref{diff14} below, namely,
\begin{align}
\int_{E(t)}|\xi|^{2}\,d\xi&\leq C\left(f(t)^{\frac{\alpha+3\beta}{2\alpha\beta}}+f(t)^{\frac{3\alpha+\beta}
{2\alpha\beta}}\right).\nonumber
\end{align}
Thus, one concludes from \eqref{mkttye14} that
\begin{align} \label{mkttye15}
\frac{d}{dt}\| {\theta}(t)\|_{L^{2}}^{2}+f(t)\| {\theta}\|_{L^{2}}^{2}\leq& f(t)\int_{E(t)}|\widehat{\theta_{0}}
(\xi)|^{2}e^{-2(|\xi_{1}|^{2\alpha}+|\xi_{2}|^{2\beta})t}\, d\xi\nonumber\\&+Ct^{2}\left(f(t)^{1+\frac{\alpha+3\beta}{2\alpha\beta}}+f(t)^{1+\frac{3\alpha+\beta}
{2\alpha\beta}}\right)\nonumber\\ \leq& f(t)\int_{E(t)}|\widehat{\theta_{0}}
(\xi)|^{2}\, d\xi\nonumber\\&+Ct^{2}\left(f(t)^{1+\frac{\alpha+3\beta}{2\alpha\beta}}+f(t)^{1+\frac{3\alpha+\beta}
{2\alpha\beta}}\right).
\end{align}
Thanks to the definition of $E(t)$ and $\theta_{0}\in L^{2}$, it is shown that for any $\epsilon>0$, there exists a $t_{\epsilon}>0$ such that
$$\int_{E(t)}|\widehat{\theta _{0}}(\xi)|^{2}\, d\xi\leq \epsilon,\ \ \forall t\geq t_{\epsilon}.$$
It follows from \eqref{mkttye15} that for any $t\geq t_{\epsilon}$
\begin{align}\label{mkttye16}
 \frac{d}{dt}\| {\theta}(t)\|_{L^{2}}^{2}+f(t)\| {\theta}\|_{L^{2}}^{2} \leq \epsilon f(t)
+C(1+t)^{2}\left(f(t)^{1+\frac{\alpha+3\beta}{2\alpha\beta}}+f(t)^{1+\frac{3\alpha+\beta}
{2\alpha\beta}}\right).
\end{align}
Multiplying $e^{\int_{t_{\epsilon}}^{t}f(r)\, d r}$ both side of (\ref{mkttye16}) and integrating in time over $[t_{\epsilon},t]$ imply
\begin{align}\label{split2}
e^{ \int_{t_{\epsilon}}^{t}f(r)\, d r}\|\widehat{\theta}(t)\|_{L^{2}}^{2}
&\leq \|\widehat{\theta}(t_{\epsilon})\|_{L^{2}}^{2}
+\epsilon\int_{t_{\epsilon}}^{t}e^{\int_{t_{\epsilon}}^{\lambda}
f(r)\, d r} f(\lambda)\,d\lambda\nonumber\\&\quad+C\int_{t_{\epsilon}}^{t}e^{\int_{t_{\epsilon}}^{\lambda}
f(r)\, d r}  (1+\lambda)^{2}\left(f(\lambda)^{1+\frac{\alpha+3\beta}{2\alpha\beta}}+f(\lambda)^{1+\frac{3\alpha+\beta}
{2\alpha\beta}}\right)\,d\lambda.
\end{align}
Taking $g(t)=\frac{k}{(e+t)\ln(e+t)}$ with some $k>\max\{\frac{\alpha+3\beta}{2\alpha\beta},\,\frac{3\alpha+\beta}{2\alpha\beta}\}$, we deduce from (\ref{split2}) that
\begin{align}
 \|\widehat{\theta}(t)\|_{L^{2}}^{2}\frac{\ln^{k}(e+t)}{\ln^{k}(e+t_{\epsilon})}
 &\leq \|\widehat{\theta}(t_{\epsilon})\|_{L^{2}}^{2}+\frac{k\epsilon}{\ln^{k}
 (e+t_{\epsilon})}\int_{t_{\epsilon}}^{t}\frac{\ln^{k-1}(e+\lambda)}
 {e+\lambda}\,d\lambda\nonumber\\&\quad+\frac{C}{\ln^{k}(e+t_{\epsilon})}
 \int_{t_{\epsilon}}^{t}\frac{\ln^{k-1-\frac{\alpha+3\beta}{2\alpha\beta}} (e+\lambda)}
 {(e+\lambda)^{\frac{\alpha+3\beta}{2\alpha\beta}-1}}\, d \lambda
 \nonumber\\&\quad+\frac{C}{\ln^{k}(e+t_{\epsilon})}
 \int_{t_{\epsilon}}^{t}\frac{\ln^{k-1-\frac{3\alpha+\beta}{2\alpha\beta}} (e+\lambda)}
 {(e+\lambda)^{\frac{3\alpha+\beta}{2\alpha\beta}-1}}\, d \lambda
 \nonumber\\
  &\leq \|{\theta}(t_{\epsilon})\|_{L^{2}}^{2}+\frac{k\epsilon}{\ln^{k}
 (e+t_{\epsilon})}\int_{t_{\epsilon}}^{t}\frac{\ln^{k-1}(e+\lambda)}
 {e+\lambda}\,d\lambda\nonumber\\&\quad+\frac{C}{\ln^{k}(e+t_{\epsilon})}
 \int_{t_{\epsilon}}^{t}\frac{\ln^{k-1-\frac{\alpha+3\beta}{2\alpha\beta}} (e+\lambda)}
 {e+\lambda}\, d \lambda\nonumber\\&\quad+\frac{C}{\ln^{k}(e+t_{\epsilon})}
 \int_{t_{\epsilon}}^{t}\frac{\ln^{k-1-\frac{3\alpha+\beta}{2\alpha\beta}} (e+\lambda)}
 {e+\lambda}\, d \lambda
 \nonumber\\
&\leq\|{ \theta}_{0}\|_{L^{2}}^{2}+\epsilon\frac{\ln^{k}(e+t)-\ln^{k}
(e+t_{\epsilon})}{\ln^{k}(e+t_{\epsilon})}
\nonumber\\&\quad+C\frac{\ln^{k-\frac{\alpha+3\beta}{2\alpha\beta}}(e+t)-
\ln^{k-\frac{\alpha+3\beta}{2\alpha\beta}}
(e+t_{\epsilon})}{\ln^{k}
(e+t_{\epsilon})}\nonumber\\&\quad+C\frac{\ln^{k-\frac{3\alpha+\beta}{2\alpha\beta}}(e+t)-
\ln^{k-\frac{3\alpha+\beta}{2\alpha\beta}}
(e+t_{\epsilon})}{\ln^{k}
(e+t_{\epsilon})},
\nonumber\\
&\leq\|{\theta}_{0}\|_{L^{2}}^{2}+\epsilon\frac{\ln^{k}(e+t)}{\ln^{k}(e+t_{\epsilon})}
+C\frac{\ln^{k-\frac{\alpha+3\beta}{2\alpha\beta}}(e+t)}{\ln^{k}(e+t_{\epsilon})}
+C\frac{\ln^{k-\frac{3\alpha+\beta}{2\alpha\beta}}(e+t)}{\ln^{k}(e+t_{\epsilon})},\nonumber
\end{align}
where we have used $\frac{\alpha+3\beta}{2\alpha\beta}-1\geq1$ and $\frac{3\alpha+\beta}{2\alpha\beta}-1\geq1$ due to $\min\{\alpha+3\beta,\,3\alpha+\beta\}\geq4\alpha\beta$.
This yields
\begin{align}\label{sdfghhyu1}
 \|\widehat{\theta}(t)\|_{L^{2}}^{2}
 &\leq \frac{\ln^{k}(e+t_{\epsilon})}{\ln^{k}(e+t)} \|{\theta}_{0}\|_{L^{2}}^{2}
 +\epsilon+\frac{C}{\ln^{\frac{\alpha+3\beta}{2\alpha\beta}}(e+t)}
 +\frac{C}{\ln^{\frac{3\alpha+\beta}{2\alpha\beta}}(e+t)}.
\end{align}
Based on \eqref{sdfghhyu1}, with regard to the above $\epsilon>0$, there exists a $T_{\epsilon}>t_{\epsilon}>0$ such that
$$\frac{\ln^{k}(e+t_{\epsilon})}{\ln^{k}(e+t)} \|{\theta}_{0}\|_{L^{2}}^{2}
 +\frac{C}{\ln^{\frac{\alpha+3\beta}{2\alpha\beta}}(e+t)}
 +\frac{C}{\ln^{\frac{3\alpha+\beta}{2\alpha\beta}}(e+t)}\leq  \epsilon,\ \ \forall t\geq T_{\epsilon}.$$
As a result, we show that for any $\epsilon>0$, there exists a $T_{\epsilon}>0$ such that
 $$\|\widehat{\theta}(t)\|_{L^{2}}^{2}\leq 2\epsilon,\ \ \forall t\geq T_{\epsilon}.$$
This implies
 $$\lim_{t\rightarrow\infty}\|\widehat{\theta}(t)\|_{L^{2}}=0.$$
Thus, one obtains
\begin{align}
 \lim_{t\rightarrow\infty}\|\theta(t)\|_{L^{2}}=0.\nonumber
\end{align}
Therefore, this ends the proof of Proposition \ref{mkjqq1}.
\end{proof}

\vskip .1in
Next, we derive the  preliminary decay rate for $\nabla \theta$, which is essential to obtain the sharp decay rate for itself.
\begin{Pros}\label{fggytq11} If $\alpha$ and $\beta$ satisfy (\ref{sdf2334}) and $\theta_{0}\in \dot{H}^{1}(\mathbb{R}^{2})\cap L^p(\mathbb{R}^{2})$ with $p\in [1,2)$, then the solution of (\ref{SQG}) admits the decay estimate
\begin{align}\label{pre310}
 &\|\nabla\theta(t)\|_{L^2}\leq C_{0}(1+t)^{-\frac{(\alpha+\beta)(2-p) }{4\alpha\beta p}},
\end{align}
where the constant $C_{0}>0$ depends only on $\alpha,\,\beta,\,p$ and $\theta_{0}$.
\end{Pros}

\begin{proof}
It follows from \eqref{t3326t009} and \eqref{ghtyde239} that
\begin{eqnarray}
 \frac{d}{dt}\|\nabla \theta(t)\|_{L^{2}}^{2}+
 \|\Lambda_{x_{1}}^{\alpha}\nabla\theta\|_{L^{2}}^{2}+
 \|\Lambda_{x_{2}}^{\beta}\nabla\theta\|_{L^{2}}^{2} \leq G(t)\| \theta\|_{L^{2}}^{2},\nonumber
\end{eqnarray}
where $G(t)$ is given by
 $$G(t)=\max\{F(t),\, D(t)\}.$$
By means of Proposition \ref{Spr1}, one derives
$$\sup_{t\geq0}G(t)\leq C$$
for some absolute constant $C$ depending only on $\alpha,\,\beta$ and $\theta_{0}$. This further shows
\begin{align}\label{pre311}
 \frac{d}{dt}\|\nabla \theta(t)\|_{L^{2}}^{2}+
 \|\Lambda_{x_{1}}^{\alpha}\nabla\theta\|_{L^{2}}^{2}+
 \|\Lambda_{x_{2}}^{\beta}\nabla\theta\|_{L^{2}}^{2} \leq C\| \theta\|_{L^{2}}^{2}.
\end{align}
Making use of the following interpolation inequalities (see \eqref{avbnuy8911})
$$\|\Lambda_{x_{1}} \theta\|_{L^{2}}\leq C\|\theta\|_{L^{2}}^{\frac{\alpha}{1+\alpha}}
\|\Lambda_{x_{1}}^{1+\alpha}\theta\|_{L^{2}}^{\frac{1}{1+\alpha}},\qquad \|\Lambda_{x_{2}}\theta\|_{L^{2}}\leq C\|\theta\|_{L^{2}}^{\frac{\beta}{1+\beta}}
\|\Lambda_{x_{2}}^{1+\beta}\theta\|_{L^{2}}^{\frac{1}{1+\beta}},$$
we have that
\begin{align}
 \|\nabla {\theta}\|_{L^{2}}^{2}&=\|\Lambda_{x_{1}}\theta\|_{L^{2}}^{2}+\|\Lambda_{x_{2}}\theta\|_{L^{2}}^{2}
 \nonumber\\& \leq C\|\theta\|_{L^{2}}^{\frac{2\alpha}{1+\alpha}}
\|\Lambda_{x_{1}}^{1+\alpha}\theta\|_{L^{2}}^{\frac{2}{1+\alpha}}
+C\|\theta\|_{L^{2}}^{\frac{2\beta}{1+\beta}}
\|\Lambda_{x_{2}}^{1+\beta}\theta\|_{L^{2}}^{\frac{2}{1+\beta}}\nonumber\\ &\leq
 \|\Lambda_{x_{1}}^{\alpha}\nabla {\theta}\|_{L^{2}}^{2}+
 \|\Lambda_{x_{2}}^{\beta}\nabla{\theta}\|_{L^{2}}^{2}+ {C}\| {\theta}\|_{L^{2}}^{2},\nonumber
\end{align}
which along with \eqref{pre311} and \eqref{Proshsh78} yield
\begin{align}\label{pre312}
\frac{d}{dt}\|\nabla{\theta}(t)\|_{L^{2}}^{2}
+ \|\nabla{\theta}\|_{L^{2}}^{2} \leq C\|\theta\|_{L^{2}}^{2} \leq C(1+t)^{-\frac{(\alpha+\beta)(2-p)}{2\alpha\beta p}}.
\end{align}
We deduce from \eqref{pre312} that
\begin{align}
\frac{d}{dt}(e^{t}\|\nabla{\theta}(t)\|_{L^{2}}^{2})  \leq Ce^{t}(1+t)^{-\frac{(\alpha+\beta)(2-p)}{2\alpha\beta p}},\nonumber
\end{align}
which further yields
\begin{align}
e^{t}\|\nabla{\theta}(t)\|_{L^{2}}^{2}  \leq \|\nabla {\theta}_{0}\|_{L^{2}}^{2}+C\int_{0}^{t}e^{\tau}(1+\tau)^{-\frac{(\alpha+\beta)(2-p)}{2\alpha\beta p}}\,d\tau.\nonumber
\end{align}
As a result, one has
\begin{align}\label{pre313}
\|\nabla{\theta}(t)\|_{L^{2}}^{2}  \leq e^{-t}\|\nabla {\theta}_{0}\|_{L^{2}}^{2}+C\int_{0}^{t}e^{-(t-\tau)}(1+\tau)^{-\frac{(\alpha+\beta)(2-p)}{2\alpha\beta p}}\,d\tau.
\end{align}
We now appeal to the following simple fact
\begin{align}\label{pre314}
\int_{0}^{t}e^{-\vartheta(t-\tau)}(1+\tau)^{-\varrho}\,d\tau\leq C(\vartheta,\varrho)(1+t)^{-\varrho},\quad \vartheta>0,\,\varrho>0,\end{align}
which is due to
\begin{align}
\int_{0}^{t}e^{-\vartheta(t-\tau)}(1+\tau)^{-\varrho}\,d\tau&=\int_{0}^{\frac{t}{2}}
e^{-\vartheta(t-\tau)}
(1+\tau)^{-\varrho}\,d\tau+
\int_{\frac{t}{2}}^{t}e^{-\vartheta(t-\tau)}(1+\tau)^{-\varrho}\,d\tau
\nonumber\\
&\leq e^{-\frac{\vartheta t}{2}}\int_{0}^{\frac{t}{2}}(1+\tau)^{-\varrho}
\,d\tau+\Big(1+\frac{t}{2}\Big)^{-\varrho}
\int_{\frac{t}{2}}^{t}e^{-\vartheta(t-\tau)}\,d\tau\nonumber\\
&\leq e^{-\frac{\vartheta t}{2}}\frac{t}{2}+\Big(1+\frac{t}{2}\Big)^{-\varrho}
\int_{0}^{\frac{t}{2}}e^{-\vartheta s}\,d s
\nonumber\\
&\leq
C(\vartheta)(1+t)^{-\varrho}+ \frac{C(\varrho)}{\vartheta}(1+t)^{-\varrho}\nonumber\\
&\leq C(\vartheta,\varrho)(1+t)^{-\varrho}.\nonumber
\end{align}
Invoking \eqref{pre314}, we deduce from \eqref{pre313} that
\begin{align*}
 &\|\nabla\theta(t)\|_{L^2}\leq C(1+t)^{-\frac{(\alpha+\beta)(2-p) }{4\alpha\beta p}}.
\end{align*}
Therefore, this completes the proof of Proposition \ref{fggytq11}.
\end{proof}

\vskip .1in
We now proceed to improve the decay estimate \eqref{pre310}, which
turns out to be the key to obtain all the higher order decay estimates.
\begin{Pros}\label{fggyvby7}
If $\alpha$ and $\beta$ satisfy (\ref{sdf2334}) and $\theta_{0}\in \dot{H}^{1}(\mathbb{R}^{2})\cap L^p(\mathbb{R}^{2})$ with $p\in [1,2)$, then the solution admits the decay estimate
\begin{align}\label{pre1y315}
 &\|\nabla\theta(t)\|_{L^2}\leq C_{0}(1+t)^{-\frac{(\alpha+\beta)(2-p)+2\min\{\alpha,\,\beta\}p}{4\alpha\beta p}},
\end{align}
where the constant $C_{0}>0$ depends only on $\alpha,\,\beta,\,p$ and $\theta_{0}$.
\end{Pros}
\begin{proof}
We first claim that \eqref{pre1y315} holds true for the case: $\beta> \frac{1}{2\alpha+1}, \,0<\alpha\leq\frac{1}{2}$. In this case, we deduce  from \eqref{udfr5}, \eqref{hhh3e456} and \eqref{t3nhy71qw}, respectively
\begin{align}\label{pre1y316}
\mathcal{J}_{1}, \ \mathcal{J}_{2}&\leq C\|\Lambda^{\frac{2}{q}}\Lambda_{x_{2}}^{\delta}\partial_{x_{1}}\theta\|_{L^{2}} \|\theta\|_{L^{q}} \|\Lambda_{x_{2}}^{\beta} \nabla\theta\|_{L^{2}}^{\frac{\alpha}{\alpha+1}}\|\Lambda_{x_{1}}^{\alpha} \nabla\theta\|_{L^{2}}^{\frac{1}{\alpha+1}},
\end{align}
\begin{align}\label{pre1y317}
  \mathcal{J}_{3}\leq & C\|\Lambda^{\frac{2}{\widetilde{q}}}\Lambda_{x_{2}}^{\delta}\partial_{x_{2}}
\theta\|_{L^{2}}
 \|\theta\|_{L^{\widetilde{q}}} \|\Lambda_{x_{2}}^{\beta} \nabla\theta\|_{L^{2}}^{\frac{\alpha}{\alpha+1}}\|\Lambda_{x_{1}}^{\alpha} \nabla\theta\|_{L^{2}}^{\frac{1}{\alpha+1}}\nonumber\\&+C\|\Lambda^{\frac{2}{q}}\Lambda_{x_{2}}^{\delta}\partial_{x_{1}}\theta\|_{L^{2}} \|\theta\|_{L^{q}} \|\Lambda_{x_{2}}^{\beta} \nabla\theta\|_{L^{2}}^{\frac{\alpha}{\alpha+1}}\|\Lambda_{x_{1}}^{\alpha} \nabla\theta\|_{L^{2}}^{\frac{1}{\alpha+1}},
\end{align}
\begin{align}\label{pre1y318}
  \mathcal{J}_{4}
 \leq&C\|\Lambda^{\frac{2}{\widehat{q}}}\Lambda_{x_{2}}^{1-\beta}\partial_{x_{2}}\theta
 \|_{L^{2}}   \|\theta\|_{L^{\widehat{q}}} \|\Lambda_{x_{2}}^{\beta}\nabla\theta\|_{L^{2}},
\end{align}
for any $q,\widetilde{q},\widehat{q}\in (2,\,\infty)$ and $\delta=\frac{1-\alpha\beta}{\alpha +1}$. Plugging \eqref{pre1y316}, \eqref{pre1y317} and \eqref{pre1y318} into \eqref{t3326t002} gives
\begin{align}\label{pre1y319}
&\frac{1}{2}\frac{d}{dt}\|\nabla \theta(t)\|_{L^{2}}^{2}+
 \|\Lambda_{x_{1}}^{\alpha}\nabla\theta\|_{L^{2}}^{2}+
 \|\Lambda_{x_{2}}^{\beta}\nabla\theta\|_{L^{2}}^{2}\nonumber\\& \leq
 C\|\Lambda^{\frac{2}{q}}\Lambda_{x_{2}}^{\delta}\partial_{x_{1}}\theta\|_{L^{2}} \|\theta\|_{L^{q}} \|\Lambda_{x_{2}}^{\beta} \nabla\theta\|_{L^{2}}^{\frac{\alpha}{\alpha+1}}\|\Lambda_{x_{1}}^{\alpha} \nabla\theta\|_{L^{2}}^{\frac{1}{\alpha+1}}\nonumber\\&\quad +C\|\Lambda^{\frac{2}{\widetilde{q}}}\Lambda_{x_{2}}^{\delta}\partial_{x_{2}}
\theta\|_{L^{2}}
 \|\theta\|_{L^{\widetilde{q}}} \|\Lambda_{x_{2}}^{\beta} \nabla\theta\|_{L^{2}}^{\frac{\alpha}{\alpha+1}}\|\Lambda_{x_{1}}^{\alpha} \nabla\theta\|_{L^{2}}^{\frac{1}{\alpha+1}}
 \nonumber\\&\quad +
 C\|\Lambda^{\frac{2}{\widehat{q}}}\Lambda_{x_{2}}^{1-\beta}\partial_{x_{2}}\theta
 \|_{L^{2}}   \|\theta\|_{L^{\widehat{q}}} \|\Lambda_{x_{2}}^{\beta}\nabla\theta\|_{L^{2}}
 \nonumber\\& \leq
 C\|\Lambda_{x_{1}}^{\frac{2}{q}}\Lambda_{x_{2}}^{\delta}\nabla\theta\|_{L^{2}} \|\theta\|_{L^{q}} \|\Lambda_{x_{2}}^{\beta} \nabla\theta\|_{L^{2}}^{\frac{\alpha}{\alpha+1}}\|\Lambda_{x_{1}}^{\alpha} \nabla\theta\|_{L^{2}}^{\frac{1}{\alpha+1}}\nonumber\\&\quad +C\|\Lambda_{x_{2}}^{\frac{2}{\widetilde{q}}}\Lambda_{x_{2}}^{\delta}\nabla
\theta\|_{L^{2}}
 \|\theta\|_{L^{\widetilde{q}}} \|\Lambda_{x_{2}}^{\beta} \nabla\theta\|_{L^{2}}^{\frac{\alpha}{\alpha+1}}\|\Lambda_{x_{1}}^{\alpha} \nabla\theta\|_{L^{2}}^{\frac{1}{\alpha+1}}
 \nonumber\\&\quad +
 C\|\Lambda_{x_{2}}^{\frac{2}{\widehat{q}}}\Lambda_{x_{2}}^{1-\beta}\nabla\theta
 \|_{L^{2}}   \|\theta\|_{L^{\widehat{q}}} \|\Lambda_{x_{2}}^{\beta}\nabla\theta\|_{L^{2}}
 \nonumber\\& \leq C
 \|\nabla\theta\|_{L^{2}}^{1-\frac{\delta}{\beta}-\frac{\frac{2}{q}}{\alpha}} \|\Lambda_{x_{1}}^{\alpha}\nabla\theta\|_{L^{2}}^{\frac{\frac{2}{q}}{\alpha}}
 \|\Lambda_{x_{2}}^{\beta}\nabla\theta\|_{L^{2}}^{\frac{\delta}{\beta}}
 \|\theta\|_{L^{q}} \|\Lambda_{x_{2}}^{\beta} \nabla\theta\|_{L^{2}}^{\frac{\alpha}{\alpha+1}}\|\Lambda_{x_{1}}^{\alpha} \nabla\theta\|_{L^{2}}^{\frac{1}{\alpha+1}}\nonumber\\&\quad +C
 \|\nabla\theta\|_{L^{2}} ^{1-\frac{\delta+\frac{2}{\widetilde{q}}}{\beta}}
 \|\Lambda_{x_{2}}^{\beta} \nabla\theta\|_{L^{2}}^{\frac{\delta+\frac{2}{\widetilde{q}}}{\beta}}
 \|\theta\|_{L^{\widetilde{q}}} \|\Lambda_{x_{2}}^{\beta} \nabla\theta\|_{L^{2}}^{\frac{\alpha}{\alpha+1}}\|\Lambda_{x_{1}}^{\alpha} \nabla\theta\|_{L^{2}}^{\frac{1}{\alpha+1}}
 \nonumber\\&\quad +
 C\|\nabla\theta\|_{L^{2}} ^{1-\frac{1-\beta+\frac{2}{\widehat{q}}}{\beta}}
 \|\Lambda_{x_{2}}^{\beta} \nabla\theta\|_{L^{2}}^{\frac{1-\beta+\frac{2}{\widehat{q}}}{\beta}}
 \|\theta\|_{L^{\widehat{q}}} \|\Lambda_{x_{2}}^{\beta}\nabla\theta\|_{L^{2}}
  \nonumber\\& \leq
 \frac{1}{2}\|\Lambda_{x_{1}}^{\alpha}\nabla\theta\|_{L^{2}}^{2}+ \frac{1}{2}
 \|\Lambda_{x_{2}}^{\beta}\nabla\theta\|_{L^{2}}^{2}+C \|\theta\|_{L^{q}}^{\frac{2\alpha(\alpha+1)\beta q}{[(2\alpha+1)\beta-1]\alpha q-2(\alpha+1)\beta}} \|\nabla\theta\|_{L^{2}}^{2}\nonumber\\&\quad
 +C \|\theta\|_{L^{\widetilde{q}}}^{\frac{2(\alpha+1)\beta\widetilde{q}}
 {[(2\alpha+1)\beta-1]\widetilde{q}-2(\alpha+1)}} \|\nabla\theta\|_{L^{2}}^{2}
 +C \|\theta\|_{L^{q}}^{\frac{2\beta \widehat{q}}{(2\beta-1)\widehat{q}-2}} \|\nabla\theta\|_{L^{2}}^{2},
\end{align}
where we have used \eqref{avbnuy8911} and \eqref{t2326001}.
It follows from \eqref{pre1y319} that
\begin{align}\label{pre1y320}
  &\frac{d}{dt}\|\nabla \theta(t)\|_{L^{2}}^{2}+
 \|\Lambda_{x_{1}}^{\alpha}\nabla\theta\|_{L^{2}}^{2}+
 \|\Lambda_{x_{2}}^{\beta}\nabla\theta\|_{L^{2}}^{2}\nonumber\\ & \leq
 C \|\theta\|_{L^{q}}^{\frac{2\alpha(\alpha+1)\beta q}{[(2\alpha+1)\beta-1]\alpha q-2(\alpha+1)\beta}} \|\nabla\theta\|_{L^{2}}^{2}
 +C \|\theta\|_{L^{\widetilde{q}}}^{\frac{2(\alpha+1)\beta\widetilde{q}}
 {[(2\alpha+1)\beta-1]\widetilde{q}-2(\alpha+1)}} \|\nabla\theta\|_{L^{2}}^{2}
 \nonumber\\&\quad+C \|\theta\|_{L^{q}}^{\frac{2\beta \widehat{q}}{(2\beta-1)\widehat{q}-2}} \|\nabla\theta\|_{L^{2}}^{2}.
\end{align}
We deduce from \eqref{Proshsh78} and \eqref{pre310} that for any $k\in (2,\infty)$
\begin{align}\label{prbbhyu8}\|\theta\|_{L^{k}}\leq C\|\theta\|_{L^{2}}^{\frac{2}{ {k}}}\|\nabla\theta\|_{L^{2}}
^{\frac{k-2}{ {k}}}\leq C(1+t)^{-\frac{(\alpha+\beta)(2-p) }{4\alpha\beta p}}.
\end{align}
Combining this and \eqref{pre1y320} gives
\begin{align}\label{pre1y321}
 \frac{d}{dt}\|\nabla{\theta}(t)\|_{L^{2}}^{2}
+\|\Lambda_{x_{1}}^{\alpha}\nabla{\theta}\|_{L^{2}}^{2}+
 \|\Lambda_{x_{2}}^{\beta}\nabla{\theta}\|_{L^{2}}^{2}\leq C(1+t)^{-\frac{(\alpha+\beta)(2-p)\kappa }{4\alpha\beta p}}\|\nabla {\theta}\|_{L^{2}}^{2},
\end{align}
where $\kappa=\min\left\{\Gamma_{1},\,\Gamma_{2},\, \Gamma_{3}\right\}$
with
 $$\Gamma_{1}=\frac{2\alpha(\alpha+1)\beta q}{[(2\alpha+1)\beta-1]\alpha q-2(\alpha+1)\beta},$$
 $$ \Gamma_{2}=\frac{2(\alpha+1)\beta\widetilde{q}}
 {[(2\alpha+1)\beta-1]\widetilde{q}-2(\alpha+1)},\qquad \Gamma_{3}=\frac{2\beta \widehat{q}}{(2\beta-1)\widehat{q}-2}.$$
Due to the arbitrariness of $q,\widetilde{q},\widehat{q}\in (2,\,\infty)$, one can take some suitable $q,\widetilde{q},\widehat{q}$ such that
\begin{align}\label{pre1y322}
\frac{(\alpha+\beta)(2-p)\kappa}{4\alpha\beta p}>1.
\end{align}
Now we denote
$$E(t)=\{\xi\in \mathbb{R}^2: \  |\xi_{1}|^{2\alpha}+|\xi_{2}|^{2\beta}\leq f(t)\}$$
for a smooth decreasing function $f(t)$ to be determined hereafter.
It thus yields
\begin{align}\label{pre1y323}
\|\Lambda_{x_{1}}^{\alpha}\nabla\theta\|_{L^{2}}^{2}+
 \|\Lambda_{x_{2}}^{\beta}\nabla\theta\|_{L^{2}}^{2}
 &=\int_{\mathbb{R}^{2}}|\xi_{1}|^{2\alpha}|\widehat{\nabla\theta}(t,\xi)|^{2}\,d\xi +\int_{\mathbb{R}^{2}}|\xi_{2}|^{2\beta}|\widehat{\nabla\theta}(t,\xi)|^{2}\,d\xi\nonumber\\
 &=\int_{\mathbb{R}^{2}}(|\xi_{1}|^{2\alpha}+|\xi_{2}|^{2\beta})|\widehat{\nabla\theta}(t,\xi)|^{2}\,d\xi
 \nonumber\\
 &\geq \int_{E(t)^{c}}(|\xi_{1}|^{2\alpha}+|\xi_{2}|^{2\beta})|\widehat{\nabla\theta}(t,\xi)|^{2}\,d\xi
 \nonumber\\
 &\geq f(t)\int_{E(t)^{c}}|\widehat{\nabla\theta}(t,\xi)|^{2}\,d\xi\nonumber\\
 &= f(t)\int_{\mathbb{R}^{2}}|\widehat{\nabla\theta}(t,\xi)|^{2}\,d\xi
 -f(t)\int_{E(t)}|\widehat{\nabla\theta}(t,\xi)|^{2}\,d\xi,
\end{align}
which together with \eqref{pre1y321} gives
\begin{align}\label{pre1y324}
\frac{d}{dt}\|\nabla{\theta}(t)\|_{L^{2}}^{2}+f(t)\|\nabla {\theta}\|_{L^{2}}^{2}\leq& f(t)\int_{E(t)}|\widehat{\nabla {\theta}}(t,\xi)|^{2}\,d\xi\nonumber\\&+C(1+t)^{-\frac{(\alpha+\beta)(2-p)\kappa }{4\alpha\beta p}}  \|\nabla{\theta}\|_{L^{2}}^{2}.
\end{align}
By direct computations, we may check that
\begin{align}\label{pre1y325}
\int_{E(t)}|\widehat{\nabla {\theta}}(t,\xi)|^{2}\,d\xi
&\leq
\int_{E(t)}
|\xi|^{2}|\widehat{\theta}(\xi)|^{2}\,d\xi\nonumber\\
&\leq
\int_{E(t)}
(|\xi_{1}|^{2}+|\xi_{2}|^{2}) |\widehat{\theta}(\xi)|^{2}\,d\xi\nonumber\\
&\leq C\int_{E(t)}
\left(f(t)^{\frac{1}{\alpha}}+f(t)^{\frac{1}{\beta}}\right)|\widehat{\theta}(\xi)|^{2}\,d\xi
\nonumber\\
&\leq C\left(f(t)^{\frac{1}{\alpha}}+f(t)^{\frac{1}{\beta}}\right)\|{\theta}\|_{L^{2}}^{2}.
\end{align}
Putting \eqref{pre1y325} into \eqref{pre1y324} implies
\begin{align}\label{pre1y326}
\frac{d}{dt}\|\nabla{\theta}(t)\|_{L^{2}}^{2}+f(t)\|\nabla {\theta}\|_{L^{2}}^{2}\leq& C\left(f(t)^{1+\frac{1}{\alpha}}+f(t)^{1+\frac{1}{\beta}}\right)\|{\theta}\|_{L^{2}}^{2}\nonumber\\&+C(1+t)^{-\frac{(\alpha+\beta)(2-p)\kappa }{4\alpha\beta p}}  \|\nabla{\theta}\|_{L^{2}}^{2}\nonumber\\ \leq& C\left(f(t)^{1+\frac{1}{\alpha}}+f(t)^{1+\frac{1}{\beta}}\right)(1+t)^{-\frac{(\alpha+\beta)(2-p) }{2\alpha\beta p}}\nonumber\\&+C(1+t)^{-\frac{(\alpha+\beta)(2-p)\kappa }{4\alpha\beta p}}  \|\nabla{\theta}\|_{L^{2}}^{2}.
\end{align}
Taking $f(t)=\frac{\varrho}{1+t}$ with $\varrho$ suitable large, we deduce from \eqref{pre1y326} that
\begin{align}\label{pre1y327}
\frac{d}{dt}\|\nabla{\theta}(t)\|_{L^{2}}^{2}+\frac{\varrho}{1+t}\|\nabla {\theta}\|_{L^{2}}^{2}\leq& C(1+t)^{-1-\min\{\frac{1}{\alpha},\,\frac{1}{\beta}\}-\frac{(\alpha+\beta)(2-p) }{2\alpha\beta p}}\nonumber\\&+\widetilde{C}(1+t)^{-\frac{(\alpha+\beta)(2-p)\kappa }{4\alpha\beta p}}  \|\nabla{\theta}\|_{L^{2}}^{2}.
\end{align}
Recalling \eqref{pre1y322}, we deduce from \eqref{pre1y327} that
\begin{align}
\frac{d}{dt}\|\nabla{\theta}(t)\|_{L^{2}}^{2}+\frac{\varrho-\widetilde{C}}{1+t}
\|\nabla {\theta}\|_{L^{2}}^{2}\leq& C(1+t)^{-1-\frac{\min\{\alpha,\,\beta\}}{\alpha\beta}-\frac{(\alpha+\beta)(2-p) }{2\alpha\beta p}}.\nonumber
\end{align}
As a result, one derives
\begin{align}\label{pre1y328}
\frac{d}{dt}\left((1+t)^{\varrho-\widetilde{C}}\|\nabla{\theta}(t)\|_{L^{2}}^{2}\right)\leq& C(1+t)^{\varrho-\widetilde{C}-1-\frac{\min\{\alpha,\,\beta\}}{\alpha\beta}-\frac{(\alpha+\beta)(2-p) }{2\alpha\beta p}}.
\end{align}
Integrating \eqref{pre1y328} in time, we obtain
\begin{align}
(1+t)^{\varrho-\widetilde{C}}\|\nabla{\theta}(t)\|_{L^{2}}^{2}\leq& \|\nabla{\theta_{0}}\|_{L^{2}}^{2}+ C(1+t)^{\varrho-\widetilde{C}-\frac{\min\{\alpha,\,\beta\} }{\alpha\beta} -\frac{(\alpha+\beta)(2-p) }{2\alpha\beta p}},\nonumber
\end{align}
which implies
\begin{align}
\|\nabla{\theta}(t)\|_{L^{2}}^{2}\leq&  C(1+t)^{-\frac{\min\{\alpha,\,\beta\} }{\alpha\beta}-\frac{(\alpha+\beta)(2-p)}{2\alpha\beta p}}.\nonumber
\end{align}

Next we claim that \eqref{pre1y315} still holds true for the case: $\beta>\frac{1-\alpha}{2\alpha}, \,\frac{1}{2}<\alpha<1$.
We get from \eqref{ghtyde229}, \eqref{ghtyde232} and  \eqref{ghtyde235}, respectively
\begin{align} \label{pre1y329}
\mathcal{J}_{1}
 \leq&C\|\Lambda^{\frac{2}{r}}\Lambda_{x_{1}}^{1-\alpha}\partial_{x_{1}}
\theta\|_{L^{2}} \|\theta\|_{L^{r}} \|\Lambda_{x_{1}}^{\alpha}\partial_{x_{1}}\theta\|_{L^{2}},
\end{align}
\begin{align} \label{pre1y330}
\mathcal{J}_{2}
 \leq&C\|\Lambda^{\frac{2}{\widetilde{r}}}
 \Lambda_{x_{1}}^{1-\widetilde{\delta}}\partial_{x_{1}}
 \theta\|_{L^{2}}  \|\theta\|_{L^{\widetilde{r}}}\|\Lambda_{x_{1}}^{\alpha+1}\theta
\|_{L^{2}}^{\frac{\beta}{\beta+1}} \|\Lambda_{x_{2}}^{\beta+1}\theta
\|_{L^{2}}^{\frac{1}{\beta+1}}\nonumber\\&+C\|\Lambda^{\frac{2}{r}}\Lambda_{x_{1}}^{1-\alpha}\partial_{x_{1}}
\theta\|_{L^{2}} \|\theta\|_{L^{r}} \|\Lambda_{x_{1}}^{\alpha}\partial_{x_{1}}\theta\|_{L^{2}},
\end{align}
\begin{align} \label{pre1y331}
\mathcal{J}_{3},\,\mathcal{J}_{4}
\leq& C\|\Lambda^{\frac{2}{\widehat{r}}}
 \Lambda_{x_{1}}^{1-\widetilde{\delta}}\partial_{x_{2}}
 \theta\|_{L^{2}}  \|\theta\|_{L^{\widehat{r}}}\|\Lambda_{x_{1}}^{\alpha+1}\theta
\|_{L^{2}}^{\frac{\beta}{\beta+1}} \|\Lambda_{x_{2}}^{\beta+1}\theta
\|_{L^{2}}^{\frac{1}{\beta+1}}
\end{align}
for any $r,\widetilde{r},\widehat{r}\in (2,\,\infty)$ and $\widetilde{\delta}=\frac{(\alpha+1)\beta}{\beta+1}$.
Putting \eqref{pre1y329}, \eqref{pre1y330} and \eqref{pre1y331} into \eqref{t36t0ghd} yields
\begin{align}
&\frac{1}{2}\frac{d}{dt}\|\nabla \theta(t)\|_{L^{2}}^{2}+
 \|\Lambda_{x_{1}}^{\alpha}\nabla\theta\|_{L^{2}}^{2}+
 \|\Lambda_{x_{2}}^{\beta}\nabla\theta\|_{L^{2}}^{2}\nonumber\\ &\leq C\|\Lambda^{\frac{2}{r}}\Lambda_{x_{1}}^{1-\alpha}\partial_{x_{1}}
\theta\|_{L^{2}} \|\theta\|_{L^{r}} \|\Lambda_{x_{1}}^{\alpha+1} \theta\|_{L^{2}}\nonumber\\&\quad +
C\|\Lambda^{\frac{2}{\widetilde{r}}}
 \Lambda_{x_{1}}^{1-\widetilde{\delta}}\partial_{x_{1}}
 \theta\|_{L^{2}}  \|\theta\|_{L^{\widetilde{r}}}\|\Lambda_{x_{1}}^{\alpha+1}\theta
\|_{L^{2}}^{\frac{\beta}{\beta+1}} \|\Lambda_{x_{2}}^{\beta+1}\theta
\|_{L^{2}}^{\frac{1}{\beta+1}}
\nonumber\\&\quad +
C\|\Lambda^{\frac{2}{\widehat{r}}}
 \Lambda_{x_{1}}^{1-\widetilde{\delta}}\partial_{x_{2}}
 \theta\|_{L^{2}}  \|\theta\|_{L^{\widehat{r}}}\|\Lambda_{x_{1}}^{\alpha+1}\theta
\|_{L^{2}}^{\frac{\beta}{\beta+1}} \|\Lambda_{x_{2}}^{\beta+1}\theta
\|_{L^{2}}^{\frac{1}{\beta+1}}
\nonumber\\ &\leq C\|\Lambda_{x_{1}}^{\frac{2}{r}}\Lambda_{x_{1}}^{1-\alpha}\nabla
\theta\|_{L^{2}} \|\theta\|_{L^{r}} \|\Lambda_{x_{1}}^{\alpha}\nabla \theta\|_{L^{2}}\nonumber\\&\quad +
C\|\Lambda_{x_{1}}^{\frac{2}{\widetilde{r}}}
 \Lambda_{x_{1}}^{1-\widetilde{\delta}}\nabla
 \theta\|_{L^{2}}  \|\theta\|_{L^{\widetilde{r}}}\|\Lambda_{x_{1}}^{\alpha}\nabla\theta
\|_{L^{2}}^{\frac{\beta}{\beta+1}} \|\Lambda_{x_{2}}^{\beta}\nabla\theta
\|_{L^{2}}^{\frac{1}{\beta+1}}
\nonumber\\&\quad +
C\|\Lambda_{x_{2}}^{\frac{2}{\widehat{r}}}
 \Lambda_{x_{1}}^{1-\widetilde{\delta}}\nabla
 \theta\|_{L^{2}}  \|\theta\|_{L^{\widehat{r}}}\|\Lambda_{x_{1}}^{\alpha}\nabla\theta
\|_{L^{2}}^{\frac{\beta}{\beta+1}} \|\Lambda_{x_{2}}^{\beta}\nabla\theta
\|_{L^{2}}^{\frac{1}{\beta+1}}
\nonumber\\ &\leq C\|\nabla
\theta\|_{L^{2}}^{1-\frac{1-\alpha+\frac{2}{r}}{\alpha}}\|\Lambda_{x_{1}}^{\alpha}\nabla \theta\|_{L^{2}}^{\frac{1-\alpha+\frac{2}{r}}{\alpha}}
\|\theta\|_{L^{r}} \|\Lambda_{x_{1}}^{\alpha}\nabla \theta\|_{L^{2}}\nonumber\\&\quad +C\|\nabla
\theta\|_{L^{2}}^{1-\frac{1-\widetilde{\delta}+\frac{2}{\widetilde{r}}}{\alpha}}\|\Lambda_{x_{1}}^{\alpha}\nabla \theta\|_{L^{2}}^{\frac{1-\widetilde{\delta}+\frac{2}{\widetilde{r}}}{\alpha}} \|\theta\|_{L^{\widetilde{r}}}\|\Lambda_{x_{1}}^{\alpha}\nabla\theta
\|_{L^{2}}^{\frac{\beta}{\beta+1}} \|\Lambda_{x_{2}}^{\beta}\nabla\theta
\|_{L^{2}}^{\frac{1}{\beta+1}}
\nonumber\\&\quad +
C\|\nabla
\theta\|_{L^{2}}^{1-\frac{1-\widetilde{\delta}}
{\alpha}-\frac{\frac{2}{\widehat{r}}}{\beta}}
\|\Lambda_{x_{1}}^{\alpha}\nabla
\theta\|_{L^{2}}^{\frac{1-\widetilde{\delta}}{\alpha}}
\|\Lambda_{x_{2}}^{\beta}\nabla
\theta\|_{L^{2}}^{\frac{\frac{2}{\widehat{r}}}{\beta}}
 \|\theta\|_{L^{\widehat{r}}}\|\Lambda_{x_{1}}^{\alpha}\nabla\theta
\|_{L^{2}}^{\frac{\beta}{\beta+1}} \|\Lambda_{x_{2}}^{\beta}\nabla\theta
\|_{L^{2}}^{\frac{1}{\beta+1}}
  \nonumber\\& \leq
 \frac{1}{2}\|\Lambda_{x_{1}}^{\alpha}\nabla\theta\|_{L^{2}}^{2}+ \frac{1}{2}
 \|\Lambda_{x_{2}}^{\beta}\nabla\theta\|_{L^{2}}^{2}+C \|\theta\|_{L^{r}}^{\frac{2\alpha r}{(2\alpha-1)r-2}} \|\nabla\theta\|_{L^{2}}^{2}\nonumber\\&\quad
 +C \|\theta\|_{L^{\widetilde{r}}}^{\frac{2\alpha(\beta+1)\widetilde{r}}
 {(2\alpha\beta+\alpha-1)\widetilde{r}-2(\beta+1)}} \|\nabla\theta\|_{L^{2}}^{2}
 +C \|\theta\|_{L^{q}}^{\frac{2\alpha\beta(\beta+1) \widehat{r}}{(2\alpha\beta+\alpha-1)\beta \widehat{r}-2\alpha(\beta+1)} } \|\nabla\theta\|_{L^{2}}^{2},\nonumber
\end{align}
where we have used \eqref{avbnuy8911} and \eqref{t2326001} again. As a result, we have
\begin{align}\label{pre1ffgy89k}
& \frac{d}{dt}\|\nabla \theta(t)\|_{L^{2}}^{2}+
 \|\Lambda_{x_{1}}^{\alpha}\nabla\theta\|_{L^{2}}^{2}+
 \|\Lambda_{x_{2}}^{\beta}\nabla\theta\|_{L^{2}}^{2} \nonumber\\& \leq
 C \|\theta\|_{L^{r}}^{\frac{2\alpha r}{(2\alpha-1)r-2}} \|\nabla\theta\|_{L^{2}}^{2}
 +C \|\theta\|_{L^{\widetilde{r}}}^{\frac{2\alpha(\beta+1)\widetilde{r}}
 {(2\alpha\beta+\alpha-1)\widetilde{r}-2(\beta+1)}} \|\nabla\theta\|_{L^{2}}^{2}
 \nonumber\\&\quad+C \|\theta\|_{L^{q}}^{\frac{2\alpha\beta(\beta+1) \widehat{r}}{(2\alpha\beta+\alpha-1)\beta\widehat{r}-2\alpha(\beta+1)} } \|\nabla\theta\|_{L^{2}}^{2}.
\end{align}
Using \eqref{prbbhyu8}, we conclude from \eqref{pre1ffgy89k} that
\begin{align}\label{pnbh2nm}
 \frac{d}{dt}\|\nabla{\theta}(t)\|_{L^{2}}^{2}
+\|\Lambda_{x_{1}}^{\alpha}\nabla{\theta}\|_{L^{2}}^{2}+
 \|\Lambda_{x_{2}}^{\beta}\nabla{\theta}\|_{L^{2}}^{2}\leq C(1+t)^{-\frac{(\alpha+\beta)(2-p)\widetilde{\kappa} }{4\alpha\beta p}}\|\nabla {\theta}\|_{L^{2}}^{2},
\end{align}
where $\widetilde{\kappa}$ is given by
\begin{align}
\widetilde{\kappa}=\min\left\{\frac{2\alpha r}{(2\alpha-1)r-2},\,\frac{2\alpha(\beta+1)\widetilde{r}}
 {(2\alpha\beta+\alpha-1)\widetilde{r}-2(\beta+1)},\, \frac{2\alpha\beta(\beta+1) \widehat{r}}{(2\alpha\beta+\alpha-1)\beta\widehat{r}-2\alpha(\beta+1)}\right\}.
 \nonumber
 \end{align}
Due to the arbitrariness of $r,\widetilde{r},\widehat{r}\in (2,\,\infty)$, we take some suitable $r,\widetilde{r},\widehat{r}$ such that
\begin{align}
\frac{(\alpha+\beta)(2-p)\widetilde{\kappa}}{4\alpha\beta p}>1. \nonumber
\end{align}
The remainder proof of this case can be performed as that for the above case:  $\beta> \frac{1}{2\alpha+1}, \,0<\alpha\leq\frac{1}{2}$. Thus, we omit the details. Consequently, we complete the proof of Proposition \ref{fggyvby7}.
\end{proof}

\vskip .1in

As a corollary of the proof of Proposition \ref{fggyvby7}, we have the following result.
\begin{Pros}\label{fgdsfe34}
If $\alpha$ and $\beta$ satisfy (\ref{sdf2334}) and $\theta_{0}\in \dot{H}^{1}(\mathbb{R}^{2})\cap L^p(\mathbb{R}^{2})$ with $p\in [1,2)$, then it holds
\begin{align}\label{pre1y335}
\int_{0}^{\infty}(\|\Lambda_{x_{1}}^{\alpha}\nabla{\theta}(t)\|_{L^{2}}^{2}+
 \|\Lambda_{x_{2}}^{\beta}\nabla{\theta}(t)\|_{L^{2}}^{2})\leq C_{0},
\end{align}
where the constant $C_{0}>0$ depends only on $\alpha,\,\beta,\,p$ and $\theta_{0}$.
\end{Pros}

\begin{proof}
We recall \eqref{pre1y321} and \eqref{pnbh2nm}, namely
\begin{align}\label{pre1y336}
 \frac{d}{dt}\|\nabla{\theta}(t)\|_{L^{2}}^{2}
+\|\Lambda_{x_{1}}^{\alpha}\nabla{\theta}\|_{L^{2}}^{2}+
 \|\Lambda_{x_{2}}^{\beta}\nabla{\theta}\|_{L^{2}}^{2}\leq C(1+t)^{-\frac{(\alpha+\beta)(2-p)\kappa }{4\alpha\beta p}}\|\nabla {\theta}\|_{L^{2}}^{2},
\end{align}
\begin{align}\label{pre1y337}
 \frac{d}{dt}\|\nabla{\theta}(t)\|_{L^{2}}^{2}
+\|\Lambda_{x_{1}}^{\alpha}\nabla{\theta}\|_{L^{2}}^{2}+
 \|\Lambda_{x_{2}}^{\beta}\nabla{\theta}\|_{L^{2}}^{2}\leq C(1+t)^{-\frac{(\alpha+\beta)(2-p)\widetilde{\kappa} }{4\alpha\beta p}}\|\nabla {\theta}\|_{L^{2}}^{2},
\end{align}
where $\frac{(\alpha+\beta)(2-p)\kappa }{4\alpha\beta p}>1$ and $\frac{(\alpha+\beta)(2-p)\widetilde{\kappa} }{4\alpha\beta p}>1$. We thus get by applying the Gronwall inequality to \eqref{pre1y336} and \eqref{pre1y337} that
$$\sup_{t\geq0}\|\nabla{\theta}(t)\|_{L^{2}}^{2}+\int_{0}^{\infty}
(\|\Lambda_{x_{1}}^{\alpha}\nabla{\theta}(\tau)\|_{L^{2}}^{2}+
 \|\Lambda_{x_{2}}^{\beta}\nabla{\theta}(\tau)\|_{L^{2}}^{2})\,d\tau\leq C_{0}.$$
 This completes the proof of Proposition \ref{fgdsfe34}.
\end{proof}

\vskip .1in
However, the estimates obtained above are not sufficient to show the decay estimate of the solution in $\dot{H}^{s}$-norm for $s>1$ because we consider all the case $p\in[1,2)$ other than $p\in[1,p_{0})$ for some $p_{0}\in[1,2)$.
To overcome this difficulty, we need to establish the uniform-in-time of the $\dot{H}^2$-bound of the solution. With \eqref{pre1y315} and \eqref{pre1y335} in hand, we are ready to show following result.
\begin{Pros}\label{Lfgyt789j}
If $\alpha$ and $\beta$ satisfy (\ref{sdf2334}) and $\theta_{0}\in \dot{H}^{2}(\mathbb{R}^{2})\cap L^p(\mathbb{R}^{2})$ with $p\in [1,2)$, then it holds
\begin{eqnarray}\label{pre1y338}
\sup_{t\geq0}\|\Delta\theta(t)\|_{L^{2}}\leq
C_{0},
\end{eqnarray}
where the constant $C_{0}>0$ depends only on $\alpha,\,\beta,\,p$ and $\theta_{0}$.
\end{Pros}

\begin{proof}
The proof is inspired by the proof of \cite[Proposition 3.4]{Yenonli}.
But compared with \cite[Proposition 3.4]{Yenonli}, here we need more delicate estimates, which act some key roles in deriving the sharp decay estimate of the solution in $\dot{H}^{2}$-norm (see \eqref{pdfgtrq1} below).
To achieve this goal, we apply $\Delta$ to $(\ref{SQG})$ and multiply it by $\Delta \theta$, then integrate it over $\mathbb{R}^{2}$ by parts to show
\begin{eqnarray}\label{uhjefgre1}
\frac{1}{2}\frac{d}{dt}\|\Delta\theta(t)\|_{L^{2}}^{2}+\|\Lambda_{x_{1}}^{\alpha}\Delta\theta\|_{L^{2}}^{2}+
 \|\Lambda_{x_{2}}^{\beta}\Delta\theta\|_{L^{2}}^{2}
=-\int_{\mathbb{R}^{2}}{\Delta\{(u\cdot\nabla)\theta\} \Delta \theta\,dx}.
\end{eqnarray}
By means of the divergence free condition, the term at the right hand side of (\ref{uhjefgre1}) can be rewritten as follows
\begin{align}\label{uhjefgre2}
& -\int_{\mathbb{R}^{2}}{\Delta\{(u\cdot\nabla)\theta\} \Delta \theta\,dx}
\nonumber\\
&=  \int_{\mathbb{R}^{2}}{\Delta u_{1} \partial_{x_{1}}\theta \Delta \theta\,dx}+\int_{\mathbb{R}^{2}}{ \Delta u_{2} \partial_{x_{2}}\theta  \Delta \theta\,dx}
+2\int_{\mathbb{R}^{2}}{\partial_{x_{1}}u_{1}\partial_{x_{1}x_{1}}\theta \Delta \theta\,dx}\nonumber\\& \quad +2\int_{\mathbb{R}^{2}}{\partial_{x_{2}}u_{1}\partial_{x_{1}x_{2}}
\theta \Delta \theta\,dx}+2\int_{\mathbb{R}^{2}}{\partial_{x_{1}}u_{2}\partial_{x_{1}x_{2}}\theta \Delta \theta\,dx}
+2
\int_{\mathbb{R}^{2}}{\partial_{x_{2}}u_{2}\partial_{x_{2}x_{2}}\theta\Delta \theta\,dx}\nonumber\\
&\triangleq \sum_{m=1}^{6}\Xi_{m}.
\end{align}
Now our target is to handle the six terms at the right hand side of (\ref{uhjefgre2}).
We should point out that if $\alpha$ and $\beta$ satisfy (\ref{sdf2334}), then $\alpha>\frac{1}{2}$ or $\beta>\frac{1}{2}$ holds true. Therefore, we split the proof into two cases, namely,
$$\mbox{\textbf{Case 1}}:\  \alpha>\frac{1}{2};\qquad \mbox{\textbf{Case 2}}:\  \beta>\frac{1}{2}.$$
We first deal with \textbf{Case 1}. In this case, the terms $\Xi_{2}$-$\Xi_{5}$ can be estimated by using (\ref{qtri}) that
\begin{align}
\Xi_{2}
 \leq& C\|\partial_{x_{2}}\theta\|_{L^2} \, \|\Delta \theta\|_{L^2}^{1-\frac{1}{2\alpha}}
 \|\Lambda_{x_{1}}^{\alpha}\Delta \theta\|_{L^2}^{ \frac{1}{2\alpha}} \,\|\Delta u_{2}\|_{L^2}^{1-\frac{1}{2\alpha}}\|\Lambda_{x_{2}}^{\alpha}\Delta u_{2}\|_{L^2}^{ \frac{1}{2\alpha}}\nonumber\\
 \leq& C\|\nabla\theta\|_{L^2} \, \|\Delta\theta\|_{L^2}^{1-\frac{1}{2\alpha}}\|\Lambda_{x_{1}}^{\alpha}\Delta\theta\|_{L^2}^{ \frac{1}{2\alpha}} \,\|\Delta \theta\|_{L^2}^{1-\frac{1}{2\alpha}}\|\Lambda_{x_{1}}^{\alpha}\Delta \theta\|_{L^2}^{ \frac{1}{2\alpha}}
\nonumber\\
 \leq&
\epsilon\|\Lambda_{x_{1}}^{\alpha}\Delta\theta\|_{L^{2}}^{2}+C(\epsilon)
\|\nabla\theta\|_{L^2}^{\frac{2\alpha}{2\alpha-1}}\,\|\Delta \theta\|_{L^2}^{2},\nonumber
\end{align}
\begin{align}
\Xi_{3}
 \leq& C\|\partial_{x_{1}}u_{1}\|_{L^2} \, \|\Delta \theta\|_{L^2}^{1-\frac{1}{2\alpha}}
 \|\Lambda_{x_{1}}^{\alpha}\Delta \theta\|_{L^2}^{ \frac{1}{2\alpha}} \,\|\partial_{x_{1}x_{1}}\theta\|_{L^2}^{1-\frac{1}{2\alpha}}
\|\Lambda_{x_{2}}^{\alpha}\partial_{x_{1}x_{1}}\theta\|_{L^2}^{ \frac{1}{2\alpha}}\nonumber\\
 \leq& C\|\nabla \theta\|_{L^2} \, \|\Delta\theta\|_{L^2}^{1-\frac{1}{2\alpha}}
 \|\Lambda_{x_{1}}^{\alpha}\Delta\theta\|_{L^2}^{ \frac{1}{2\alpha}} \,\|\Delta \theta\|_{L^2}^{1-\frac{1}{2\alpha}}\|\Lambda_{x_{1}}^{\alpha}\Delta \theta\|_{L^2}^{ \frac{1}{2\alpha}}
\nonumber\\
 \leq&
\epsilon\|\Lambda_{x_{1}}^{\alpha}\Delta\theta\|_{L^{2}}^{2}+C(\epsilon)
\|\nabla\theta\|_{L^2}^{\frac{2\alpha}{2\alpha-1}}\,\|\Delta \theta\|_{L^2}^{2},\nonumber
\end{align}
\begin{align}
\Xi_{4}
 \leq& C\|\partial_{x_{2}}u_{1}\|_{L^2} \, \|\Delta \theta\|_{L^2}^{1-\frac{1}{2\alpha}}\|\Lambda_{x_{1}}^{\alpha}\Delta \theta\|_{L^2}^{ \frac{1}{2\alpha}} \,\|\partial_{x_{1}x_{2}}
\theta\|_{L^2}^{1-\frac{1}{2\alpha}}
\|\Lambda_{x_{2}}^{\alpha}\partial_{x_{1}x_{2}}
\theta\|_{L^2}^{ \frac{1}{2\alpha}}\nonumber\\
 \leq& C\|\nabla \theta\|_{L^2} \, \|\Delta\theta\|_{L^2}^{1-\frac{1}{2\alpha}}\|\Lambda_{x_{1}}^{\alpha}\Delta\theta\|_{L^2}^{ \frac{1}{2\alpha}} \,\|\Delta \theta\|_{L^2}^{1-\frac{1}{2\alpha}}\|\Lambda_{x_{1}}^{\alpha}\Delta \theta\|_{L^2}^{ \frac{1}{2\alpha}}
\nonumber\\
 \leq&
\epsilon\|\Lambda_{x_{1}}^{\alpha}\Delta\theta\|_{L^{2}}^{2}+C(\epsilon)
\|\nabla\theta\|_{L^2}^{\frac{2\alpha}{2\alpha-1}}\,\|\Delta \theta\|_{L^2}^{2},\nonumber
\end{align}
\begin{align}
\Xi_{5}
 \leq& C\|\partial_{x_{1}}u_{2}\|_{L^2} \, \|\Delta \theta\|_{L^2}^{1-\frac{1}{2\alpha}}\|\Lambda_{x_{1}}^{\alpha}\Delta \theta\|_{L^2}^{ \frac{1}{2\alpha}} \,\|\partial_{x_{1}x_{2}}\theta\|_{L^2}^{1-\frac{1}{2\alpha}}
\|\Lambda_{x_{2}}^{\alpha}\partial_{x_{1}x_{2}}\theta\|_{L^2}^{ \frac{1}{2\alpha}}\nonumber\\
 \leq& C\| \nabla\theta\|_{L^2} \, \|\Delta\theta\|_{L^2}^{1-\frac{1}{2\alpha}}\|\Lambda_{x_{1}}^{\alpha}\Delta\theta\|_{L^2}^{ \frac{1}{2\alpha}} \,\|\Delta \theta\|_{L^2}^{1-\frac{1}{2\alpha}}\|\Lambda_{x_{1}}^{\alpha}\Delta \theta\|_{L^2}^{ \frac{1}{2\alpha}}
\nonumber\\
 \leq&
\epsilon\|\Lambda_{x_{1}}^{\alpha}\Delta\theta\|_{L^{2}}^{2}+C(\epsilon)
\|\nabla\theta\|_{L^2}^{\frac{2\alpha}{2\alpha-1}}\,\|\Delta \theta\|_{L^2}^{2},\nonumber
\end{align}
where we have frequently used the Plancherel Theorem and the fact $u=(-\mathcal
{R}_{2}\theta,\,\,\mathcal {R}_{1}\theta)$. The details can be found from  (3.26) to (3.36) of \cite[Proposition 3.4]{Yenonli}.
The terms $\Xi_{1}$ and $\Xi_{6}$ should be handled differently. More precisely, we appeal to (\ref{qtri}) to derive
\begin{align}
\Xi_{1}
 \leq& C\|\Delta \theta\|_{L_{x_{1}}^2L_{x_{2}}^{\frac{2\alpha}{\alpha-(2\alpha-1)\beta}}} \, \|\partial_{x_{1}}\theta\|_{L^2}
 ^{1-\frac{\alpha-(2\alpha-1)\beta}{2\alpha^{2}}}
 \|\Lambda_{x_{2}}^{\alpha}\partial_{x_{1}}\theta\|_{L^2}^{ \frac{\alpha-(2\alpha-1)\beta}{2\alpha^{2}}} \,\|\Delta u_{1}\|_{L^2}^{1-\frac{1}{2\alpha}}\|\Lambda_{x_{1}}^{\alpha}\Delta u_{1}\|_{L^2}^{ \frac{1}{2\alpha}}\nonumber\\
 \leq&
 C\|\Lambda_{x_{2}}^{\frac{(2\alpha-1)\beta}{2\alpha}}\Delta \theta\|_{L_{x_{1}}^2L_{x_{2}}^{2}} \, \|\nabla\theta\|_{L^2}
 ^{1-\frac{\alpha-(2\alpha-1)\beta}{2\alpha^{2}}}
 \|\Lambda_{x_{1}}^{\alpha}\nabla\theta\|_{L^2}^{ \frac{\alpha-(2\alpha-1)\beta}{2\alpha^{2}}} \,\|\Delta \theta\|_{L^2}^{1-\frac{1}{2\alpha}}\|\Lambda_{x_{1}}^{\alpha}\Delta \theta\|_{L^2}^{ \frac{1}{2\alpha}}\nonumber\\
 \leq& C\|\Delta \theta\|_{L^{2}}^{\frac{1}{2\alpha}}\|\Lambda_{x_{2}}^{\beta}\Delta \theta\|_{L^{2}}^{\frac{2\alpha-1}{2\alpha}} \, \|\nabla\theta\|_{L^2}
 ^{1-\frac{\alpha-(2\alpha-1)\beta}{2\alpha^{2}}}
 \|\Lambda_{x_{1}}^{\alpha}\nabla\theta\|_{L^2}^{ \frac{\alpha-(2\alpha-1)\beta}{2\alpha^{2}}} \,\|\Delta \theta\|_{L^2}^{1-\frac{1}{2\alpha}}\|\Lambda_{x_{1}}^{\alpha}\Delta \theta\|_{L^2}^{ \frac{1}{2\alpha}}
\nonumber\\
 \leq&
\epsilon\|\Lambda_{x_{1}}^{\alpha}\Delta\theta\|_{L^{2}}^{2}
+\epsilon\|\Lambda_{x_{2}}^{\beta}\Delta\theta\|_{L^{2}}^{2}+C(\epsilon)
\|\nabla\theta\|_{L^2}
 ^{2-\frac{\alpha-(2\alpha-1)\beta}{\alpha^{2}}}
 \|\Lambda_{x_{1}}^{\alpha}\nabla\theta\|_{L^2}^{ \frac{\alpha-(2\alpha-1)\beta}{\alpha^{2}}} \,\|\Delta \theta\|_{L^2}^{2}\nonumber\\
 \leq&
\epsilon\|\Lambda_{x_{1}}^{\alpha}\Delta\theta\|_{L^{2}}^{2}
+\epsilon\|\Lambda_{x_{2}}^{\beta}\Delta\theta\|_{L^{2}}^{2}+C(\epsilon)
(\|\nabla\theta\|_{L^2}
 ^{2}+
 \|\Lambda_{x_{1}}^{\alpha}\nabla\theta\|_{L^2}^{ 2}) \,\|\Delta \theta\|_{L^2}^{2},\nonumber
\end{align}
\begin{align}
\Xi_{6}=&2
\int_{\mathbb{R}^{2}}{\partial_{x_{1}}\mathcal {R}_{2}\theta\partial_{x_{2}x_{2}}\theta\Delta \theta\,dx}
 \nonumber\\ \leq& C\|\Delta \theta\|_{L_{x_{1}}^2L_{x_{2}}^{\frac{2\alpha}{\alpha-(2\alpha-1)\beta}}} \, \|\partial_{x_{1}}\mathcal {R}_{2}\theta\|_{L^2}
 ^{1-\frac{\alpha-(2\alpha-1)\beta}{2\alpha^{2}}}
 \|\Lambda_{x_{2}}^{\alpha}\partial_{x_{1}}\mathcal {R}_{2}\theta\|_{L^2}^{ \frac{\alpha-(2\alpha-1)\beta}{2\alpha^{2}}} \nonumber\\ &\times\|\partial_{x_{2}x_{2}}\theta\|_{L^2}^{1-\frac{1}{2\alpha}}
 \|\Lambda_{x_{1}}^{\alpha}\partial_{x_{2}x_{2}}\theta\|_{L^2}^{ \frac{1}{2\alpha}}\nonumber\\
 \leq&
 C\|\Lambda_{x_{2}}^{\frac{(2\alpha-1)\beta}{2\alpha}}\Delta \theta\|_{L_{x_{1}}^2L_{x_{2}}^{2}} \, \|\nabla\theta\|_{L^2}
 ^{1-\frac{\alpha-(2\alpha-1)\beta}{2\alpha^{2}}}
 \|\Lambda_{x_{1}}^{\alpha}\nabla\theta\|_{L^2}^{ \frac{\alpha-(2\alpha-1)\beta}{2\alpha^{2}}} \,\|\Delta \theta\|_{L^2}^{1-\frac{1}{2\alpha}}\|\Lambda_{x_{1}}^{\alpha}\Delta \theta\|_{L^2}^{ \frac{1}{2\alpha}}\nonumber\\
 \leq& C\|\Delta \theta\|_{L^{2}}^{\frac{1}{2\alpha}}\|\Lambda_{x_{2}}^{\beta}\Delta \theta\|_{L^{2}}^{\frac{2\alpha-1}{2\alpha}} \, \|\nabla\theta\|_{L^2}
 ^{1-\frac{\alpha-(2\alpha-1)\beta}{2\alpha^{2}}}
 \|\Lambda_{x_{1}}^{\alpha}\nabla\theta\|_{L^2}^{ \frac{\alpha-(2\alpha-1)\beta}{2\alpha^{2}}} \,\|\Delta \theta\|_{L^2}^{1-\frac{1}{2\alpha}}\|\Lambda_{x_{1}}^{\alpha}\Delta \theta\|_{L^2}^{ \frac{1}{2\alpha}}
\nonumber\\
 \leq&
\epsilon\|\Lambda_{x_{1}}^{\alpha}\Delta\theta\|_{L^{2}}^{2}
+\epsilon\|\Lambda_{x_{2}}^{\beta}\Delta\theta\|_{L^{2}}^{2}+C(\epsilon)
(\|\nabla\theta\|_{L^2}
 ^{2}+
 \|\Lambda_{x_{1}}^{\alpha}\nabla\theta\|_{L^2}^{ 2}) \,\|\Delta \theta\|_{L^2}^{2}.\nonumber
\end{align}
Combining the above estimates and taking $\epsilon$ suitable small, we have
\begin{align}\label{hjywe350}
 &\frac{d}{dt}\|\Delta\theta(t)\|_{L^{2}}^{2}+\|\Lambda_{x_{1}}^{\alpha}\Delta\theta\|_{L^{2}}^{2}+
 \|\Lambda_{x_{2}}^{\beta}\Delta\theta\|_{L^{2}}^{2}
\nonumber\\ &\leq C
(\|\nabla\theta\|_{L^2}^{\frac{2\alpha}{2\alpha-1}}+\|\nabla\theta\|_{L^2}
 ^{2}+
 \|\Lambda_{x_{1}}^{\alpha}\nabla\theta\|_{L^2}^{ 2}) \,\|\Delta \theta\|_{L^2}^{2}.
\end{align}
Keeping in mind $\alpha>\frac{1}{2}$, \eqref{pre1y315} and \eqref{pre1y335}, it is not hard to check that
$$ \int_{0}^{\infty}
(\|\nabla\theta(\tau)\|_{L^2}^{\frac{2\alpha}{2\alpha-1}}+\|\nabla\theta(\tau)\|_{L^2}
 ^{2}+
 \|\Lambda_{x_{1}}^{\alpha}\nabla\theta(\tau)\|_{L^2}^{ 2})\,d\tau\leq C_{0}.$$
As a result, by the Gronwall inequality, we conclude
$$\sup_{t\geq0}\|\Delta\theta(t)\|_{L^{2}}^{2}+ \int_{0}^{\infty}{(\|\Lambda_{x_{1}}^{\alpha}\Delta\theta(
\tau)\|_{L^{2}}^{2}+\|\Lambda_{x_{2}}^{\beta}\Delta\theta(
\tau)\|_{L^{2}}^{2})\,d\tau}\leq
C_{0}.$$
Now let us turn our attention to \textbf{Case 2}. In this case, one may conclude by using the inequality (\ref{qtri}) that
 \begin{align}
\Xi_{1}
 \leq& C\|\partial_{x_{1}} \theta\|_{L^2} \, \|\Delta\theta\|_{L^2}^{1-\frac{1}{2\beta}}\|\Lambda_{x_{2}}^{\beta}\Delta\theta\|_{L^2}^{ \frac{1}{2\beta}} \,\|\Delta u_{1}\|_{L^2}^{1-\frac{1}{2\beta}}\|\Lambda_{x_{1}}^{\beta}\Delta u_{1}\|_{L^2}^{ \frac{1}{2\beta}}\nonumber\\
 \leq& C\| \nabla\theta\|_{L^2} \, \|\Delta\theta\|_{L^2}^{1-\frac{1}{2\beta}}\|\Lambda_{x_{2}}^{\beta}\Delta\theta\|_{L^2}^{ \frac{1}{2\beta}} \,\|\Delta \theta\|_{L^2}^{1-\frac{1}{2\beta}}\|\Lambda_{x_{2}}^{\beta}\Delta \theta\|_{L^2}^{ \frac{1}{2\beta}}
  \nonumber\\
 \leq&
\epsilon\|\Lambda_{x_{2}}^{\beta}\Delta\theta\|_{L^{2}}^{2}+C(\epsilon)
\|\nabla\theta\|_{L^2}^{\frac{2\beta}{2\beta-1}} \,\|\Delta \theta\|_{L^2}^{2},\nonumber
\end{align}
\begin{align}
\Xi_{4}
 \leq& C\|\partial_{x_{2}}u_{1}\|_{L^2} \, \|\Delta \theta\|_{L^2}^{1-\frac{1}{2\beta}}\|\Lambda_{x_{2}}^{\beta}\Delta \theta\|_{L^2}^{ \frac{1}{2\beta}} \,\|\partial_{x_{1}x_{2}}\theta\|_{L^2}^{1-\frac{1}{2\beta}}
\|\Lambda_{x_{1}}^{\beta}\partial_{x_{1}x_{2}}\theta\|_{L^2}^{ \frac{1}{2\beta}}\nonumber\\
 \leq& C\|\nabla \theta\|_{L^2} \, \|\Delta\theta\|_{L^2}^{1-\frac{1}{2\beta}}\|\Lambda_{x_{2}}^{\beta}\Delta\theta\|_{L^2}^{ \frac{1}{2\beta}} \,\|\Delta \theta\|_{L^2}^{1-\frac{1}{2\beta}}\|\Lambda_{x_{2}}^{\beta}\Delta \theta\|_{L^2}^{ \frac{1}{2\beta}}
\nonumber\\
 \leq&
\epsilon\|\Lambda_{x_{2}}^{\beta}\Delta\theta\|_{L^{2}}^{2}+C(\epsilon)
\|\nabla\theta\|_{L^2}^{\frac{2\beta}{2\beta-1}} \,\|\Delta \theta\|_{L^2}^{2},\nonumber
\end{align}
\begin{align}
\Xi_{5}
 \leq& C\|\partial_{x_{1}}u_{2}\|_{L^2} \, \|\Delta \theta\|_{L^2}^{1-\frac{1}{2\beta}}\|\Lambda_{x_{2}}^{\beta}\Delta\theta\|_{L^2}^{ \frac{1}{2\beta}} \,\|\partial_{x_{1}x_{2}}\theta\|_{L^2}^{1-\frac{1}{2\beta}}
\|\Lambda_{x_{1}}^{\beta}\partial_{x_{1}x_{2}}\theta\|_{L^2}^{ \frac{1}{2\beta}}\nonumber\\
 \leq& C\|\nabla \theta\|_{L^2} \, \|\Delta\theta\|_{L^2}^{1-\frac{1}{2\beta}}\|\Lambda_{x_{2}}^{\beta}\Delta\theta\|_{L^2}^{ \frac{1}{2\beta}} \,\|\Delta \theta\|_{L^2}^{1-\frac{1}{2\beta}}\|\Lambda_{x_{2}}^{\beta}\Delta \theta\|_{L^2}^{ \frac{1}{2\beta}}
\nonumber\\
 \leq&
\epsilon\|\Lambda_{x_{2}}^{\beta}\Delta\theta\|_{L^{2}}^{2}+C(\epsilon)
\|\nabla\theta\|_{L^2}^{\frac{2\beta}{2\beta-1}} \,\|\Delta \theta\|_{L^2}^{2}, \nonumber
\end{align}
\begin{align}
\Xi_{6}
 \leq& C\|\partial_{x_{2}}u_{2}\|_{L^2} \, \|\Delta \theta\|_{L^2}^{1-\frac{1}{2\beta}}
 \|\Lambda_{x_{2}}^{\beta}\Delta \theta\|_{L^2}^{ \frac{1}{2\beta}} \,\|\partial_{x_{2}x_{2}}\theta\|_{L^2}^{1-\frac{1}{2\beta}}
\|\Lambda_{x_{1}}^{\beta}\partial_{x_{2}x_{2}}\theta\|_{L^2}^{ \frac{1}{2\beta}}\nonumber\\
 \leq& C\|\nabla \theta\|_{L^2} \, \|\Delta\theta\|_{L^2}^{1-\frac{1}{2\beta}}\|\Lambda_{x_{2}}^{\beta}\Delta\theta\|_{L^2}^{ \frac{1}{2\beta}} \,\|\Delta \theta\|_{L^2}^{1-\frac{1}{2\beta}}\|\Lambda_{x_{2}}^{\beta}\Delta \theta\|_{L^2}^{ \frac{1}{2\beta}}
\nonumber\\
 \leq&
\epsilon\|\Lambda_{x_{2}}^{\beta}\Delta\theta\|_{L^{2}}^{2}+C(\epsilon)
\|\nabla\theta\|_{L^2}^{\frac{2\beta}{2\beta-1}} \,\|\Delta \theta\|_{L^2}^{2}. \nonumber
\end{align}
The terms $\Xi_{2}$ and $\Xi_{3}$ should be treated differently. Again, according to (\ref{qtri}), one has
\begin{align}
\Xi_{2}
\leq& C\|\Delta \theta\|_{L_{x_{1}}^{\frac{2\beta}{\beta-(2\beta-1)\alpha}}L_{x_{2}}^2} \, \|\partial_{x_{2}}\theta\|_{L^2}
 ^{1-\frac{\beta-(2\beta-1)\alpha}{2\beta^{2}}}
 \|\Lambda_{x_{1}}^{\beta}\partial_{x_{2}}\theta\|_{L^2}^{ \frac{\beta-(2\beta-1)\alpha}{2\beta^{2}}} \,\|\Delta u_{2}\|_{L^2}^{1-\frac{1}{2\beta}}\|\Lambda_{x_{2}}^{\beta}\Delta u_{2}\|_{L^2}^{ \frac{1}{2\beta}}\nonumber\\
 \leq&
 C\|\Lambda_{x_{1}}^{ \frac{(2\beta-1)\alpha}{2\beta}}\Delta \theta\|_{L^2} \, \|\nabla\theta\|_{L^2}
 ^{1-\frac{\beta-(2\beta-1)\alpha}{2\beta^{2}}}
 \|\Lambda_{x_{2}}^{\beta}\nabla\theta\|_{L^2}^{ \frac{\beta-(2\beta-1)\alpha}{2\beta^{2}}} \,\|\Delta \theta\|_{L^2}^{1-\frac{1}{2\beta}}\|\Lambda_{x_{2}}^{\beta}\Delta \theta\|_{L^2}^{ \frac{1}{2\beta}}
 \nonumber\\
 \leq&  C\| \Delta \theta\|_{L^2}^{ 1- \frac{2\beta-1}{2\beta}}\|\Lambda_{x_{1}}^{\alpha}\Delta \theta\|_{L^2}^{ \frac{2\beta-1}{2\beta}} \, \|\nabla\theta\|_{L^2}
 ^{1-\frac{\beta-(2\beta-1)\alpha}{2\beta^{2}}}
 \|\Lambda_{x_{2}}^{\beta}\nabla\theta\|_{L^2}^{ \frac{\beta-(2\beta-1)\alpha}{2\beta^{2}}} \,\|\Delta \theta\|_{L^2}^{1-\frac{1}{2\beta}}\|\Lambda_{x_{2}}^{\beta}\Delta \theta\|_{L^2}^{ \frac{1}{2\beta}}
\nonumber\\
 \leq&
\epsilon\|\Lambda_{x_{1}}^{\alpha}\Delta\theta\|_{L^{2}}^{2}
+\epsilon\|\Lambda_{x_{2}}^{\beta}\Delta\theta\|_{L^{2}}^{2}+C(\epsilon)
 \|\nabla\theta\|_{L^2}
 ^{2-\frac{\beta-(2\beta-1)\alpha}{\beta^{2}}}
 \|\Lambda_{x_{2}}^{\beta}\nabla\theta\|_{L^2}^{ \frac{\beta-(2\beta-1)\alpha}{\beta^{2}}}\,\|\Delta \theta\|_{L^2}^{2}
 \nonumber\\
 \leq&
\epsilon\|\Lambda_{x_{1}}^{\alpha}\Delta\theta\|_{L^{2}}^{2}
+\epsilon\|\Lambda_{x_{2}}^{\beta}\Delta\theta\|_{L^{2}}^{2}+C(\epsilon)(
 \|\nabla\theta\|_{L^2}^{2}+\|\Lambda_{x_{2}}^{\beta}\nabla\theta\|_{L^2}^{2})\,\|\Delta \theta\|_{L^2}^{2},\nonumber
\end{align}
\begin{align}
\Xi_{3}=&
-2\int_{\mathbb{R}^{2}}{\partial_{x_{2}}\mathcal{R}_{1}\theta
\partial_{x_{1}x_{1}}\theta \Delta \theta\,dx}
\nonumber\\
 \leq&  C\|\Delta \theta\|_{L_{x_{1}}^{\frac{2\beta}{\beta-(2\beta-1)\alpha}}L_{x_{2}}^2} \, \|\partial_{x_{2}}\mathcal{R}_{1}\theta\|_{L^2}
 ^{1-\frac{\beta-(2\beta-1)\alpha}{2\beta^{2}}}
 \|\Lambda_{x_{1}}^{\beta}\partial_{x_{2}}\mathcal{R}_{1}\theta\|_{L^2}^{ \frac{\beta-(2\beta-1)\alpha}{2\beta^{2}}} \nonumber\\&\times\|\partial_{x_{1}x_{1}}\theta \|_{L^2}^{1-\frac{1}{2\beta}}\|\Lambda_{x_{2}}^{\beta}\partial_{x_{1}x_{1}}\theta \|_{L^2}^{ \frac{1}{2\beta}}\nonumber\\
 \leq&
 C\|\Lambda_{x_{1}}^{ \frac{(2\beta-1)\alpha}{2\beta}}\Delta \theta\|_{L^2} \, \|\nabla\theta\|_{L^2}
 ^{1-\frac{\beta-(2\beta-1)\alpha}{2\beta^{2}}}
 \|\Lambda_{x_{2}}^{\beta}\nabla\theta\|_{L^2}^{ \frac{\beta-(2\beta-1)\alpha}{2\beta^{2}}} \,\|\Delta \theta\|_{L^2}^{1-\frac{1}{2\beta}}\|\Lambda_{x_{2}}^{\beta}\Delta \theta\|_{L^2}^{ \frac{1}{2\beta}}
 \nonumber\\
 \leq&  C\| \Delta \theta\|_{L^2}^{ 1- \frac{2\beta-1}{2\beta}}\|\Lambda_{x_{1}}^{\alpha}\Delta \theta\|_{L^2}^{ \frac{2\beta-1}{2\beta}} \, \|\nabla\theta\|_{L^2}
 ^{1-\frac{\beta-(2\beta-1)\alpha}{2\beta^{2}}}
 \|\Lambda_{x_{2}}^{\beta}\nabla\theta\|_{L^2}^{ \frac{\beta-(2\beta-1)\alpha}{2\beta^{2}}} \,\|\Delta \theta\|_{L^2}^{1-\frac{1}{2\beta}}\|\Lambda_{x_{2}}^{\beta}\Delta \theta\|_{L^2}^{ \frac{1}{2\beta}}
 \nonumber\\
 \leq&
\epsilon\|\Lambda_{x_{1}}^{\alpha}\Delta\theta\|_{L^{2}}^{2}
+\epsilon\|\Lambda_{x_{2}}^{\beta}\Delta\theta\|_{L^{2}}^{2}+C(\epsilon)(
 \|\nabla\theta\|_{L^2}^{2}+\|\Lambda_{x_{2}}^{\beta}\nabla\theta\|_{L^2}^{2})\,\|\Delta \theta\|_{L^2}^{2}.\nonumber
\end{align}
Combining the above estimates and taking $\epsilon$ suitable small, it follows that
\begin{align}\label{hjywe351}
 &\frac{d}{dt}\|\Delta\theta(t)\|_{L^{2}}^{2}+\|\Lambda_{x_{1}}^{\alpha}\Delta\theta\|_{L^{2}}^{2}+
 \|\Lambda_{x_{2}}^{\beta}\Delta\theta\|_{L^{2}}^{2}
\nonumber\\&\leq C(\|\nabla\theta\|_{L^2}^{\frac{2\beta}{2\beta-1}}+
 \|\nabla\theta\|_{L^2}^{2}+\|\Lambda_{x_{2}}^{\beta}\nabla\theta\|_{L^2}^{2})\|\Delta \theta\|_{L^2}^{2}.
\end{align}
Keeping in mind $\beta>\frac{1}{2}$, \eqref{pre1y315} and \eqref{pre1y335}, we  see
$$ \int_{0}^{\infty}
(\|\nabla\theta(\tau)\|_{L^2}^{\frac{2\beta}{2\beta-1}}+\|\nabla\theta(\tau)\|_{L^2}
 ^{2}+
 \|\Lambda_{x_{2}}^{\beta}\nabla\theta(\tau)\|_{L^2}^{ 2})\,d\tau\leq C_{0}.$$
It thus follows from the Gronwall inequality that
$$\sup_{t\geq0}\|\Delta\theta(t)\|_{L^{2}}^{2}+ \int_{0}^{\infty}{(\|\Lambda_{x_{1}}^{\alpha}\Delta\theta(
\tau)\|_{L^{2}}^{2}+\|\Lambda_{x_{2}}^{\beta}\Delta\theta(
\tau)\|_{L^{2}}^{2})\,d\tau}\leq
C_{0}.$$
Therefore, we end the proof of Proposition \ref{Lfgyt789j}.
\end{proof}

\vskip .1in
With Proposition \ref{Lfgyt789j} and Proposition \ref{fggyvby7} at hand, we are now in a position to derive the preliminary decay rate of the $\dot{H}^s$-norm of the solution.
\begin{Pros}\label{dfdsfp36}
If $\alpha$ and $\beta$ satisfy (\ref{sdf2334}), $\theta_{0}\in \dot{H}^{s}(\mathbb{R}^{2})\cap L^{p}(\mathbb{R}^{2})$ with $s\geq2$ and $p\in [1,2)$.
Then the global smooth solution of \eqref{SQG} admits the decay estimate
\begin{align}\label{pre1y343}
  \|\Lambda^{s} \theta(t)\|_{L^2}\leq C_{0}(1+t)^{-\frac{(\alpha+\beta)(2-p)+2\min\{\alpha,\,\beta\}p}{4\alpha\beta p}},
\end{align}
where the constant $C_{0}>0$ depends only on $\alpha,\,\beta,\,s,\,p$ and $\theta_{0}$.
\end{Pros}

\begin{proof}
Applying $\Lambda^{s}$ to \eqref{SQG} and testing the equation by $\Lambda^{s}\theta$, we obtain by \eqref{yzz1} that
\begin{align}\label{pre1y344}
&\frac{1}{2}\frac{d}{dt}\|\Lambda^{s} {\theta}(t)\|_{L^{2}}^{2}
+\|\Lambda_{x_{1}}^{\alpha}\Lambda^{s} {\theta}\|_{L^{2}}^{2}+
 \|\Lambda_{x_{2}}^{\beta}\Lambda^{s} {\theta}\|_{L^{2}}^{2}\nonumber\\&=-\int_{\mathbb{R}^{2}}\Lambda^{s}(u \cdot \nabla\theta) \Lambda^{s}\theta\,dx\nonumber\\
 &=-\int_{\mathbb{R}^{2}}[\Lambda^{s},u]\cdot \nabla\theta  \Lambda^{s}\theta\,dx
 \nonumber\\
 &\leq C\|[\Lambda^{s},u]\cdot \nabla\theta\|_{L^{\frac{2q}{q+1}}}  \|\Lambda^{s}\theta\|_{L^{\frac{2q}{q-1}}}
  \nonumber\\
 &\leq C(\|\nabla u\|_{L^{q}} \|\Lambda^{s}\theta\|_{L^{\frac{2q}{q-1}}}+\|\nabla \theta\|_{L^{q}} \|\Lambda^{s}u\|_{L^{\frac{2q}{q-1}}}) \|\Lambda^{s}\theta\|_{L^{\frac{2q}{q-1}}}
 \nonumber\\
 &\leq C\|\nabla \theta\|_{L^{q}} \|\Lambda^{s}\theta\|_{L^{\frac{2q}{q-1}}}^{2}
  \nonumber\\
 &\leq C\|\nabla \theta\|_{L^{q}} \|\Lambda^{s+\frac{1}{q}}\theta\|_{L^{2}}^{2}.
\end{align}
According to the following interpolation inequalities (see \eqref{avbnuy8911})
$$\|\Lambda_{x_{1}}^{s+\frac{1}{q}}\theta\|_{L^{2}}\leq C\|\Lambda_{x_{1}}\theta\|_{L^{2}}^{\frac{\alpha-\frac{1}{q}}{s+\alpha-1}}
\|\Lambda_{x_{1}}^{s+\alpha}\theta\|_{L^{2}}^{\frac{s+\frac{1}{q}-1}{s+\alpha-1}},
\quad \|\Lambda_{x_{2}}^{s+\frac{1}{q}}\theta\|_{L^{2}}\leq C\|\Lambda_{x_{2}}\theta\|_{L^{2}}^{\frac{\beta-\frac{1}{q}}{s+\beta-1}}
\|\Lambda_{x_{2}}^{s+\beta}\theta\|_{L^{2}}^{\frac{s+\frac{1}{q}-1}{s+\beta-1}},$$
it follows from \eqref{pre1y344} that
\begin{align}
&\frac{1}{2}\frac{d}{dt}\|\Lambda^{s} {\theta}(t)\|_{L^{2}}^{2}
+\|\Lambda_{x_{1}}^{\alpha}\Lambda^{s} {\theta}\|_{L^{2}}^{2}+
 \|\Lambda_{x_{2}}^{\beta}\Lambda^{s} {\theta}\|_{L^{2}}^{2}\nonumber\\&\leq
  C\|\nabla \theta\|_{L^{q}} (\|\Lambda_{x_{1}}^{s+\frac{1}{q}}\theta\|_{L^{2}}^{2}
  +\|\Lambda_{x_{2}}^{s+\frac{1}{q}}\theta\|_{L^{2}}^{2})
  \nonumber\\&\leq
  C\|\nabla \theta\|_{L^{q}} (\|\Lambda_{x_{1}}\theta\|_{L^{2}}^{\frac{2\alpha-\frac{2}{q}}{s+\alpha-1}}
\|\Lambda_{x_{1}}^{s+\alpha}\theta\|_{L^{2}}^{\frac{2s+\frac{2}{q}-2}{s+\alpha-1}}
  +\|\Lambda_{x_{2}}\theta\|_{L^{2}}^{\frac{2\beta-\frac{2}{q}}{s+\beta-1}}
\|\Lambda_{x_{2}}^{s+\beta}\theta\|_{L^{2}}^{\frac{2s+\frac{2}{q}-2}{s+\beta-1}})
 \nonumber\\&\leq
  C\|\nabla \theta\|_{L^{q}} (\|\nabla\theta\|_{L^{2}}^{\frac{2\alpha-\frac{2}{q}}{s+\alpha-1}}
\|\Lambda_{x_{1}}^{\alpha}\Lambda^{s}\theta\|_{L^{2}}^{\frac{2s+\frac{2}{q}-2}{s+\alpha-1}}
  +\|\nabla\theta\|_{L^{2}}^{\frac{2\beta-\frac{2}{q}}{s+\beta-1}}
\|\Lambda_{x_{2}}^{\beta}\Lambda^{s}\theta\|_{L^{2}}^{\frac{2s+\frac{2}{q}-2}{s+\beta-1}})
 \nonumber\\&\leq
\frac{1}{2}\|\Lambda_{x_{1}}^{\alpha}\Lambda^{s} {\theta}\|_{L^{2}}^{2}+
 \frac{1}{2}\|\Lambda_{x_{2}}^{\beta}\Lambda^{s} {\theta}\|_{L^{2}}^{2}+C\|\nabla \theta\|_{L^{q}}^{\frac{(s+\alpha-1)q}{\alpha q-1}} \| \nabla{\theta}\|_{L^{2}}^{2}
 +C\|\nabla \theta\|_{L^{q}}^{\frac{(s+\beta-1)q}{\beta q-1}} \| \nabla {\theta}\|_{L^{2}}^{2}, \nonumber
\end{align}
which implies
\begin{align}\label{pre1y345}
  \frac{d}{dt}\|\Lambda^{s} {\theta}(t)\|_{L^{2}}^{2}
+\|\Lambda_{x_{1}}^{\alpha}\Lambda^{s} {\theta}\|_{L^{2}}^{2}+
 \|\Lambda_{x_{2}}^{\beta}\Lambda^{s} {\theta}\|_{L^{2}}^{2}\leq&
 C\|\nabla \theta\|_{L^{q}}^{\frac{(s+\alpha-1)q}{\alpha q-1}} \| \nabla{\theta}\|_{L^{2}}^{2}
 \nonumber\\&+C\|\nabla \theta\|_{L^{q}}^{\frac{(s+\beta-1)q}{\beta q-1}} \| \nabla{\theta}\|_{L^{2}}^{2}\nonumber\\ \leq&
 C\|\nabla \theta\|_{L^{2}}^{\frac{2(s+\alpha-1)}{\alpha q-1}}\|\Delta \theta\|_{L^{2}}^{\frac{(q-2)(s+\alpha-1)}{\alpha q-1}} \| \nabla{\theta}\|_{L^{2}}^{2}
 \nonumber\\&+C\|\nabla \theta\|_{L^{2}}^{\frac{2(s+\beta-1)}{\beta q-1}}\|\Delta \theta\|_{L^{2}}^{\frac{(q-2)(s+\beta-1)}{\beta q-1}} \| \nabla{\theta}\|_{L^{2}}^{2}\nonumber\\ \leq&
 C\|\nabla {\theta}\|_{L^{2}}^{2},
\end{align}
where in the last line we have used \eqref{pre310} and \eqref{pre1y338}.
Making use of the following interpolation inequalities (see \eqref{avbnuy8911})
$$\|\Lambda_{x_{1}}^{s}\theta\|_{L^{2}}\leq C\|\Lambda_{x_{1}}\theta\|_{L^{2}}^{\frac{\alpha}{s+\alpha-1}}
\|\Lambda_{x_{1}}^{s+\alpha}\theta\|_{L^{2}}^{\frac{s-1}{s+\alpha-1}},\qquad \|\Lambda_{x_{2}}^{s}\theta\|_{L^{2}}\leq C\|\Lambda_{x_{2}}\theta\|_{L^{2}}^{\frac{\beta}{s+\beta-1}}
\|\Lambda_{x_{2}}^{s+\beta}\theta\|_{L^{2}}^{\frac{s-1}{s+\beta-1}},$$
we have that
\begin{align}
 \|\Lambda^{s} {\theta}\|_{L^{2}}^{2}&=\|\Lambda_{x_{1}}^{s}\theta\|_{L^{2}}^{2}
 +\|\Lambda_{x_{2}}^{s}\theta\|_{L^{2}}^{2}
 \nonumber\\& \leq C\|\Lambda_{x_{1}}\theta\|_{L^{2}}^{\frac{2\alpha}{s+\alpha-1}}
\|\Lambda_{x_{1}}^{s+\alpha}\theta\|_{L^{2}}^{\frac{2s-2}{s+\alpha-1}}
+C\|\Lambda_{x_{2}}\theta\|_{L^{2}}^{\frac{2\beta}{s+\beta-1}}
\|\Lambda_{x_{2}}^{s+\beta}\theta\|_{L^{2}}^{\frac{2s-2}{s+\beta-1}}
 \nonumber\\& \leq C\|\nabla\theta\|_{L^{2}}^{\frac{2\alpha}{s+\alpha-1}}
\|\Lambda_{x_{1}}^{\alpha}\Lambda^{s}\theta\|_{L^{2}}^{\frac{2s-2}{s+\alpha-1}}
+C\|\nabla\theta\|_{L^{2}}^{\frac{2\beta}{s+\beta-1}}
\|\Lambda_{x_{2}}^{\beta}\Lambda^{s}\theta\|_{L^{2}}^{\frac{2s-2}{s+\beta-1}}\nonumber\\ &\leq
 \|\Lambda_{x_{1}}^{\alpha}\Lambda^{s} {\theta}\|_{L^{2}}^{2}+
 \|\Lambda_{x_{2}}^{\beta}\Lambda^{s} {\theta}\|_{L^{2}}^{2}+ {C}\| \nabla{\theta}\|_{L^{2}}^{2}.\nonumber
\end{align}
Putting the above estimate into \eqref{pre1y345} yields
\begin{align}\label{pre1y346}
\frac{d}{dt}\|\Lambda^{s} {\theta}(t)\|_{L^{2}}^{2}
+ \|\Lambda^{s} {\theta}\|_{L^{2}}^{2}&\leq C\|\nabla\theta\|_{L^{2}}^{2}\leq C(1+t)^{-\frac{(\alpha+\beta)(2-p)+2\min\{\alpha,\,\beta\}p}{2\alpha\beta p}},
\end{align}
where we have used \eqref{pre1y315}. It follows from \eqref{pre1y346} that
\begin{align}
\frac{d}{dt}(e^{t}\|\Lambda^{s} {\theta}(t)\|_{L^{2}}^{2})  \leq Ce^{t}(1+t)^{-\frac{(\alpha+\beta)(2-p)+2\min\{\alpha,\,\beta\}p}{2\alpha\beta p}},\nonumber
\end{align}
which further implies
\begin{align}
e^{t}\|\Lambda^{s} {\theta}(t)\|_{L^{2}}^{2}  \leq \|\Lambda^{s} {\theta}_{0}\|_{L^{2}}^{2}+C\int_{0}^{t}e^{\tau}(1+\tau)
^{-\frac{(\alpha+\beta)(2-p)+2\min\{\alpha,\,\beta\}p}{2\alpha\beta p}}\,d\tau.\nonumber
\end{align}
As a result, one has
\begin{align}\label{pre1y347}
\|\Lambda^{s} {\theta}(t)\|_{L^{2}}^{2}  \leq e^{-t}\|\Lambda^{s} {\theta}_{0}\|_{L^{2}}^{2}+C\int_{0}^{t}e^{-(t-\tau)}(1+\tau)^{-\frac{(\alpha+\beta)
(2-p)+2\min\{\alpha,\,\beta\}p}{2\alpha\beta p}}\,d\tau.
\end{align}
Invoking \eqref{pre314}, we get from \eqref{pre1y347} that
\begin{align*}
 &\|\Lambda^{s} \theta(t)\|_{L^2}\leq C_{0}(1+t)^{-\frac{(\alpha+\beta)(2-p)+2\min\{\alpha,\,\beta\}p}{4\alpha\beta p}}.
\end{align*}
We therefore finish the proof of Proposition \ref{dfdsfp36}.
\end{proof}

\vskip .1in
With \eqref{pre1y343} in hand, we are in the position to show the sharp decay  estimate of the solution in $\dot{H}^{2}$-norm, which plays a key role in obtaining the sharp decay estimate of the solution in $\dot{H}^{s}$-norm for any $s>2$.
\begin{Pros}\label{fxcdfn9}
If $\alpha$ and $\beta$ satisfy (\ref{sdf2334}) and $\theta_{0}\in \dot{H}^{2}(\mathbb{R}^{2})\cap L^p(\mathbb{R}^{2})$ with $p\in [1,2)$, then the solution admits the decay estimate
\begin{align}\label{pdfgtrq1}
 \|\Delta\theta(t)\|_{L^2}\leq C_{0}(1+t)^{-\frac{(\alpha+\beta)(2-p)+4\min\{\alpha,\,\beta\}p}{4\alpha\beta p}},
\end{align}
where the constant $C_{0}>0$ depends only on $\alpha,\,\beta,\,p$ and $\theta_{0}$.
\end{Pros}
\begin{proof}
It follows from the proof of Proposition \ref{Lfgyt789j} that for $\alpha>\frac{1}{2}$
\begin{align}\label{mpdfhj347}
 &\frac{d}{dt}\|\Delta\theta(t)\|_{L^{2}}^{2}+\|\Lambda_{x_{1}}^{\alpha}\Delta\theta\|_{L^{2}}^{2}+
 \|\Lambda_{x_{2}}^{\beta}\Delta\theta\|_{L^{2}}^{2}\nonumber\\
&\leq C
(\|\nabla\theta\|_{L^2}^{\frac{2\alpha}{2\alpha-1}}+\|\nabla\theta\|_{L^2}
 ^{2}+
 \|\Lambda_{x_{1}}^{\alpha}\nabla\theta\|_{L^2}^{ 2}) \,\|\Delta \theta\|_{L^2}^{2}\nonumber\\
&\leq C
(\|\nabla\theta\|_{L^2}^{\frac{2\alpha}{2\alpha-1}}+\|\nabla\theta\|_{L^2}
 ^{2}+
 \|\Lambda_{}^{\alpha+1} \theta\|_{L^2}^{ 2}) \,\|\Delta \theta\|_{L^2}^{2}
 \nonumber\\
&\leq C
(\|\nabla\theta\|_{L^2}^{\frac{2\alpha}{2\alpha-1}}+\|\nabla\theta\|_{L^2}
 ^{2}+\|\nabla\theta\|_{L^2}^{2-2\alpha}
 \|\Delta\theta\|_{L^2}^{2\alpha}) \,\|\Delta \theta\|_{L^2}^{2}
  \nonumber\\
&\leq C
(\|\nabla\theta\|_{L^2}^{\frac{2\alpha}{2\alpha-1}}+\|\nabla\theta\|_{L^2}
 ^{2}+
 \|\Delta\theta\|_{L^2}^{2 }) \,\|\Delta \theta\|_{L^2}^{2}
 \nonumber\\
&\leq
 C_{0}\left((1+t)^{-\frac{(\alpha+\beta)(2-p)+2\min\{\alpha,\,\beta\}p}{2\alpha\beta p}\frac{2\alpha}{2\alpha-1}}+(1+t)^{-\frac{(\alpha+\beta)(2-p)+2\min\{\alpha,\,\beta\}p}{2\alpha\beta p}}\right)\,\|\Delta \theta\|_{L^2}^{2}
 \nonumber\\
&\leq
 C_{0}(1+t)^{-\frac{(\alpha+\beta)(2-p)+2\min\{\alpha,\,\beta\}p}{2\alpha\beta p}} \,\|\Delta \theta\|_{L^2}^{2},
\end{align}
where we have used \eqref{pre1y315} and \eqref{pre1y343} with $s=2$.
Similarly, we obtain for $\beta>\frac{1}{2}$ that
\begin{align}\label{mpdfhj348}
 &\frac{d}{dt}\|\Delta\theta(t)\|_{L^{2}}^{2}+\|\Lambda_{x_{1}}^{\alpha}\Delta\theta\|_{L^{2}}^{2}+
 \|\Lambda_{x_{2}}^{\beta}\Delta\theta\|_{L^{2}}^{2}\nonumber\\
&\leq C
(\|\nabla\theta\|_{L^2}^{\frac{2\beta}{2\beta-1}}+\|\nabla\theta\|_{L^2}
 ^{2}+
 \|\Lambda_{x_{2}}^{\beta}\nabla\theta\|_{L^2}^{ 2}) \,\|\Delta \theta\|_{L^2}^{2}
  \nonumber\\
&\leq C
(\|\nabla\theta\|_{L^2}^{\frac{2\beta}{2\beta-1}}+\|\nabla\theta\|_{L^2}
 ^{2}+
 \|\Delta\theta\|_{L^2}^{2 }) \,\|\Delta \theta\|_{L^2}^{2}
 \nonumber\\
&\leq
 C_{0}(1+t)^{-\frac{(\alpha+\beta)(2-p)+2\min\{\alpha,\,\beta\}p}{2\alpha\beta p}} \,\|\Delta \theta\|_{L^2}^{2}.
\end{align}
In one word, we get by combining \eqref{mpdfhj347} and \eqref{mpdfhj348} that
\begin{align}\label{mpdfhj349}
 &\frac{d}{dt}\|\Delta\theta(t)\|_{L^{2}}^{2}+\|\Lambda_{x_{1}}^{\alpha}\Delta\theta\|_{L^{2}}^{2}+
 \|\Lambda_{x_{2}}^{\beta}\Delta\theta\|_{L^{2}}^{2}\nonumber\\
&\leq
 C_{0}(1+t)^{-\frac{(\alpha+\beta)(2-p)+2\min\{\alpha,\,\beta\}p}{2\alpha\beta p}} \,\|\Delta \theta\|_{L^2}^{2}.
\end{align}
Obviously, one has
$$\frac{(\alpha+\beta)(2-p)+2\min\{\alpha,\,\beta\}p}{2\alpha\beta p}>1.$$
Applying the same arguments used in dealing with \eqref{pre1y351}-\eqref{tyxcv55} below to \eqref{mpdfhj349}, we thus deduce
\begin{align}
  \|\Delta\theta(t)\|_{L^2}\leq C_{0}(1+t)^{-\frac{(\alpha+\beta)(2-p)+4\min\{\alpha,\,\beta\}p}{4\alpha\beta p}}.\nonumber
\end{align}
Consequently, Proposition \ref{fxcdfn9} is completed.
\end{proof}

\vskip .1in
With the help of \eqref{pdfgtrq1}, we are ready to improve the decay of the derivatives of order $s>2$ of the solution. The decay estimate \eqref{pre1y343} is not optimal but does provide an essential step to obtain the optimal decay, namely \eqref{fvbnm023}.
\begin{Pros}\label{dfdsfp37}
If $\alpha$ and $\beta$ satisfy (\ref{sdf2334}), $\theta_{0}\in \dot{H}^{s}(\mathbb{R}^{2})\cap L^{p}(\mathbb{R}^{2})$ with $s\geq0$ and $p\in [1,2)$.
Then the global smooth solution of \eqref{SQG} admits the decay estimate
\begin{align}\label{pre1y348}
 \|\Lambda^{s} \theta(t)\|_{L^2}\leq C_{0}(1+t)^{-\frac{(\alpha+\beta)(2-p)+2\min\{\alpha,\,\beta\}sp}{4\alpha\beta p}},
\end{align}
where the constant $C_{0}>0$ depends only on $\alpha,\,\beta,\,s,\,p$ and $\theta_{0}$.
\end{Pros}

\begin{proof}
We first claim that \eqref{pre1y348} is true for some $s>2$.
Thanks to \eqref{pre1y344}, we have
\begin{align}\label{pre1y349}
 &\frac{1}{2}\frac{d}{dt}\|\Lambda^{s} {\theta}(t)\|_{L^{2}}^{2}
+\|\Lambda_{x_{1}}^{\alpha}\Lambda^{s} {\theta}\|_{L^{2}}^{2}+
 \|\Lambda_{x_{2}}^{\beta}\Lambda^{s} {\theta}\|_{L^{2}}^{2}
 \nonumber\\ &\leq C\|\nabla \theta\|_{L^{q}} \|\Lambda^{s+\frac{1}{q}}\theta\|_{L^{2}}^{2}\nonumber\\ &\leq C\|\nabla \theta\|_{L^{q}} (\|\Lambda_{x_{1}}^{\frac{1}{q}}\Lambda^{s}\theta\|_{L^{2}}^{2}
 +\|\Lambda_{x_{2}}^{\frac{1}{q}}\Lambda^{s}\theta\|_{L^{2}}^{2}).
\end{align}
By means of the following interpolation inequalities (see \eqref{avbnuy8911})
$$\|\Lambda_{x_{1}}^{\frac{1}{q}}\Lambda^{s}\theta\|_{L^{2}}\leq C\|\Lambda^{s}\theta\|_{L^{2}}^{\frac{\alpha q-1}{\alpha q}}
\|\Lambda_{x_{1}}^{\alpha}\Lambda^{s}\theta\|_{L^{2}}^{\frac{1}{\alpha q}},\qquad \|\Lambda_{x_{2}}^{\frac{1}{q}}\Lambda^{s}\theta\|_{L^{2}}\leq C\|\Lambda^{s}\theta\|_{L^{2}}^{\frac{\beta q-1}{\beta q}}
\|\Lambda_{x_{2}}^{\beta}\Lambda^{s}\theta\|_{L^{2}}^{\frac{1}{\beta q}}$$
with $q>\max\{\frac{1}{\alpha},\, \frac{1}{\beta}\}$,
we get from \eqref{pre1y349} that
\begin{align}
&\frac{1}{2}\frac{d}{dt}\|\Lambda^{s} {\theta}(t)\|_{L^{2}}^{2}
+\|\Lambda_{x_{1}}^{\alpha}\Lambda^{s} {\theta}\|_{L^{2}}^{2}+
 \|\Lambda_{x_{2}}^{\beta}\Lambda^{s} {\theta}\|_{L^{2}}^{2}\nonumber\\&\leq C\|\nabla \theta\|_{L^{q}} (\|\Lambda^{s}\theta\|_{L^{2}}^{\frac{2\alpha q-2}{\alpha q}}
\|\Lambda_{x_{1}}^{\alpha}\Lambda^{s}\theta\|_{L^{2}}^{\frac{2}{\alpha q}}
 +\|\Lambda^{s}\theta\|_{L^{2}}^{\frac{2\beta q-2}{\beta q}}
\|\Lambda_{x_{2}}^{\beta}\Lambda^{s}\theta\|_{L^{2}}^{\frac{2}{\beta q}})
\nonumber\\&\leq \frac{1}{2}\|\Lambda_{x_{1}}^{\alpha}\Lambda^{s} {\theta}\|_{L^{2}}^{2}+\frac{1}{2}
 \|\Lambda_{x_{2}}^{\beta}\Lambda^{s} {\theta}\|_{L^{2}}^{2}+ C\|\nabla \theta\|_{L^{q}}^{\frac{\alpha q}{\alpha q-1}} \|\Lambda^{s}{\theta}\|_{L^{2}}^{2}
+C\|\nabla \theta\|_{L^{q}}^{\frac{\beta q}{\beta q-1}} \|\Lambda^{s}{\theta}\|_{L^{2}}^{2}.\nonumber
\end{align}
Hence, one immediately obtains
\begin{align}\label{pre1y350}
 \frac{d}{dt}\|\Lambda^{s} {\theta}(t)\|_{L^{2}}^{2}
+\|\Lambda_{x_{1}}^{\alpha}\Lambda^{s} {\theta}\|_{L^{2}}^{2}+
 \|\Lambda_{x_{2}}^{\beta}\Lambda^{s} {\theta}\|_{L^{2}}^{2}\leq C(\|\nabla \theta\|_{L^{q}}^{\frac{\alpha q}{\alpha q-1}}+\|\nabla \theta\|_{L^{q}}^{\frac{\beta q}{\beta q-1}} )\|\Lambda^{s} {\theta}\|_{L^{2}}^{2}.
\end{align}
By \eqref{pre1y315} and \eqref{pre1y343}, we show that
$$\|\nabla \theta\|_{L^{q}}\leq C\|\theta\|_{L^{2}}^{\frac{1}{ {q}}}\|\Delta\theta\|_{L^{2}}
^{\frac{q-1}{ {q}}}\leq C(1+t)^{-\frac{(\alpha+\beta)(2-p)}{4\alpha\beta p}-\frac{\min\{\alpha,\,\beta\}(q-1)}{\alpha\beta q}}\leq C(1+t)^{-\frac{\min\{\alpha,\,\beta\}(q-1)}{\alpha\beta q}}.$$
Combining this and \eqref{pre1y350} gives
\begin{align}\label{pre1y351}
 \frac{d}{dt}\|\Lambda^{s} {\theta}(t)\|_{L^{2}}^{2}
+\|\Lambda_{x_{1}}^{\alpha}\Lambda^{s} {\theta}\|_{L^{2}}^{2}+
 \|\Lambda_{x_{2}}^{\beta}\Lambda^{s} {\theta}\|_{L^{2}}^{2}\leq C(1+t)^{-\Omega}  \|\Lambda^{s} {\theta}\|_{L^{2}}^{2},
\end{align}
where $\Omega$ is given by
$$\Omega\triangleq\frac{\min\{\alpha,\,\beta\}}{\alpha\beta }
\min\left\{\frac{\alpha (q-1)}{\alpha q-1},\,\frac{\beta (q-1)}{\beta q-1}\right\}.$$
Obviously, we can check that $\Omega>1$.
We denote
$$E(t)=\{\xi\in \mathbb{R}^2: \  |\xi_{1}|^{2\alpha}+|\xi_{2}|^{2\beta}\leq f(t)\},$$
where $f(t)$ is a smooth decreasing function to be fixed after.
Similar to \eqref{pre1y323}, one derives
\begin{align}
\|\Lambda_{x_{1}}^{\alpha}\Lambda^{s} {\theta}\|_{L^{2}}^{2}+
 \|\Lambda_{x_{2}}^{\beta}\Lambda^{s} {\theta}\|_{L^{2}}^{2}
  \geq f(t)\int_{\mathbb{R}^{2}}|\widehat{\Lambda^{s} {\theta}}(t,\xi)|^{2}\,d\xi
 -f(t)\int_{E(t)}|\widehat{\Lambda^{s} {\theta}}(t,\xi)|^{2}\,d\xi,\nonumber
\end{align}
which together with \eqref{pre1y351} shows
\begin{align}\label{pre1y352}
\frac{d}{dt}\|\Lambda^{s}{\theta}(t)\|_{L^{2}}^{2}+f(t)\|\Lambda^{s} {\theta}\|_{L^{2}}^{2}\leq& f(t)\int_{E(t)}|\widehat{\Lambda^{s} {\theta}}(t,\xi)|^{2}\,d\xi+C(1+t)^{-\Omega}  \|\Lambda^{s} {\theta}\|_{L^{2}}^{2}.
\end{align}
By direct computations, we may check that
\begin{align}\label{pre1y353}
\int_{E(t)}|\widehat{\Lambda^{s} {\theta}}(t,\xi)|^{2}\,d\xi
&\leq
\int_{E(t)}
|\xi|^{2s}|\widehat{\theta}(\xi)|^{2}\,d\xi\nonumber\\
&\leq
\int_{E(t)}
(|\xi_{1}|^{2}+|\xi_{2}|^{2})^{s}|\widehat{\theta}(\xi)|^{2}\,d\xi\nonumber\\
&\leq C\int_{E(t)}
\left(f(t)^{\frac{s}{\alpha}}+f(t)^{\frac{s}{\beta}}\right)|\widehat{\theta}(\xi)|^{2}\,d\xi
\nonumber\\
&\leq C\left(f(t)^{\frac{s}{\alpha}}+f(t)^{\frac{s}{\beta}}\right)\|{\theta}\|_{L^{2}}^{2}.
\end{align}
Putting \eqref{pre1y353} into \eqref{pre1y352} leads to
\begin{align}\label{pre1y354}
\frac{d}{dt}\|\Lambda^{s}{\theta}(t)\|_{L^{2}}^{2}+f(t)\|\Lambda^{s} {\theta}\|_{L^{2}}^{2}\leq& C\left(f(t)^{1+\frac{s}{\alpha}}+f(t)^{1+\frac{s}{\beta}}\right)\|{\theta}\|_{L^{2}}^{2}
\nonumber\\&+C(1+t)^{-\Omega}  \|\Lambda^{s} {\theta}\|_{L^{2}}^{2}\nonumber\\ \leq& C\left(f(t)^{1+\frac{s}{\alpha}}+f(t)^{1+\frac{s}{\beta}}\right)(1+t)^{-\frac{(\alpha+\beta)(2-p) }{2\alpha\beta p}}\nonumber\\&+C(1+t)^{-\Omega}  \|\Lambda^{s} {\theta}\|_{L^{2}}^{2}.
\end{align}
Taking $f(t)=\frac{\varrho}{1+t}$ with $\varrho$ suitable large, we deduce from \eqref{pre1y354} that
\begin{align}
\frac{d}{dt}\|\Lambda^{s}{\theta}(t)\|_{L^{2}}^{2}+\frac{\varrho}{1+t}\|\Lambda^{s} {\theta}\|_{L^{2}}^{2}\leq& C(1+t)^{-1-s\min\{\frac{1}{\alpha},\,\frac{1}{\beta}\}-\frac{(\alpha+\beta)(2-p) }{2\alpha\beta p}}\nonumber\\&+\widetilde{C}(1+t)^{-\Omega}  \|\Lambda^{s} {\theta}\|_{L^{2}}^{2}.\nonumber
\end{align}
Due to $\Omega>1$, we thus get
\begin{align}
\frac{d}{dt}\|\Lambda^{s}{\theta}(t)\|_{L^{2}}^{2}+\frac{\varrho-\widetilde{C}}{1+t}\|\Lambda^{s} {\theta}\|_{L^{2}}^{2}\leq& C(1+t)^{-1-\frac{\min\{\alpha,\,\beta\}s}{\alpha\beta}-\frac{(\alpha+\beta)(2-p) }{2\alpha\beta p}}.\nonumber
\end{align}
As a result, one derives
\begin{align}\label{tyxcv55}
\frac{d}{dt}\left((1+t)^{\varrho-\widetilde{C}}\|\Lambda^{s}{\theta}(t)\|_{L^{2}}^{2}\right)\leq& C(1+t)^{\varrho-\widetilde{C}-1-\frac{\min\{\alpha,\,\beta\}s}{\alpha\beta}-\frac{(\alpha+\beta)(2-p) }{2\alpha\beta p}}.
\end{align}
Integrating \eqref{tyxcv55} in time, we obtain
\begin{align}
(1+t)^{\varrho-\widetilde{C}}\|\Lambda^{s}{\theta}(t)\|_{L^{2}}^{2}\leq& \|\Lambda^{s}{\theta_{0}}\|_{L^{2}}^{2}+ C(1+t)^{\varrho-\widetilde{C}-\frac{\min\{\alpha,\,\beta\}s}{\alpha\beta} -\frac{(\alpha+\beta)(2-p) }{2\alpha\beta p}},\nonumber
\end{align}
which implies
\begin{align}
\|\Lambda^{s}{\theta}(t)\|_{L^{2}} \leq&  C(1+t)^{-\frac{\min\{\alpha,\,\beta\}s}{2\alpha\beta}-\frac{(\alpha+\beta)(2-p)}{4\alpha\beta p}},\ \ \forall\, s\geq 2.\nonumber
\end{align}
For the case $0\leq s<2$, it follows from the direct interpolation that
\begin{align}
\|\Lambda^{s}{\theta}(t)\|_{L^{2}} &\leq C\|{\theta}(t)\|_{L^{2}} ^{1-\frac{s}{2}} \|\Lambda^{2}{\theta}(t)\|_{L^{2}}^{\frac{s}{2}}
\nonumber\\
&\leq C\left((1+t)^{-\frac{(\alpha+\beta)(2-p)}{4\alpha\beta p}}\right)^{1-\frac{s}{2}} \left((1+t)^{-\frac{\min\{\alpha,\,\beta\}}{\alpha\beta}-\frac{(\alpha+\beta)(2-p)}{4\alpha\beta p}}\right)^{\frac{s}{2}}
\nonumber\\
&=
C(1+t)^{-\frac{\min\{\alpha,\,\beta\}s}{2\alpha\beta}-\frac{(\alpha+\beta)(2-p)}{4\alpha\beta p}}.\nonumber
\end{align}
We thus obtain the desired decay estimate \eqref{pre1y348}. Consequently, the proof of Proposition \ref{dfdsfp37} is completed.
\end{proof}

\vskip .3in
\section{The proof of Theorem \ref{Lderq1}}\setcounter{equation}{0}
\label{dfcse11}
The proof of Theorem \ref{Lderq1} can be performed as that for the proof of Theorem \ref{OKTh1}. We sketch the proof as follows.
\vskip .1in
We first show the following key $\dot{H}^{1}$-decay estimate, whose proof is different from Proposition \ref{Lpp302} and Proposition \ref{Lpp30we3}.
\begin{Pros}\label{Thrcas11}
If $\alpha$ and $\beta$ satisfy (\ref{sdf2334}) and $\theta_{0}\in {H}^{1}(\mathbb{R}^{2})\cap L^{\infty}(\mathbb{R}^{2})$, then the solution admits the decay estimate
\begin{align}\label{Thcwq001}
 \|\nabla\theta(t)\|_{L^2}\leq C_{0}(1+t)^{-\frac{\min\{\alpha,\,\beta\}}{2\alpha\beta }},
\end{align}
where the constant $C_{0}>0$ depends only on $\alpha,\,\beta$ and $\theta_{0}$.
\end{Pros}
\begin{proof}
We divide the proof into two cases, namely,
$$\textbf{Case 1:}\quad  \beta> \frac{1}{2\alpha+1}, \ \ 0<\alpha\leq\frac{1}{2},$$
$$\textbf{Case 2:}\quad  \beta>\frac{1-\alpha}{2\alpha}, \ \ \frac{1}{2}<\alpha<1.$$
We begin with \textbf{Case 1}.
Just as the proof of Proposition \ref{Lpp302}, we have
\begin{align}\label{Thcwq002}
\frac{1}{2}\frac{d}{dt}\|\nabla \theta(t)\|_{L^{2}}^{2}+
 \|\Lambda_{x_{1}}^{\alpha}\nabla\theta\|_{L^{2}}^{2}+
 \|\Lambda_{x_{2}}^{\beta}\nabla\theta\|_{L^{2}}^{2}
 =& \mathcal{J}_{1}+\mathcal{J}_{2}+\mathcal{J}_{3}+\mathcal{J}_{4}.
\end{align}
According to \eqref{udfr5}, one gets
\begin{align} \label{Thcwq003}
\mathcal{J}_{1}  &\leq C\|\Lambda^{\frac{2}{q}}\Lambda_{x_{2}}^{\delta}\partial_{x_{1}}\theta\|_{L^{2}} \|\theta\|_{L^{q}} \|\Lambda_{x_{2}}^{\beta} \nabla\theta\|_{L^{2}}^{\frac{\alpha}{\alpha+1}}\|\Lambda_{x_{1}}^{\alpha} \nabla\theta\|_{L^{2}}^{\frac{1}{\alpha+1}}.
\end{align}
If we take $\epsilon=0$ in \eqref{udfr7}, \eqref{udfr6} and \eqref{udfr8}, then we have
$$\|\Lambda^{\frac{2}{q}}\Lambda_{x_{2}}^{\delta}\partial_{x_{1}}\theta\|_{L^{2}}\leq C\|\Lambda_{x_{2}}^{\beta}\nabla\theta\|_{L^{2}}^{\frac{\delta}{\beta}}
 \|\Lambda_{x_{1}}^{\alpha}\nabla\theta\|_{L^{2}}
 ^{\frac{\beta-\delta}{\beta}},$$
where
$$q=\frac{2(\alpha+1)\beta}{\alpha[(2\alpha+1)\beta-1]}\geq2.$$
This along with \eqref{Thcwq003} shows
\begin{align} \label{Thcwq004}
\mathcal{J}_{1}   &\leq C \|\theta\|_{L^{q}} \|\Lambda_{x_{2}}^{\beta} \nabla\theta\|_{L^{2}}\|\Lambda_{x_{1}}^{\alpha} \nabla\theta\|_{L^{2}}\nonumber\\
&\leq C \|\theta\|_{L^{q}} (\|\Lambda_{x_{2}}^{\beta} \nabla\theta\|_{L^{2}}^{2}+\|\Lambda_{x_{1}}^{\alpha} \nabla\theta\|_{L^{2}}^{2}).
\end{align}
Moreover, $\mathcal{J}_{2}$ has the same bound of $\mathcal{J}_{1}$, namely,
\begin{align} \label{Thcwq005}
\mathcal{J}_{2}   \leq C \|\theta\|_{L^{q}} \|\Lambda_{x_{2}}^{\beta} \nabla\theta\|_{L^{2}}\|\Lambda_{x_{1}}^{\alpha} \nabla\theta\|_{L^{2}}.
\end{align}
We observe
\begin{align}
\mathcal{J}_{3} =-\frac{1}{2}\mathcal{J}_{4}+\widetilde{\mathcal{J}_{3}},\nonumber
\end{align}
and
\begin{align} \label{Thcwq006}
\widetilde{\mathcal{J}_{3}}
 \leq& C\|\Lambda^{\frac{2}{\widetilde{q}}}\Lambda_{x_{2}}^{\delta}\partial_{x_{2}}
\theta\|_{L^{2}}
 \|\theta\|_{L^{\widetilde{q}}} \|\Lambda_{x_{2}}^{\beta} \nabla\theta\|_{L^{2}}^{\frac{\alpha}{\alpha+1}}\|\Lambda_{x_{1}}^{\alpha} \nabla\theta\|_{L^{2}}^{\frac{1}{\alpha+1}}\nonumber\\
 \leq& C( \|\Lambda_{x_{1}}^{\frac{2}{\widetilde{q}}}\Lambda_{x_{2}}^{\delta+1} \theta\|_{L^{2}} +\|\Lambda_{x_{2}}^{1+\delta+\frac{2}{\widetilde{q}}} \theta\|_{L^{2}})
 \|\theta\|_{L^{\widetilde{q}}} \|\Lambda_{x_{2}}^{\beta} \nabla\theta\|_{L^{2}}^{\frac{\alpha}{\alpha+1}}\|\Lambda_{x_{1}}^{\alpha} \nabla\theta\|_{L^{2}}^{\frac{1}{\alpha+1}}
 \nonumber\\
 \leq& C(\|\Lambda_{x_{2}}^{\beta}\nabla\theta\|_{L^{2}} +\|\Lambda_{x_{2}}^{1+\beta} \theta\|_{L^{2}})
 \|\theta\|_{L^{\widetilde{q}}} \|\Lambda_{x_{2}}^{\beta} \nabla\theta\|_{L^{2}}^{\frac{\alpha}{\alpha+1}}\|\Lambda_{x_{1}}^{\alpha} \nabla\theta\|_{L^{2}}^{\frac{1}{\alpha+1}} \nonumber\\
 \leq& C\|\Lambda_{x_{2}}^{\beta}\nabla\theta\|_{L^{2}}
 \|\theta\|_{L^{\widetilde{q}}} \|\Lambda_{x_{2}}^{\beta} \nabla\theta\|_{L^{2}}^{\frac{\alpha}{\alpha+1}}\|\Lambda_{x_{1}}^{\alpha} \nabla\theta\|_{L^{2}}^{\frac{1}{\alpha+1}}
 \nonumber\\
\leq &C \|\theta\|_{L^{\widetilde{q}}} (\|\Lambda_{x_{2}}^{\beta} \nabla\theta\|_{L^{2}}^{2}+\|\Lambda_{x_{1}}^{\alpha} \nabla\theta\|_{L^{2}}^{2})
\end{align}
where
$$\widetilde{q}=\frac{2(\alpha+1) }{ (2\alpha+1)\beta-1}\geq2.$$
Finally, the term $\mathcal{J}_{4}$ can be bounded by
\begin{align}\label{Thcwq007}
\mathcal{J}_{4}
 &\leq C\|\Lambda^{\frac{2}{\widehat{q}}}\Lambda_{x_{2}}^{1-\beta}\partial_{x_{2}}\theta
 \|_{L^{2}}   \|\theta\|_{L^{\widehat{q}}} \|\Lambda_{x_{2}}^{\beta}\nabla\theta\|_{L^{2}}\nonumber\\
 &\leq C(\|\Lambda_{x_{1}}^{\frac{2}{\widehat{q}}}\Lambda_{x_{2}}^{1-\beta}\partial_{x_{2}}\theta
 \|_{L^{2}}+\|\Lambda_{x_{2}}^{\frac{2}{\widehat{q}}}\Lambda_{x_{2}}^{1-\beta}\partial_{x_{2}}\theta
 \|_{L^{2}})\|\theta\|_{L^{\widehat{q}}} \|\Lambda_{x_{2}}^{\beta}\nabla\theta\|_{L^{2}}
 \nonumber\\
 &\leq C(\|\Lambda_{x_{2}}^{\beta}\nabla\theta
 \|_{L^{2}}+\|\Lambda_{x_{2}}^{1+\beta} \theta
 \|_{L^{2}})\|\theta\|_{L^{\widehat{q}}} \|\Lambda_{x_{2}}^{\beta}\nabla\theta\|_{L^{2}}\nonumber\\
 &\leq C  \|\theta\|_{L^{\widehat{q}}} \|\Lambda_{x_{2}}^{\beta}\nabla\theta\|_{L^{2}}^{2},
\end{align}
where
$$\widehat{q}=\frac{2}{2\beta-1}\geq2.$$
Inserting \eqref{Thcwq004}, \eqref{Thcwq005}, \eqref{Thcwq006} and \eqref{Thcwq007} into \eqref{Thcwq002} yields
 \begin{align}\label{Thcwq008}
\frac{1}{2}\frac{d}{dt}\|\nabla \theta(t)\|_{L^{2}}^{2}+
 \|\Lambda_{x_{1}}^{\alpha}\nabla\theta\|_{L^{2}}^{2}+
 \|\Lambda_{x_{2}}^{\beta}\nabla\theta\|_{L^{2}}^{2}
 \leq C \mathcal{H}(t) (\|\Lambda_{x_{2}}^{\beta} \nabla\theta\|_{L^{2}}^{2}+\|\Lambda_{x_{1}}^{\alpha} \nabla\theta\|_{L^{2}}^{2}),
\end{align}
where
$$\mathcal{H}(t)\triangleq\|\theta(t)\|_{L^{q}}+\|\theta(t)\|_{L^{\widetilde{q}}}
+\|\theta(t)\|_{L^{\widehat{q}}}.$$

Now our attention focus on \textbf{Case 2}.
Actually, we just modify the proof of Proposition \ref{Lpp30we3}. It follows from \eqref{ghtyde229} that
\begin{align}  \label{Thcwq009}
\mathcal{J}_{1}
 &\leq C\|\Lambda^{\frac{2}{r}}\Lambda_{x_{1}}^{1-\alpha}\partial_{x_{1}}
\theta\|_{L^{2}} \|\theta\|_{L^{r}} \|\Lambda_{x_{1}}^{\alpha}\partial_{x_{1}}\theta\|_{L^{2}} \nonumber\\
&\leq C(\|\Lambda_{x_{1}}^{2-\alpha+\frac{2}{r}} \theta\|_{L^{2}} + \|\Lambda_{x_{1}}^{2-\alpha}\Lambda_{x_{2}}^{\frac{2}{r}} \theta\|_{L^{2}})
 \|\theta\|_{L^{r}} \|\Lambda_{x_{1}}^{\alpha}\partial_{x_{1}}\theta\|_{L^{2}}
\nonumber\\
&\leq C(\|\Lambda_{x_{1}}^{1+\alpha} \theta\|_{L^{2}} + \|\Lambda_{x_{1}}^{\alpha}\nabla \theta\|_{L^{2}})
  \|\theta\|_{L^{r}} \|\Lambda_{x_{1}}^{\alpha}\partial_{x_{1}}\theta\|_{L^{2}}
  \nonumber\\
&\leq C \|\theta\|_{L^{r}} \|\Lambda_{x_{1}}^{\alpha}\nabla \theta\|_{L^{2}}^{2},
\end{align}
where
$$r=\frac{2}{2\alpha-1}\geq2.$$
We notice that
\begin{align}
\mathcal{J}_{2}=-\frac{1}{2}\mathcal{J}_{1}+\widetilde{\mathcal{J}_{2}}\nonumber
\end{align}
and
\begin{align} \label{Thcwq010}
\widetilde{\mathcal{J}_{2}}
 &\leq C\|\Lambda^{\frac{2}{\widetilde{r}}}
 \Lambda_{x_{1}}^{1-\widetilde{\delta}}\partial_{x_{1}}
 \theta\|_{L^{2}}  \|\theta\|_{L^{\widetilde{r}}}\|\Lambda_{x_{1}}^{\alpha+1}\theta
\|_{L^{2}}^{\frac{\beta}{\beta+1}} \|\Lambda_{x_{2}}^{\beta+1}\theta
\|_{L^{2}}^{\frac{1}{\beta+1}}\nonumber\\
 &\leq C(\|\Lambda_{x_{1}}^{2-\widetilde{\delta}+\frac{2}{\widetilde{r}}} \theta\|_{L^{2}} +\|\Lambda_{x_{1}}^{2-\widetilde{\delta}}\Lambda_{x_{2}}^{\frac{2}{\widetilde{r}}} \theta\|_{L^{2}}) \|\theta\|_{L^{\widetilde{r}}}\|\Lambda_{x_{1}}^{\alpha+1}\theta
\|_{L^{2}}^{\frac{\beta}{\beta+1}} \|\Lambda_{x_{2}}^{\beta+1}\theta
\|_{L^{2}}^{\frac{1}{\beta+1}}
\nonumber\\
 &\leq C(\|\Lambda_{x_{1}}^{1+\alpha} \theta\|_{L^{2}} +\|\Lambda_{x_{1}}^{\alpha}\nabla \theta\|_{L^{2}}) \|\theta\|_{L^{\widetilde{r}}}\|\Lambda_{x_{1}}^{\alpha+1}\theta
\|_{L^{2}}^{\frac{\beta}{\beta+1}} \|\Lambda_{x_{2}}^{\beta+1}\theta
\|_{L^{2}}^{\frac{1}{\beta+1}}\nonumber\\
 &\leq C \|\theta\|_{L^{\widetilde{r}}}\|\Lambda_{x_{1}}^{\alpha}\nabla\theta
\|_{L^{2}}^{1+\frac{\beta}{\beta+1}} \|\Lambda_{x_{2}}^{\beta}\nabla\theta
\|_{L^{2}}^{\frac{1}{\beta+1}}\nonumber\\
 &\leq C \|\theta\|_{L^{\widetilde{r}}}(\|\Lambda_{x_{1}}^{\alpha}\nabla\theta
\|_{L^{2}}^{2}+\|\Lambda_{x_{2}}^{\beta}\nabla\theta
\|_{L^{2}}^{2}),
\end{align}
where
$$\widetilde{r}=\frac{2(\beta+1)}{2\alpha\beta+\alpha-1}\geq2.$$
By taking $\widetilde{\epsilon}=0$ in \eqref{ghtyde236}, the third term $\mathcal{J}_{3} $ can be bounded by
\begin{align}\label{Thcwq011}
\mathcal{J}_{3} =&\int_{\mathbb{R}^{2}}{\theta\partial_{x_{1}x_{2}}u_{1}\partial_{x_{2}}\theta\,dx}
+\int_{\mathbb{R}^{2}}{\theta\partial_{x_{2}}u_{1}\partial_{x_{1}x_{2}}\theta\,dx}
\nonumber\\
 \leq&C\|\Lambda_{x_{1}}^{1-\widetilde{\delta}}\partial_{x_{2}}u_{1}
 \|_{L^{\frac{2\widehat{r}}{\widehat{r}-2}}} \|\Lambda_{x_{1}}^{\widetilde{\delta}}(\theta\partial_{x_{2}}\theta)
 \|_{L^{\frac{2\widehat{r}}{\widehat{r}+2}}}
+C\|\Lambda_{x_{1}}^{1-\widetilde{\delta}}\partial_{x_{2}}\theta
\|_{L^{\frac{2\widehat{r}}{\widehat{r}-2}}} \|\Lambda_{x_{1}}^{\widetilde{\delta}}(\theta\partial_{x_{2}}u_{1})
\|_{L^{\frac{2\widehat{r}}{\widehat{r}+2}}}
\nonumber\\
 \leq& C\|\Lambda^{\frac{2}{\widehat{r}}}
 \Lambda_{x_{1}}^{1-\widetilde{\delta}}\partial_{x_{2}}
 \theta\|_{L^{2}}  \|\theta\|_{L^{\widehat{r}}}\|\Lambda_{x_{1}}^{\alpha+1}\theta
\|_{L^{2}}^{\frac{\beta}{\beta+1}} \|\Lambda_{x_{2}}^{\beta+1}\theta
\|_{L^{2}}^{\frac{1}{\beta+1}}\nonumber\\
 \leq& C\|\Lambda_{x_{2}}^{\beta}\nabla\theta\|_{L^{2}}^{1-\frac{
1-\alpha\beta}{\alpha(\beta+1)}}
 \|\Lambda_{x_{1}}^{\alpha}\nabla\theta\|_{L^{2}}
 ^{\frac{1-\alpha\beta}{\alpha(\beta+1)}}  \|\theta\|_{L^{\widehat{r}}}\|\Lambda_{x_{1}}^{\alpha+1}\theta
\|_{L^{2}}^{\frac{\beta}{\beta+1}} \|\Lambda_{x_{2}}^{\beta+1}\theta
\|_{L^{2}}^{\frac{1}{\beta+1}}\nonumber\\
\leq &C \|\theta\|_{L^{\widehat{r}}}(\|\Lambda_{x_{1}}^{\alpha}\nabla\theta
\|_{L^{2}}^{2}+\|\Lambda_{x_{2}}^{\beta}\nabla\theta
\|_{L^{2}}^{2}),
\end{align}
where
\begin{align}
\widehat{r}=\frac{2\alpha(\beta+1)}{(2\alpha\beta+\alpha-1)\beta}\geq2.\nonumber
\end{align}
Using the same argument adopted in dealing with (\ref{Thcwq011}), we have
\begin{align}\label{Thcwq012}
\mathcal{J}_{4} =&\int_{\mathbb{R}^{2}}{\partial_{x_{1}}u_{1} \partial_{x_{2}}\theta\partial_{x_{2}}\theta\,dx}
\nonumber\\
 =&-2
\int_{\mathbb{R}^{2}}{ u_{1} \partial_{x_{2}}\theta\partial_{x_{1}x_{2}}\theta\,dx}
\nonumber\\
 \leq&
C\|\Lambda_{x_{1}}^{1-\widetilde{\delta}}\partial_{x_{2}}
\theta\|_{L^{\frac{2\widehat{r}}{\widehat{r}-2}}} \|\Lambda_{x_{1}}^{\widetilde{\delta}}(u_{1}\partial_{x_{2}}
\theta)\|_{L^{\frac{2\widehat{r}}{\widehat{r}+2}}}
\nonumber\\
\leq &C \|\theta\|_{L^{\widehat{r}}}(\|\Lambda_{x_{1}}^{\alpha}\nabla\theta
\|_{L^{2}}^{2}+\|\Lambda_{x_{2}}^{\beta}\nabla\theta
\|_{L^{2}}^{2}).
\end{align}
Combining \eqref{Thcwq009}, \eqref{Thcwq010}, \eqref{Thcwq011}, \eqref{Thcwq012} and \eqref{Thcwq002} implies
 \begin{align}\label{Thcwq013}
\frac{1}{2}\frac{d}{dt}\|\nabla \theta(t)\|_{L^{2}}^{2}+
 \|\Lambda_{x_{1}}^{\alpha}\nabla\theta\|_{L^{2}}^{2}+
 \|\Lambda_{x_{2}}^{\beta}\nabla\theta\|_{L^{2}}^{2}
 \leq C \mathcal{\widetilde{H}}(t) (\|\Lambda_{x_{2}}^{\beta} \nabla\theta\|_{L^{2}}^{2}+\|\Lambda_{x_{1}}^{\alpha} \nabla\theta\|_{L^{2}}^{2}),
\end{align}
where
$$\mathcal{\widetilde{H}}(t)\triangleq\|\theta(t)\|_{L^{r}}
+\|\theta(t)\|_{L^{\widetilde{r}}}
+\|\theta(t)\|_{L^{\widehat{r}}}.$$
Consequently, it follows from \eqref{Thcwq009} and \eqref{Thcwq013} that
 \begin{align}\label{Thcwq014}
 \frac{1}{2}\frac{d}{dt}\|\nabla \theta(t)\|_{L^{2}}^{2}+
 \|\Lambda_{x_{1}}^{\alpha}\nabla\theta\|_{L^{2}}^{2}+
 \|\Lambda_{x_{2}}^{\beta}\nabla\theta\|_{L^{2}}^{2}
 \leq C \Gamma(t) (\|\Lambda_{x_{2}}^{\beta} \nabla\theta\|_{L^{2}}^{2}+\|\Lambda_{x_{1}}^{\alpha} \nabla\theta\|_{L^{2}}^{2}),
\end{align}
where $\alpha$ and $\beta$ satisfy (\ref{sdf2334}) and $\Gamma(t)$ is given by
$$\Gamma(t)=\max\{\mathcal{ {H}}(t),\,\mathcal{\widetilde{H}}(t)\}.$$
Notice that the fact
$$\|\theta(t)\|_{L^{\infty}}\leq \|\theta_0\|_{L^{\infty}},$$
we deduce
$$\Gamma(t)\leq C\|\theta(t)\|_{L^{2}}^{\gamma}\|\theta(t)\|_{L^{\infty}}^{1-\gamma}\leq C\|\theta(t)\|_{L^{2}}^{\gamma}\|\theta_0\|_{L^{\infty}}^{1-\gamma},\quad \gamma\in (0,1).$$
Keeping in mind of \eqref{cvbmklp1}, there exists a $t_{0}>0$ such that
$$C \Gamma(t)\leq \frac{1}{2},\quad \forall\, t\geq t_{0}.$$
As a result, we deduce from \eqref{Thcwq014} that
\begin{align}\label{Thcwq015}
 \frac{d}{dt}\|\nabla \theta(t)\|_{L^{2}}^{2}+
 \|\Lambda_{x_{1}}^{\alpha}\nabla\theta\|_{L^{2}}^{2}+
 \|\Lambda_{x_{2}}^{\beta}\nabla\theta\|_{L^{2}}^{2}
 \leq 0,\quad \forall\, t\geq t_{0}.
\end{align}
According to the proof of Proposition \ref{dfdsfp37}, we conclude from \eqref{Thcwq015} that
\begin{align}
 \|\nabla\theta(t)\|_{L^2}\leq C_{0}(1+t)^{-\frac{\min\{\alpha,\,\beta\}}{2\alpha\beta }},\quad \forall\, t\geq t_{0}.\nonumber
\end{align}
As proved in Theorem \ref{OKTh1} that  (\ref{SQG}) has a unique global solution as long as $\alpha$ and $\beta$ satisfy (\ref{sdf2334}), the above decay estimate actually holds true for any $t\geq0$, namely,
\begin{align}
 \|\nabla\theta(t)\|_{L^2}\leq C_{0}(1+t)^{-\frac{\min\{\alpha,\,\beta\}}{2\alpha\beta }},\nonumber
\end{align}
which is \eqref{Thcwq001}.
By the way, it also follows from \eqref{Thcwq015} that for any $t\geq t_{0}$,
\begin{align}\label{Tbnmuy1}
\int_{t_{0}}^{t}(\|\Lambda_{x_{1}}^{\alpha}\nabla{\theta}(t)\|_{L^{2}}^{2}+
 \|\Lambda_{x_{2}}^{\beta}\nabla{\theta}(t)\|_{L^{2}}^{2})\leq \|\nabla\theta(t_{0})\|_{L^2}.
\end{align}
Using again the fact that (\ref{SQG}) has a unique global solution as long as $\alpha$ and $\beta$ satisfy (\ref{sdf2334}), we thus obtain from \eqref{Tbnmuy1} that
\begin{align}\label{Thcwq016}
\int_{0}^{\infty}(\|\Lambda_{x_{1}}^{\alpha}\nabla{\theta}(t)\|_{L^{2}}^{2}+
 \|\Lambda_{x_{2}}^{\beta}\nabla{\theta}(t)\|_{L^{2}}^{2})\leq C_{0}.
\end{align}
This ends the proof of Proposition \ref{Thrcas11}.
\end{proof}

\vskip .1in
Based on \eqref{Thcwq001} and \eqref{Thcwq016}, we can show that \eqref{pre1y338} of Proposition \ref{Lfgyt789j} still holds true. More precisely, we have
\begin{Pros}
If $\alpha$ and $\beta$ satisfy (\ref{sdf2334}) and $\theta_{0}\in {H}^{2}(\mathbb{R}^{2})$, then it holds
\begin{eqnarray}\label{Thcwq017}
\sup_{t\geq0}\|\Delta\theta(t)\|_{L^{2}}\leq
C_{0},
\end{eqnarray}
where the constant $C_{0}>0$ depends only on $\alpha,\,\beta$ and $\theta_{0}$.
\end{Pros}
\begin{proof}
It follows from \eqref{hjywe350} and \eqref{hjywe351} that
\begin{equation}\label{ghnjkm1}
\frac{d}{dt}\|\Delta\theta(t)\|_{L^{2}}^{2}\leq\left\{\aligned
&C
(\|\nabla\theta\|_{L^2}^{\frac{2\alpha}{2\alpha-1}}+\|\nabla\theta\|_{L^2}
 ^{2}+
 \|\Lambda_{x_{1}}^{\alpha}\nabla\theta\|_{L^2}^{ 2}) \,\|\Delta \theta\|_{L^2}^{2}, \quad \frac{1}{2}<\alpha<1,\\
&C(\|\nabla\theta\|_{L^2}^{\frac{2\beta}{2\beta-1}}+
 \|\nabla\theta\|_{L^2}^{2}+\|\Lambda_{x_{2}}^{\beta}\nabla\theta\|_{L^2}^{2})\|\Delta \theta\|_{L^2}^{2}, \quad \frac{1}{2}<\beta<1.
\endaligned\right.
\end{equation}
Thanks to \eqref{Thcwq001} and \eqref{Thcwq016}, we obtain
\begin{equation*}
\left\{\aligned
&
\int_{0}^{\infty}(\|\nabla\theta(t)\|_{L^2}^{\frac{2\alpha}{2\alpha-1}}
+\|\nabla\theta(t)\|_{L^2}
 ^{2}+
 \|\Lambda_{x_{1}}^{\alpha}\nabla\theta(t)\|_{L^2}^{ 2}) \,dt\leq C_{0}, \quad \frac{1}{2}<\alpha<1,\\
& \int_{0}^{\infty}(\|\nabla\theta(t)\|_{L^2}^{\frac{2\beta}{2\beta-1}}+
 \|\nabla\theta(t)\|_{L^2}^{2}+\|\Lambda_{x_{2}}^{\beta}\nabla\theta(t)\|_{L^2}^{2})
  \,dt\leq C_{0},\quad \frac{1}{2}<\beta<1.
\endaligned\right.
\end{equation*}
Applying this and the Gronwall inequality to \eqref{ghnjkm1} yields \eqref{Thcwq017}.
\end{proof}

\vskip .1in
We point out that with \eqref{Thcwq017} in hand, we are able to show the preliminary decay estimate.
\begin{Pros}
If $\alpha$ and $\beta$ satisfy (\ref{sdf2334}), $\theta_{0}\in \dot{H}^{s}(\mathbb{R}^{2})$ with $s\geq2$.
Then the global smooth solution of \eqref{SQG} admits the decay estimate
\begin{align} \label{Thcwq018}
  \|\Lambda^{s} \theta(t)\|_{L^2}\leq C_{0}(1+t)^{-\frac{\min\{\alpha,\,\beta\}}{2\alpha\beta }},
\end{align}
where the constant $C_{0}>0$ depends only on $\alpha,\,\beta,\,s$ and $\theta_{0}$.
\end{Pros}

\begin{proof}
It follows from \eqref{pre1y346} and \eqref{Thcwq001} that
\begin{align}
\frac{d}{dt}\|\Lambda^{s} {\theta}(t)\|_{L^{2}}^{2}
+ \|\Lambda^{s} {\theta}\|_{L^{2}}^{2}&\leq C\|\nabla\theta\|_{L^{2}}^{2}\leq C(1+t)^{-\frac{\min\{\alpha,\,\beta\}}{2\alpha\beta }}.\nonumber
\end{align}
Following the same argument adopted in dealing with \eqref{pre1y347} yields \eqref{Thcwq018} immediately.
\end{proof}

\begin{Pros}
If $\alpha$ and $\beta$ satisfy (\ref{sdf2334}) and $\theta_{0}\in \dot{H}^{2}(\mathbb{R}^{2})$, then the solution admits the decay estimate
\begin{align}\label{Tnmkdd1}
  \|\Delta\theta(t)\|_{L^2}\leq C_{0}(1+t)^{-\frac{\min\{\alpha,\,\beta\}}{ \alpha\beta }},
\end{align}
where the constant $C_{0}>0$ depends only on $\alpha,\,\beta$ and $\theta_{0}$.
\end{Pros}
\begin{proof}
We deduce from \eqref{mpdfhj347} that for $\alpha>\frac{1}{2}$
\begin{align}
  \frac{d}{dt}\|\Delta\theta(t)\|_{L^{2}}^{2}+\|\Lambda_{x_{1}}^{\alpha}\Delta\theta\|_{L^{2}}^{2}+
 \|\Lambda_{x_{2}}^{\beta}\Delta\theta\|_{L^{2}}^{2}
&\leq C
(\|\nabla\theta\|_{L^2}^{\frac{2\alpha}{2\alpha-1}}+\|\nabla\theta\|_{L^2}
 ^{2}+
 \|\Delta\theta\|_{L^2}^{2 }) \,\|\Delta \theta\|_{L^2}^{2}
 \nonumber\\
&\leq
 C_{0}\left((1+t)^{-\frac{\min\{\alpha,\,\beta\}}{\alpha\beta }\frac{2\alpha}{2\alpha-1}}+(1+t)^{-\frac{\min\{\alpha,\,\beta\}}{\alpha\beta }}\right)\nonumber\\&\quad \times\|\Delta \theta\|_{L^2}^{2}
 \nonumber\\
&\leq
 C_{0}(1+t)^{-\frac{\min\{\alpha,\,\beta\}}{\alpha\beta }} \,\|\Delta \theta\|_{L^2}^{2}.\nonumber
\end{align}
Similarly, \eqref{mpdfhj348} ensures that for $\beta>\frac{1}{2}$
\begin{align}
  \frac{d}{dt}\|\Delta\theta(t)\|_{L^{2}}^{2}+\|\Lambda_{x_{1}}^{\alpha}\Delta\theta\|_{L^{2}}^{2}+
 \|\Lambda_{x_{2}}^{\beta}\Delta\theta\|_{L^{2}}^{2}
&\leq C
(\|\nabla\theta\|_{L^2}^{\frac{2\beta}{2\beta-1}}+\|\nabla\theta\|_{L^2}
 ^{2}+
 \|\Delta\theta\|_{L^2}^{2 }) \,\|\Delta \theta\|_{L^2}^{2}
 \nonumber\\
&\leq
 C_{0}(1+t)^{-\frac{\min\{\alpha,\,\beta\}}{\alpha\beta }} \,\|\Delta \theta\|_{L^2}^{2}.\nonumber
\end{align}
Then for $\alpha$ and $\beta$ satisfying (\ref{sdf2334}), one derives
\begin{align}
  \frac{d}{dt}\|\Delta\theta(t)\|_{L^{2}}^{2}+\|\Lambda_{x_{1}}^{\alpha}\Delta\theta\|_{L^{2}}^{2}+
 \|\Lambda_{x_{2}}^{\beta}\Delta\theta\|_{L^{2}}^{2}
 \leq
 C_{0}(1+t)^{-\frac{\min\{\alpha,\,\beta\}}{\alpha\beta }} \,\|\Delta \theta\|_{L^2}^{2}.\nonumber
\end{align}
Notice that
$$\frac{\min\{\alpha,\,\beta\}}{\alpha\beta }>1,$$
we thus derive \eqref{Tnmkdd1} by using the same arguments used in dealing with \eqref{pre1y351}-\eqref{tyxcv55}.
\end{proof}

At this stage, we are ready to finish the proof of Theorem \ref{Lderq1}.
\begin{Pros}
If $\alpha$ and $\beta$ satisfy (\ref{sdf2334}), $\theta_{0}\in \dot{H}^{s}(\mathbb{R}^{2})$ with $s\geq0$.
Then the global smooth solution of \eqref{SQG} admits the decay estimate
\begin{align}
 \|\Lambda^{s} \theta(t)\|_{L^2}\leq C_{0}(1+t)^{-\frac{ \min\{\alpha,\,\beta\}s }{2\alpha\beta }},\nonumber
\end{align}
where the constant $C_{0}>0$ depends only on $\alpha,\,\beta,\,s$ and $\theta_{0}$.
\end{Pros}

\begin{proof}
We come back to \eqref{pre1y350} to get
\begin{align}
 \frac{d}{dt}\|\Lambda^{s} {\theta}(t)\|_{L^{2}}^{2}
+\|\Lambda_{x_{1}}^{\alpha}\Lambda^{s} {\theta}\|_{L^{2}}^{2}+
 \|\Lambda_{x_{2}}^{\beta}\Lambda^{s} {\theta}\|_{L^{2}}^{2}\leq C(\|\nabla \theta\|_{L^{q}}^{\frac{\alpha q}{\alpha q-1}}+\|\nabla \theta\|_{L^{q}}^{\frac{\beta q}{\beta q-1}} )\|\Lambda^{s} {\theta}\|_{L^{2}}^{2}.\nonumber
\end{align}
By \eqref{Tnmkdd1}, it implies
$$\|\nabla \theta\|_{L^{q}}\leq C\|\theta\|_{L^{2}}^{\frac{1}{ {q}}}\|\Delta\theta\|_{L^{2}}
^{\frac{q-1}{ {q}}}\leq C(1+t)^{ -\frac{\min\{\alpha,\,\beta\}(q-1)}{\alpha\beta q}}\leq C(1+t)^{-\frac{\min\{\alpha,\,\beta\}(q-1)}{\alpha\beta q}}.$$
We hence have
\begin{align}
 \frac{d}{dt}\|\Lambda^{s} {\theta}(t)\|_{L^{2}}^{2}
+\|\Lambda_{x_{1}}^{\alpha}\Lambda^{s} {\theta}\|_{L^{2}}^{2}+
 \|\Lambda_{x_{2}}^{\beta}\Lambda^{s} {\theta}\|_{L^{2}}^{2}\leq C(1+t)^{-\Omega}  \|\Lambda^{s} {\theta}\|_{L^{2}}^{2}.\nonumber
\end{align}
The remainder proof is the same as that for the proof of Proposition \ref{dfdsfp37}.
Consequently, we finish the proof of Theorem \ref{Lderq1}.
\end{proof}

\vskip .3in
\section{The proof of Theorem \ref{okspcrca11} }\setcounter{equation}{0}
\label{secadd1}
In this section, we are going to prove Theorem \ref{okspcrca11}.
Firstly, we apply $\Lambda^{s}$ to \eqref{SQG} and testing the equation by $\Lambda^{s}\theta$ to obtain
\begin{align}\label{sptadd401}
&\frac{1}{2}\frac{d}{dt}\|\Lambda^{s} {\theta}(t)\|_{L^{2}}^{2}
+\|\Lambda_{x_{1}}^{\alpha}\Lambda^{s} {\theta}\|_{L^{2}}^{2}+
 \|\Lambda_{x_{2}}^{\beta}\Lambda^{s} {\theta}\|_{L^{2}}^{2}\nonumber\\&=-\int_{\mathbb{R}^{2}}\Lambda^{s}(u \cdot \nabla\theta) \Lambda^{s}\theta\,dx\nonumber\\
 &=-\int_{\mathbb{R}^{2}}[\Lambda^{s},u]\cdot \nabla\theta  \Lambda^{s}\theta\,dx
 \nonumber\\
 &\leq C\|[\Lambda^{s},u]\cdot \nabla\theta\|_{L^{\frac{2(\alpha+\beta)}{\alpha+\beta+2\alpha\beta}}} \|\Lambda^{s}\theta\|_{L^{\frac{2(\alpha+\beta)}{\alpha+\beta-2\alpha\beta}}}
  \nonumber\\
 &\leq C(\|\nabla u\|_{L^{\frac{\alpha+\beta}{2\alpha\beta}}} \|\Lambda^{s}\theta\|_{L^{\frac{2(\alpha+\beta)}{\alpha+\beta-2\alpha\beta}}}
 +\|\nabla \theta\|_{L^{\frac{\alpha+\beta}{2\alpha\beta}}} \|\Lambda^{s}u\|_{L^{\frac{2(\alpha+\beta)}{\alpha+\beta-2\alpha\beta}}}) \|\Lambda^{s}\theta\|_{L^{\frac{2(\alpha+\beta)}{\alpha+\beta-2\alpha\beta}}}
 \nonumber\\
 &\leq C\|\nabla \theta\|_{L^{\frac{\alpha+\beta}{2\alpha\beta}}} \|\Lambda^{s}\theta\|_{L^{\frac{2(\alpha+\beta)}{\alpha+\beta-2\alpha\beta}}}^{2}.
\end{align}
Now we claim
\begin{align}\label{sptadd402}
 \|f\|_{L^{\frac{2(\alpha+\beta)}{\alpha+\beta-2\alpha\beta}}}\leq C
 \|\Lambda_{x_{1}}^{\alpha}f\|_{L^{2}}^{\frac{\beta}{\alpha+\beta}}
 \|\Lambda_{x_{2}}^{\beta}f\|_{L^{2}}^{\frac{\alpha}{\alpha+\beta}},
\end{align}
where $\alpha>0,\,\beta>0$ satisfy $\alpha+\beta>2\alpha\beta$. To show \eqref{sptadd402}, we recall the following fact due to \eqref{xcvyadi1},
\begin{align}\label{sfvhyt001u}
 \|\Lambda_{x_{2}}^{\delta_{2}}\Lambda_{x_{1}}^{\delta_{1}}f
\|_{L^{2}}\leq C\|\Lambda_{x_{2}}^{\frac{\delta_{2}}{1-\gamma}}f\|_{L^{2}}
^{1-\gamma}
\|\Lambda_{x_{1}}^{\frac{\delta_{1}}{\gamma}}f\|_{L^{2}}
^{\gamma}
\end{align}
where $\delta_{1}>0,\ \delta_{2}>0,\ \gamma\in (0,1)$.
By means of the one-dimensional Sobolev inequality, the Minkowski inequality and \eqref{sfvhyt001u}, we observe
\begin{align}
 \|f\|_{L^{\frac{2(\alpha+\beta)}{\alpha+\beta-2\alpha\beta}}}&= \Big\| \|f\|_{L_{x_{2}}^{\frac{2(\alpha+\beta)}{\alpha+\beta-2\alpha\beta}}}
\Big\|_{L_{x_{1}}^{\frac{2(\alpha+\beta)}{\alpha+\beta-2\alpha\beta}}}
\nonumber\\&\leq C \Big\|  \|\Lambda_{x_{2}}^{\frac{\alpha\beta}{\alpha+\beta}}f\|_{L_{x_{2}}^{2}}
\Big\|_{L_{x_{1}}^{\frac{2(\alpha+\beta)}{\alpha+\beta-2\alpha\beta}}}
\nonumber\\&\leq C \Big\|  \|\Lambda_{x_{2}}^{\frac{\alpha\beta}{\alpha+\beta}}f
\|_{L_{x_{1}}^{\frac{2(\alpha+\beta)}{\alpha+\beta-2\alpha\beta}}}
\Big\|_{L_{x_{2}}^{2}}
\nonumber\\&\leq
C \Big\|  \|\Lambda_{x_{1}}^{\frac{\alpha\beta}{\alpha+\beta}}\Lambda_{x_{2}}^{\frac{\alpha\beta}{\alpha+\beta}}f
\|_{L_{x_{1}}^{2}}
\Big\|_{L_{x_{2}}^{2}}
\nonumber\\&\equiv C \|  \Lambda_{x_{1}}^{\frac{\alpha\beta}{\alpha+\beta}}\Lambda_{x_{2}}^{\frac{\alpha\beta}{\alpha+\beta}}f
\|_{L^{2}}
\nonumber\\&\leq
C\|\Lambda_{x_{1}}^{\alpha}f\|_{L^{2}}^{\frac{\beta}{\alpha+\beta}}
\|\Lambda_{x_{2}}^{\beta}f\|_{L^{2}}^{\frac{\alpha}{\alpha+\beta}}
,\nonumber
\end{align}
which yields \eqref{sptadd402}. It is obvious to check that $\alpha+\beta>2\alpha\beta$ when $(\alpha,\,\beta)\in \digamma$. Consequently, using \eqref{sptadd402}, we deduce from \eqref{sptadd401} that
\begin{align}\label{sptadd403}
&\frac{1}{2}\frac{d}{dt}\|\Lambda^{s} {\theta}(t)\|_{L^{2}}^{2}
+\|\Lambda_{x_{1}}^{\alpha}\Lambda^{s} {\theta}\|_{L^{2}}^{2}+
 \|\Lambda_{x_{2}}^{\beta}\Lambda^{s} {\theta}\|_{L^{2}}^{2}
 \nonumber\\
 &\leq C\|\nabla \theta\|_{L^{\frac{\alpha+\beta}{2\alpha\beta}}}
 \|\Lambda_{x_{1}}^{\alpha}\Lambda^{s}\theta\|_{L^{2}}^{\frac{2\beta}{\alpha+\beta}}
 \|\Lambda_{x_{2}}^{\beta}\Lambda^{s}\theta\|_{L^{2}}^{\frac{2\alpha}{\alpha+\beta}}\nonumber\\
 &\leq C_{1}\|\Lambda^{2-\frac{4\alpha\beta}{\alpha+\beta}} {\theta}\|_{L^{2}}
 \|\Lambda_{x_{1}}^{\alpha}\Lambda^{s}\theta\|_{L^{2}}^{\frac{2\beta}{\alpha+\beta}}
 \|\Lambda_{x_{2}}^{\beta}\Lambda^{s}\theta\|_{L^{2}}^{\frac{2\alpha}{\alpha+\beta}},
\end{align}
where we have used the fact $\frac{\alpha+\beta}{2\alpha\beta}\geq2$ due to $(\alpha,\,\beta)\in \digamma$. Taking $s=2-\frac{4\alpha\beta}{\alpha+\beta}\geq0$, \eqref{sptadd403} reduces to
\begin{align}
&\frac{1}{2}\frac{d}{dt}\|\Lambda^{2-\frac{4\alpha\beta}{\alpha+\beta}} {\theta}(t)\|_{L^{2}}^{2}
+\|\Lambda_{x_{1}}^{\alpha}\Lambda^{2-\frac{4\alpha\beta}{\alpha+\beta}} {\theta}\|_{L^{2}}^{2}+
 \|\Lambda_{x_{2}}^{\beta}\Lambda^{2-\frac{4\alpha\beta}{\alpha+\beta}} {\theta}\|_{L^{2}}^{2}
 \nonumber\\
 &\leq C\|\Lambda^{2-\frac{4\alpha\beta}{\alpha+\beta}} {\theta}\|_{L^{2}}
 \|\Lambda_{x_{1}}^{\alpha}\Lambda^{2-\frac{4\alpha\beta}
 {\alpha+\beta}}\theta\|_{L^{2}}^{\frac{2\beta}{\alpha+\beta}}
 \|\Lambda_{x_{2}}^{\beta}\Lambda^{2-\frac{4\alpha\beta}
 {\alpha+\beta}}\theta\|_{L^{2}}^{\frac{2\alpha}{\alpha+\beta}}
 \nonumber\\
 &\leq C\|\Lambda^{2-\frac{4\alpha\beta}{\alpha+\beta}} {\theta}\|_{L^{2}}
 (\|\Lambda_{x_{1}}^{\alpha}\Lambda^{2-\frac{4\alpha\beta}
 {\alpha+\beta}}\theta\|_{L^{2}}^{2}+
 \|\Lambda_{x_{2}}^{\beta}\Lambda^{2-\frac{4\alpha\beta}
 {\alpha+\beta}}\theta\|_{L^{2}}^{2}),\nonumber
\end{align}
which implies
\begin{align}\label{sptadd404}
&\frac{1}{2}\frac{d}{dt}\|\Lambda^{2-\frac{4\alpha\beta}{\alpha+\beta}} {\theta}(t)\|_{L^{2}}^{2}
+\|\Lambda_{x_{1}}^{\alpha}\Lambda^{2-\frac{4\alpha\beta}{\alpha+\beta}} {\theta}\|_{L^{2}}^{2}+
 \|\Lambda_{x_{2}}^{\beta}\Lambda^{2-\frac{4\alpha\beta}{\alpha+\beta}} {\theta}\|_{L^{2}}^{2}
\nonumber\\
 &\leq C_{2}\|\Lambda^{2-\frac{4\alpha\beta}{\alpha+\beta}} {\theta}\|_{L^{2}}
 (\|\Lambda_{x_{1}}^{\alpha}\Lambda^{2-\frac{4\alpha\beta}
 {\alpha+\beta}}\theta\|_{L^{2}}^{2}+
 \|\Lambda_{x_{2}}^{\beta}\Lambda^{2-\frac{4\alpha\beta}
 {\alpha+\beta}}\theta\|_{L^{2}}^{2}).
\end{align}
Therefore, if initially
$$\|\Lambda^{2-\frac{4\alpha\beta}{\alpha+\beta}} {\theta}_{0}\|_{L^{2}}\leq \epsilon$$
with $\epsilon\leq\frac{1}{ C_{2}}$, then we get by applying the bootstrapping
argument to \eqref{sptadd404} that for all $t\geq0$,
\begin{align}\label{sptadd405}
\|\Lambda^{2-\frac{4\alpha\beta}{\alpha+\beta}} {\theta}(t)\|_{L^{2}}\leq \|\Lambda^{2-\frac{4\alpha\beta}{\alpha+\beta}} {\theta}_{0}\|_{L^{2}}\leq \epsilon.
\end{align}
Coming back to \eqref{sptadd403}, one gets
\begin{align}\label{sptadd406}
&\frac{1}{2}\frac{d}{dt}\|\Lambda^{s} {\theta}(t)\|_{L^{2}}^{2}
+\|\Lambda_{x_{1}}^{\alpha}\Lambda^{s} {\theta}\|_{L^{2}}^{2}+
 \|\Lambda_{x_{2}}^{\beta}\Lambda^{s} {\theta}\|_{L^{2}}^{2}
  \nonumber\\
 &\leq C_{1}\|\Lambda^{2-\frac{4\alpha\beta}{\alpha+\beta}} {\theta}\|_{L^{2}}
 \|\Lambda_{x_{1}}^{\alpha}\Lambda^{s}\theta\|_{L^{2}}^{\frac{2\beta}{\alpha+\beta}}
 \|\Lambda_{x_{2}}^{\beta}\Lambda^{s}\theta\|_{L^{2}}^{\frac{2\alpha}{\alpha+\beta}}
  \nonumber\\
 &\leq C_{1}\epsilon(
 \|\Lambda_{x_{1}}^{\alpha}\Lambda^{s}\theta\|_{L^{2}}^{2}+
 \|\Lambda_{x_{2}}^{\beta}\Lambda^{s}\theta\|_{L^{2}}^{2}),
\end{align}
where in the last line we have used \eqref{sptadd405}.
Taking $\epsilon$ small again such that $\epsilon\leq\frac{1}{2C_{1}}$, we derive from \eqref{sptadd406} that
\begin{align}\label{sptadd407}
 \frac{d}{dt}\|\Lambda^{s} {\theta}(t)\|_{L^{2}}^{2}
+\|\Lambda_{x_{1}}^{\alpha}\Lambda^{s} {\theta}\|_{L^{2}}^{2}+
 \|\Lambda_{x_{2}}^{\beta}\Lambda^{s} {\theta}\|_{L^{2}}^{2}
 \leq0.
\end{align}
Integrating \eqref{sptadd407} in time yields the global solution in $H^{s}$ with small initial data. To show the uniqueness, we consider two solutions $\theta^{(1)}$ and $\theta^{(2)}$ of \eqref{SQG} with the same initial data. Denoting $\widetilde{\theta}=\theta^{(1)}-\theta^{(2)}$ and $\widetilde{u}=u^{(1)}-u^{(2)}$ with $\widetilde{u}_{1}=-\mathcal
{R}_{2}\widetilde{\theta}$ and $\widetilde{u}_{2}=\mathcal {R}_{1}\widetilde{\theta}$, one thus derives
\begin{equation}\label{dfdnmu19}
\left\{\aligned
&\partial_{t}\widetilde{\theta}+(u^{(1)}\cdot \nabla)\widetilde{\theta}+\Lambda_{x_{1}}^{2\alpha}\widetilde{\theta}+
\Lambda_{x_{2}}^{2\beta}\widetilde{\theta}=-(\widetilde{u} \cdot \nabla)\theta^{(2)},\\
&\widetilde{\theta}(x, 0)=0.
\endaligned\right.
\end{equation}
It comes out by applying the basic $L^{2}$-estimate to \eqref{dfdnmu19} that
\begin{align}
\frac{1}{2}\frac{d}{dt}\|\widetilde{\theta}(t)\|_{L^{2}}^{2}+
 \|\Lambda_{x_{1}}^{\alpha}\widetilde{\theta}\|_{L^{2}}^{2}+
 \|\Lambda_{x_{2}}^{\beta}\widetilde{\theta}\|_{L^{2}}^{2}
 &=-\int_{\mathbb{R}^{2}}{(\widetilde{u} \cdot \nabla)\theta^{(2)}\widetilde{\theta}\,dx}\nonumber\\&
 \leq C\|\widetilde{u}\|_{L^{\frac{2(\alpha+\beta)}{\alpha+\beta-2\alpha\beta}}}\|\nabla \theta^{(2)}\|_{L^{\frac{\alpha+\beta}{2\alpha\beta}}}
 \|\widetilde{\theta}\|_{L^{\frac{2(\alpha+\beta)}{\alpha+\beta-2\alpha\beta}}} \nonumber\\&
 \leq C \|\nabla \theta^{(2)}\|_{L^{\frac{\alpha+\beta}{2\alpha\beta}}}
 \|\widetilde{\theta}\|_{L^{\frac{2(\alpha+\beta)}{\alpha+\beta-2\alpha\beta}}}^{2} \nonumber\\
 &\leq C_{3}\|\Lambda^{2-\frac{4\alpha\beta}{\alpha+\beta}} \theta^{(2)}\|_{L^{2}}
 \|\Lambda_{x_{1}}^{\alpha}\widetilde{\theta}\|_{L^{2}}^{\frac{2\beta}{\alpha+\beta}}
 \|\Lambda_{x_{2}}^{\beta}\widetilde{\theta}\|_{L^{2}}^{\frac{2\alpha}{\alpha+\beta}}
 \nonumber\\
 &\leq C_{3}\epsilon
 (\|\Lambda_{x_{1}}^{\alpha}\widetilde{\theta}\|_{L^{2}}^{2}+
 \|\Lambda_{x_{2}}^{\beta}\widetilde{\theta}\|_{L^{2}}^{2}),\nonumber
 \end{align}
where in the last line we have used \eqref{sptadd405} again.
Taking $\epsilon$ small again such that $\epsilon\leq\frac{1}{C_{3}}$, we have
$$\frac{d}{dt}\|\widetilde{\theta}(t)\|_{L^{2}}^{2}\leq0,$$
which along with $\widetilde{\theta}(x, 0)=0$ yields $\widetilde{\theta}(t)\equiv0$, thus the uniqueness follows.

Finally, we show the desired decay estimate \eqref{oksffpp11}. In fact, following the proof of Proposition \ref{dfdsfp37}, one gets from \eqref{sptadd407} that
\begin{align}
  \|\Lambda^{s} \theta(t)\|_{L^2}\leq C_{0}(1+t)^{-\frac{(\alpha+\beta)(2-p)+2\min\{\alpha,\,\beta\}sp}{4\alpha\beta p}},\nonumber
\end{align}
which is nothing but \eqref{oksffpp11}.
As a result, it suffices to take
$$\epsilon=\min\left\{\frac{1}{2C_{1}},\ \frac{1}{C_{2}},\ \frac{1}{C_{3}}\right\}.$$
Therefore, this ends the proof of Theorem \ref{okspcrca11}.

\vskip .3in
\section{The proof of Theorem \ref{Th3}}\setcounter{equation}{0}
\label{sec4}
This section is devoted to the proof of Theorem \ref{Th3}, which depends on an appropriately modified Fourier Splitting method.
Letting $W=\theta-\widetilde{\theta}$, we deduce from \eqref{SQG} and \eqref{QGssd} that
\begin{equation}\label{diffSQG}
\left\{\aligned
&\partial_{t}W+\Lambda_{x_{1}}^{2\alpha}W+\Lambda_{x_{2}}^{2\beta}W=-(u \cdot \nabla)\theta, \\
&W(x, 0)=0.
\endaligned\right.
\end{equation}
A straightforward energy estimate of \eqref{diffSQG} yields
\begin{align}\label{diff02}
\frac{1}{2}\frac{d}{dt}\|W(t)\|_{L^{2}}^{2}+\|\Lambda_{x_{1}}^{\alpha}W\|_{L^{2}}^{2}+
 \|\Lambda_{x_{2}}^{\beta}W\|_{L^{2}}^{2}&=-\int_{\mathbb{R}^{2}}(u \cdot \nabla)\theta W\,dx\nonumber\\
 &=-\int_{\mathbb{R}^{2}}(u \cdot \nabla)\theta (\theta-\widetilde{\theta})\,dx\nonumber\\
 &=\int_{\mathbb{R}^{2}}(u \cdot \nabla)\theta \widetilde{\theta}\,dx\nonumber\\
 &=\int_{\mathbb{R}^{2}}(u \cdot \nabla)\widetilde{\theta} \theta\,dx
 \nonumber\\
 &\leq C\|\nabla\widetilde{\theta}\|_{L^{\infty}}
 \|u\|_{L^{2}}\|\theta\|_{L^{2}} \nonumber\\
 &\leq C\|\nabla\widetilde{\theta}\|_{L^{\infty}}
\|\theta\|_{L^{2}}^{2}.
\end{align}
To deal with $\nabla\widetilde{\theta}$, we infer from \eqref{QGssd} that
\begin{equation}
\partial_{t}\widehat{\widetilde{\theta}}+|\xi_{1}|^{2\alpha}\widehat{\widetilde{\theta}}
+|\xi_{2}|^{2\beta}\widehat{\widetilde{\theta}}=0,\nonumber
\end{equation}
which allows us to derive
$$\widehat{\widetilde{\theta}}(t,\xi)=e^{-(|\xi_{1}|^{2\alpha}+|\xi_{2}|^{2\beta})t}\widehat{\theta_{0}}(\xi).$$
We thus have
\begin{align}\label{diff03}
\|\nabla\widetilde{\theta}\|_{L^{\infty}}&\leq \int_{\mathbb{R}^{2}}|\xi||\widehat{\widetilde{\theta}}(t,\xi)|\,d\xi\nonumber\\
&= \int_{\mathbb{R}^{2}}|\xi|e^{-(|\xi_{1}|^{2\alpha}+|\xi_{2}|^{2\beta})t}|\widehat{\theta_{0}}(\xi)|\,d\xi
\nonumber\\
&\leq \left(\int_{\mathbb{R}^{2}}|\xi|^{p}e^{-p(|\xi_{1}|^{2\alpha}+|\xi_{2}|^{2\beta})t}
\,d\xi\right)^{\frac{1}{p}}
\left(\int_{\mathbb{R}^{2}}
|\widehat{\theta_{0}}(\xi)|^{\frac{p}{p-1}}\,d\xi\right)^{1-\frac{1}{p}}
\nonumber\\
&\leq \left(\int_{\mathbb{R}^{2}}(|\xi_{1}|^{p}e^{-p(|\xi_{1}|^{2\alpha}+|\xi_{2}|^{2\beta})t}
+|\xi_{2}|^{p}e^{-p(|\xi_{1}|^{2\alpha}+|\xi_{2}|^{2\beta})t})
\,d\xi\right)^{\frac{1}{p}}\|\theta_{0}\|_{L^p},
\end{align}
where we have used the fact $|\xi|^{\varrho}\leq |\xi_{1}|^{\varrho}+|\xi_{2}|^{\varrho}$ for $\varrho\in [1,2]$, and the Hausdorff-Young inequality
$$\|\widehat{f}\|_{L^{\frac{\sigma}{\sigma-1}}}\leq \|{f}\|_{L^{\sigma}},\quad \sigma\in [1,2].$$
Direct computations show
\begin{align}\label{diff04}
\int_{\mathbb{R}^{2}}|\xi_{1}|^{p}e^{-p(|\xi_{1}|^{2\alpha}+|\xi_{2}|^{2\beta})t}\,d\xi&=
\int_{\mathbb{R}} e^{-p|\xi_{2}|^{2\beta}t} \,d\xi_{2}\int_{\mathbb{R}}|\xi_{1}|^{p}e^{-p|\xi_{1}|^{2\alpha}t}\,d\xi_{1}\nonumber\\
&=C
\int_{0}^{\infty} e^{-pr^{2\beta}t} \,dr\ \int_{0}^{\infty}\lambda^{p}e^{-p\lambda^{2\alpha}t}\,d\lambda
\nonumber\\
&=C t^{-\frac{1}{2\beta}}t^{-\frac{p+1}{2\alpha}}\nonumber\\
&\leq C (1+t)^{-\frac{\alpha+(p+1)\beta}{2\alpha\beta}},\quad \mbox{for}\  t\geq1.
\end{align}
Similarly, we have
\begin{align}\label{diff05}
\int_{\mathbb{R}^{2}}|\xi_{2}|^{p}e^{-p(|\xi_{1}|^{2\alpha}+|\xi_{2}|^{2\beta})t}\,d\xi \leq C (1+t)^{-\frac{(p+1)\alpha+\beta}{2\alpha\beta}},\quad \mbox{for}\  t\geq1.
\end{align}
Putting \eqref{diff04} and \eqref{diff05} into \eqref{diff03} gives
\begin{align}
\|\nabla\widetilde{\theta}\|_{L^{\infty}}\leq C (1+t)^{-\frac{\min\{\alpha+(p+1)\beta,\, (p+1)\alpha+\beta\}}{2\alpha\beta p}}.\nonumber
\end{align}
This along with \eqref{diff02} yields
\begin{align}
&\frac{1}{2}\frac{d}{dt}\|W(t)\|_{L^{2}}^{2}+\|\Lambda_{x_{1}}^{\alpha}W\|_{L^{2}}^{2}+
 \|\Lambda_{x_{2}}^{\beta}W\|_{L^{2}}^{2}
 \nonumber\\&\leq C(1+t)^{-\frac{\min\{\alpha+(p+1)\beta,\, (p+1)\alpha+\beta\}}{2\alpha\beta p}}
\|\theta\|_{L^{2}}^{2} \label{dfghas6}\\
  &\leq C(1+t)^{-\frac{\min\{\alpha+(p+1)\beta,\, (p+1)\alpha+\beta\}}{2\alpha\beta p}}
 (1+t)^{-\frac{(\alpha+\beta)(2-p)}{2\alpha\beta p}}
 \nonumber\\
  &\leq C(1+t)^{-\frac{\min\{\alpha+(p+1)\beta,\, (p+1)\alpha+\beta\}+(\alpha+\beta)(2-p)}{2\alpha\beta p}}.\label{diff06}
\end{align}
Denoting
$$E(t)=\{\xi\in \mathbb{R}^2: \  |\xi_{1}|^{2\alpha}+|\xi_{2}|^{2\beta}\leq f(t)\},$$
we deduce from \eqref{pre1y323} that
\begin{align}
\|\Lambda_{x_{1}}^{\alpha}W\|_{L^{2}}^{2}+
 \|\Lambda_{x_{2}}^{\beta}W\|_{L^{2}}^{2}
 \geq f(t)\int_{\mathbb{R}^{2}}|\widehat{W}(t,\xi)|^{2}\,d\xi
 -f(t)\int_{E(t)}|\widehat{W}(t,\xi)|^{2}\,d\xi,\nonumber
\end{align}
which along with \eqref{diff06} implies
\begin{align}\label{diff07}
\frac{d}{dt}\|W(t)\|_{L^{2}}^{2}+2f(t)\|W\|_{L^{2}}^{2}\leq& 2f(t)\int_{E(t)}|\widehat{W}(t,\xi)|^{2}\,d\xi\nonumber\\&+ C(1+t)^{-\frac{\min\{\alpha+(p+1)\beta,\, (p+1)\alpha+\beta\}+(\alpha+\beta)(2-p)}{2\alpha\beta p}}.
\end{align}
We conclude from \eqref{diff07} that
\begin{align}\label{diff08}
\frac{d}{dt}\left(e^{2\int_{0}^{t}f(\tau)\,d\tau}\|W(t)\|_{L^{2}}^{2}\right)\leq& 2e^{2\int_{0}^{t}f(\tau)\,d\tau}f(t)\int_{E(t)}|\widehat{W}(t,\xi)|^{2}\,d\xi\nonumber\\&+ Ce^{2\int_{0}^{t}f(\tau)\,d\tau}(1+t)^{-\frac{\min\{\alpha+(p+1)\beta,\, (p+1)\alpha+\beta\}+(\alpha+\beta)(2-p)}{2\alpha\beta p}}.
\end{align}
Keeping in mind \eqref{diffSQG}, we get
\begin{align}
\widehat{W}(t,\xi)&=\int_{0}^{t}e^{-(|\xi_{1}|^{2\alpha}+|\xi_{2}|^{2\beta})(t-\tau)}\xi\cdot \widehat{u\theta}(\tau,\xi)\,d\tau.\nonumber
\end{align}
This further shows
\begin{align}\label{diff09}
|\widehat{W}(t,\xi)|&\leq|\xi|\int_{0}^{t}|\widehat{u\theta}(\tau,\xi)|\,d\tau\nonumber\\
&\leq|\xi|\int_{0}^{t}\|u\theta(\tau)\|_{L^{1}}\,d\tau\nonumber\\
&\leq |\xi|\int_{0}^{t}\|u(\tau)\|_{L^{2}}\|\theta(\tau)\|_{L^{2}}\,d\tau\nonumber\\
&\leq |\xi|\int_{0}^{t}\|\theta(\tau)\|_{L^{2}}^{2}\,d\tau
\nonumber\\
&\leq C|\xi|\int_{0}^{t} (1+\tau)^{-\frac{(\alpha+\beta)(2-p)}{2\alpha\beta p}}\,d\tau.
\end{align}
 Therefore, we deduce from \eqref{diff09} that
\begin{equation}\label{diff10}
|\widehat{W}(t,\xi)|^{2}\leq\left\{\aligned
&C|\xi|^{2}\ln ^{2}(1+t), \qquad \ \ \qquad \qquad \mbox{if}\ \ p=\frac{2(\alpha+\beta)}{2\alpha\beta+\alpha+\beta},\\
&C|\xi|^{2},\qquad \qquad \qquad \ \qquad \ \ \qquad\mbox{if}\ \ p<\frac{2(\alpha+\beta)}{2\alpha\beta+\alpha+\beta},\\
&C|\xi|^{2}(1+t)^{2-\frac{(\alpha+\beta)(2-p)}{\alpha\beta p}}, \ \quad \qquad \mbox{if}\ \ p>\frac{2(\alpha+\beta)}{2\alpha\beta+\alpha+\beta}.
\endaligned\right.
\end{equation}
Notice
\begin{align}\label{diff11}
\int_{E(t)}|\xi|^{2}\,d\xi&=\int_{E(t)}|\xi_{1}|^{2}\,d\xi+\int_{E(t)}|\xi_{2}|^{2}\,d\xi.
\end{align}
We obtain by direct computations that
\begin{align}\label{diff12}
\int_{E(t)}|\xi_{1}|^{2}\,d\xi&=\int_{\{\xi\in \mathbb{R}^2: \  |\xi_{1}|^{2\alpha}+|\xi_{2}|^{2\beta}\leq f(t)\}}|\xi_{1}|^{2}\,d\xi\nonumber\\
&=\int_{\{\xi_{2}\in \mathbb{R}: \  |\xi_{2}|\leq f(t)^{\frac{1}{2\beta}} \}}\int_{\{\xi_{1}\in \mathbb{R}: \  |\xi_{1}|\leq (f(t)-|\xi_{2}|^{2\beta})^{\frac{1}{2\alpha}}\}}|\xi_{1}|^{2}\,d\xi_{1}d\xi_{2}
\nonumber\\
&=\frac{2}{3}\int_{\{\xi_{2}\in \mathbb{R}: \  |\xi_{2}|\leq f(t)^{\frac{1}{2\beta}} \}}(f(t)-|\xi_{2}|^{2\beta})^{\frac{3}{2\alpha}}d\xi_{2}
\nonumber\\
&=\frac{2}{3\beta}\int_{0}^{1}(f(t)-\eta f(t))^{\frac{3}{2\alpha}}\eta^{\frac{1}{2\beta}-1}f(t)^{\frac{1}{2\beta}}\,d\eta
\nonumber\\
&=\frac{2}{3\beta}f(t)^{\frac{\alpha+3\beta}{2\alpha\beta}}\int_{0}^{1}(1-\eta )^{\frac{3}{2\alpha}}\eta^{\frac{1}{2\beta}-1}\,d\eta
\nonumber\\
&\triangleq Cf(t)^{\frac{\alpha+3\beta}{2\alpha\beta}}.
\end{align}
Similarly, one derives
\begin{align}\label{diff13}
\int_{E(t)}|\xi_{2}|^{2}\,d\xi\leq Cf(t)^{\frac{3\alpha+\beta}{2\alpha\beta}}.
\end{align}
Inserting \eqref{diff12} and \eqref{diff13} into \eqref{diff11}, we arrive at
\begin{align}\label{diff14}
\int_{E(t)}|\xi|^{2}\,d\xi&\leq C\left(f(t)^{\frac{\alpha+3\beta}{2\alpha\beta}}+f(t)^{\frac{3\alpha+\beta}{2\alpha\beta}}\right).
\end{align}
Combining \eqref{diff10} and \eqref{diff14} implies
\begin{align}\label{difgh89p}
 &\int_{E(t)}|\widehat{W}(t,\xi)|^{2}\,d\xi
\nonumber \\ &\leq
  \left\{
  \begin{aligned}
   C\left(f(t)^{\frac{\alpha+3\beta}{2\alpha\beta}}+f(t)^{\frac{3\alpha+\beta}{2\alpha\beta}}\right)\ln ^{2}(1+t), \qquad \quad \quad \ \  \mbox{if}\ \ p=\frac{2(\alpha+\beta)}{2\alpha\beta+\alpha+\beta}, \\
   C\left(f(t)^{\frac{\alpha+3\beta}{2\alpha\beta}}+f(t)^{\frac{3\alpha+\beta}{2\alpha\beta}}\right),\qquad \qquad \qquad \quad \ \ \ \quad\mbox{if} \ \ p<\frac{2(\alpha+\beta)}{2\alpha\beta+\alpha+\beta},\\
   C\left(f(t)^{\frac{\alpha+3\beta}{2\alpha\beta}}
   +f(t)^{\frac{3\alpha+\beta}{2\alpha\beta}}\right)
(1+t)^{2-\frac{(\alpha+\beta)(2-p)]}{\alpha\beta p}}, \quad  \mbox{if}\  \ p>\frac{2(\alpha+\beta)}{2\alpha\beta+\alpha+\beta}.
  \end{aligned}
  \right.
\end{align}
Now taking $f(t)=\frac{\kappa}{2(1+t)}$ with $\kappa$ suitable large, \eqref{difgh89p} and \eqref{diff08} reduce to
\begin{align}\label{diff15}
 &\int_{E(t)}|\widehat{W}(t,\xi)|^{2}\,d\xi
\nonumber \\ &\leq
  \left\{
  \begin{aligned}
   C(1+t)^{-\frac{\min\{\alpha+3\beta,\, 3\alpha+\beta\}}{2\alpha\beta}}\ln ^{2}(1+t),  \quad \quad \mbox{if}\ \ p=\frac{2(\alpha+\beta)}{2\alpha\beta+\alpha+\beta}, \\
   C(1+t)^{-\frac{\min\{\alpha+3\beta,\, 3\alpha+\beta\}}{2\alpha\beta}},\qquad   \qquad \quad \  \quad\mbox{if}\ \ p<\frac{2(\alpha+\beta)}{2\alpha\beta+\alpha+\beta},\\
   C(1+t)^{2-\frac{(\alpha+\beta)(2-p)}{\alpha\beta p}-\frac{\min\{\alpha+3\beta,\, 3\alpha+\beta\}}{2\alpha\beta}},  \qquad  \mbox{if}\ \ p>\frac{2(\alpha+\beta)}{2\alpha\beta+\alpha+\beta},
  \end{aligned}
  \right.
\end{align}
\begin{align}\label{diff16}
\frac{d}{dt}\left((1+t)^{\kappa}\|W(t)\|_{L^{2}}^{2}\right)\leq& C(1+t)^{\kappa-1}\int_{E(t)}|\widehat{W}(t,\xi)|^{2}\,d\xi\nonumber\\&+ C(1+t)^{\kappa-\frac{\min\{\alpha+(p+1)\beta,\, (p+1)\alpha+\beta\}+(\alpha+\beta)(2-p)}{2\alpha\beta p}}.
\end{align}
Integrating \eqref{diff16} in time gives
\begin{align}
(1+t)^{\kappa}\|W(t)\|_{L^{2}}^{2}\leq& C(1+t)^{\kappa}\int_{E(t)}|\widehat{W}(t,\xi)|^{2}\,d\xi\nonumber\\&+ C(1+t)^{\kappa+1-\frac{\min\{\alpha+(p+1)\beta,\, (p+1)\alpha+\beta\}+(\alpha+\beta)(2-p)}{2\alpha\beta p}},\nonumber
\end{align}
which further implies
\begin{align}\label{diff17}
\|W(t)\|_{L^{2}}^{2}\leq&  C\int_{E(t)}|\widehat{W}(t,\xi)|^{2}\,d\xi  + C(1+t)^{-\frac{\mathcal{A}}{2\alpha\beta p}}.
\end{align}
where $\mathcal{A}$ is given by \eqref{sdfg68}.
Therefore, for the case $p=\frac{2(\alpha+\beta)}{2\alpha\beta+\alpha+\beta}$, we have from \eqref{diff15} and \eqref{diff17} that
\begin{align}\label{diff18}
\|W(t)\|_{L^{2}}^{2}&\leq  C(1+t)^{-\frac{\min\{\alpha+3\beta,\, 3\alpha+\beta\}}{2\alpha\beta}}\ln ^{2}(1+t) + C(1+t)^{-\frac{\mathcal{A}}{2\alpha\beta p}}\nonumber\\
&\leq  C(1+t)^{-\frac{\mathcal{A}}{2\alpha\beta p}}\nonumber\\
&\leq  C(1+t)^{-\frac{\min\{\alpha+(p+1)\beta,\, (p+1)\alpha+\beta\}}{2\alpha\beta p}}.
\end{align}
For the case $p<\frac{2(\alpha+\beta)}{2\alpha\beta+\alpha+\beta}$, one gets
\begin{align}\label{diff19}
\|W(t)\|_{L^{2}}^{2}&\leq  C(1+t)^{-\frac{\min\{\alpha+3\beta,\, 3\alpha+\beta\}}{2\alpha\beta}}  + C(1+t)^{-\frac{\mathcal{A}}{2\alpha\beta p}}\nonumber\\
&\leq   C(1+t)^{-\frac{\min\{\min\{\alpha+3\beta,\, 3\alpha+\beta\}p,\  \mathcal{A}\}}{2\alpha\beta p}}.
\end{align}
Finally, for the case $p>\frac{2(\alpha+\beta)}{2\alpha\beta+\alpha+\beta}$, we obtain
\begin{align}\label{diff20}
\|W(t)\|_{L^{2}}^{2}&\leq  C(1+t)^{2-\frac{(\alpha+\beta)(2-p)}{\alpha\beta p}-\frac{\min\{\alpha+3\beta,\, 3\alpha+\beta\}}{2\alpha\beta}} + C(1+t)^{-\frac{\mathcal{A}}{2\alpha\beta p}}\nonumber\\
&\leq    C(1+t)^{-\frac{\min\{\min\{\alpha+3\beta,\, 3\alpha+\beta\}p+2(\alpha+\beta)(2-p)-4\alpha\beta p,\ \mathcal{A}\}}{2\alpha\beta p}}.
\end{align}
Collecting \eqref{diff18}, \eqref{diff19} and \eqref{diff20} yields the desired estimate.
Therefore, the proof of Theorem \ref{Th3} is completed.

\vskip .3in
\section{The proof of Theorem \ref{sdfthe61}}\setcounter{equation}{0}
\label{yetee1}
The proof can be performed as that for the proof of Theorem \ref{Th3}. For the sake of completeness, we sketch it here.
It follows from \eqref{dfghas6} that
\begin{align}\label{dfvgh701}
 \frac{1}{2}\frac{d}{dt}\|W(t)\|_{L^{2}}^{2}+\|\Lambda_{x_{1}}^{\alpha}W\|_{L^{2}}^{2}+
 \|\Lambda_{x_{2}}^{\beta}W\|_{L^{2}}^{2}
 &\leq C(1+t)^{-\frac{\min\{\alpha+3\beta,\, 3\alpha+\beta\}}{4\alpha\beta }}
\|\theta\|_{L^{2}}^{2}  \nonumber\\&\leq  C
(1+t)^{-\frac{\min\{\alpha+3\beta,\, 3\alpha+\beta\}}{4\alpha\beta }}\|\theta_{0}\|_{L^{2}}^{2}\nonumber\\&\leq C(1+t)^{-\frac{\min\{\alpha+3\beta,\, 3\alpha+\beta\}}{4\alpha\beta }}.
\end{align}
Just as for \eqref{diff08}, we get from \eqref{dfvgh701} that
\begin{align} \label{dfvgh702}
\frac{d}{dt}\left(e^{2\int_{0}^{t}f(\tau)\,d\tau}\|W(t)\|_{L^{2}}^{2}\right)\leq& 2e^{2\int_{0}^{t}f(\tau)\,d\tau}f(t)\int_{E(t)}|\widehat{W}(t,\xi)|^{2}\,d\xi\nonumber\\&+ Ce^{2\int_{0}^{t}f(\tau)\,d\tau}(1+t)^{-\frac{\min\{\alpha+3\beta,\, 3\alpha+\beta\}}{4\alpha\beta }}.
\end{align}
Moreover, \eqref{diff09} ensures
\begin{align}
|\widehat{W}(t,\xi)|
\leq |\xi|\int_{0}^{t}\|\theta(\tau)\|_{L^{2}}^{2}\,d\tau
\leq Ct|\xi|\|\theta_{0}\|_{L^{2}}^{2}.\nonumber
\end{align}
This along with \eqref{difgh89p} implies
$$\int_{E(t)}|\widehat{W}(t,\xi)|^{2}\,d\xi\leq C\left(f(t)^{\frac{\alpha+3\beta}{2\alpha\beta}}
   +f(t)^{\frac{3\alpha+\beta}{2\alpha\beta}}\right)
(1+t)^{2}.$$
We thus get by inserting the above estimate into \eqref{dfvgh702}
\begin{align} \label{dfvgh703}
\frac{d}{dt}\left(e^{2\int_{0}^{t}f(\tau)\,d\tau}\|W(t)\|_{L^{2}}^{2}\right)\leq& 2e^{2\int_{0}^{t}f(\tau)\,d\tau}f(t)\left(f(t)^{\frac{\alpha+3\beta}{2\alpha\beta}}
   +f(t)^{\frac{3\alpha+\beta}{2\alpha\beta}}\right)
(1+t)^{2}\nonumber\\&+ Ce^{2\int_{0}^{t}f(\tau)\,d\tau}(1+t)^{-\frac{\min\{\alpha+3\beta,\, 3\alpha+\beta\}}{4\alpha\beta }}.
\end{align}
Via taking $f(t)=\frac{\kappa}{2(1+t)}$ with $\kappa$ suitable large, \eqref{dfvgh703} becomes
\begin{align}
\frac{d}{dt}\left((1+t)^{\kappa}\|W(t)\|_{L^{2}}^{2}\right)\leq& C (1+t)^{\kappa-\frac{\min\{\alpha+3\beta,\, 3\alpha+\beta\}-2\alpha\beta}{2\alpha\beta }} + C(1+t)^{\kappa-\frac{\min\{\alpha+3\beta,\, 3\alpha+\beta\}}{4\alpha\beta }}.\nonumber
\end{align}
Integrating it in time yields
\begin{align}
(1+t)^{\kappa}\|W(t)\|_{L^{2}}^{2}\leq& C (1+t)^{\kappa+1-\frac{\min\{\alpha+3\beta,\, 3\alpha+\beta\}-2\alpha\beta}{2\alpha\beta }} + C(1+t)^{\kappa+1-\frac{\min\{\alpha+3\beta,\, 3\alpha+\beta\}}{4\alpha\beta }}.\nonumber
\end{align}
This further shows
\begin{align}
\|W(t)\|_{L^{2}}^{2}\leq& C (1+t)^{-\frac{\min\{\alpha+3\beta,\, 3\alpha+\beta\}-4\alpha\beta}{2\alpha\beta }} + C(1+t)^{-\frac{\min\{\alpha+3\beta,\, 3\alpha+\beta\}-4\alpha\beta}{4\alpha\beta }}\nonumber\\
\leq& C(1+t)^{-\frac{\min\{\alpha+3\beta,\, 3\alpha+\beta\}-4\alpha\beta}{4\alpha\beta }},\nonumber
\end{align}
where we have used the restriction $\min\{\alpha+3\beta,\, 3\alpha+\beta\}>4\alpha\beta$. Thus, this gives \eqref{dfcnmza1} and and completes the proof of Theorem \ref{sdfthe61}.

\appendix

\vskip .3in
\section{The case \eqref{SQG} with $\alpha=\beta=\frac{1}{2}$}
\label{appdexA}

In this appendix, we consider \eqref{SQG} with $\Lambda_{x_{1}} \theta
+ \Lambda_{x_{2}}\theta$, namely,
\begin{equation}\label{crSQG}
\left\{\aligned
&\partial_{t}\theta+(u \cdot \nabla)\theta+\Lambda_{x_{1}}\theta+\Lambda_{x_{2}}\theta=0, \\
&{u}=\mathcal {R}^{\perp}\theta,\\
&\theta(x, 0)=\theta_{0}(x).
\endaligned\right.
\end{equation}
More precisely, we have the following small data global existence and the decay estimate results of (\ref{crSQG}).
\begin{thm}\label{apedth1}
Let $\theta_{0}\in H^{s}(\mathbb{R}^{2})$ with $s>1$. If there exists an absolute constant $\delta>0$ such that
\begin{align}\label{bncond11}
\|\theta_{0}\|_{L^{\infty}}\leq \delta,
\end{align}
then (\ref{crSQG}) has a unique global smooth solution $\theta$ such that for any given $t>0$,
\begin{align}\label{bncond12}
\|\Lambda^{s} \theta(t)\|_{L^2}^{2}+ \int_{0}^{t}\|\Lambda^{s+\frac{1}{2}} \theta(\tau)\|_{L^2}^{2}\,d\tau\leq \|\Lambda^{s} \theta_{0}\|_{L^2}^{2}.
\end{align}
Additionally, if $\theta_{0}\in L^{p}(\mathbb{R}^{2})$ with $p\in [1,\,2]$, then the solution as stated above admits the decay estimate
\begin{align}\label{uchaa02}
 \|\Lambda^{s} \theta(t)\|_{L^2}\leq C_{0}(1+t)^{-\frac{2+(s-1)p}{ p}},
\end{align}
where the constant $C_{0}>0$ depends only on $s,\,p$ and $\theta_{0}$.
\end{thm}

\begin{proof}
Applying $\Lambda^{s}$ to (\ref{crSQG}) and multiplying it by $\Lambda^{s}\theta$ imply
\begin{align}\label{uchaa03}
\frac{1}{2}\frac{d}{dt}\|\Lambda^{s}{\theta}(t)\|_{L^{2}}^{2}
+\|\Lambda^{s}\Lambda_{x_{1}}^{\frac{1}{2}}{\theta}\|_{L^{2}}^{2}
+\|\Lambda^{s}\Lambda_{x_{2}}^{\frac{1}{2}}{\theta}\|_{L^{2}}^{2}
&=-\int_{\mathbb{R}^{2}}\Lambda^{s}(u \cdot \nabla\theta) \Lambda^{s}\theta\,dx.
\end{align}
By \eqref{yzz1}, we obtain
\begin{align} \label{uchaa04}
\left|\int_{\mathbb{R}^{2}}\Lambda^{s}(u \cdot \nabla\theta) \Lambda^{s}\theta\,dx\right|
&=\left|\int_{\mathbb{R}^{2}}[\Lambda^{s}, u\cdot\nabla]\theta \Lambda^{s}\theta\,dx\right|
\nonumber\\ &\leq C(\|\nabla u\|_{L^{2s+1}}\|\Lambda^{s}\theta\|_{L^{\frac{2s+1}{s}}}+
\|\nabla
\theta\|_{L^{2s+1}}\|\Lambda^{s}u\|_{L^{\frac{2s+1}{s }}})
\|\Lambda^{s}\theta\|_{L^{\frac{2s+1}{s}}}
\nonumber\\ &\leq C\|\nabla\theta\|_{L^{2s+1}}\|\Lambda^{s}\theta\|_{L^{\frac{2s+1}{s}}}^{2}
\nonumber\\ &\leq C\|\nabla\theta\|_{\dot{B}_{\infty,\infty}^{-1}}^{\frac{2s-1}{2s+1}}
 \|\nabla\theta\|_{\dot{B}_{2,2}^{s-\frac{1}{2}}}^{\frac{2}{2s+1}}
 (\|\Lambda^{s}\theta\|_{\dot{B}_{\infty,\infty}^{-s}}^{\frac{1}{2s+1}}
 \|\Lambda^{s}\theta\|_{\dot{B}_{2,2}^{\frac{1}{2}}}^{\frac{2s}{2s+1}})^{2}
 \nonumber\\ &\leq C\|\theta\|_{\dot{B}_{\infty,\infty}^{0}}^{\frac{2s-1}{2s+1}}
 \|\Lambda^{s+\frac{1}{2}}\theta\|_{L^{2}}^{\frac{2}{2s+1}}
 (\|\theta\|_{\dot{B}_{\infty,\infty}^{0}}^{\frac{1}{2s+1}}
 \|\Lambda^{s+\frac{1}{2}}\theta\|_{L^{2}}^{\frac{2s}{2s+1}})^{2}
 \nonumber\\ &\leq C\|\theta\|_{L^{\infty}}\|\Lambda^{s+\frac{1}{2}}\theta\|_{L^{2}}^{2}\nonumber\\ &\leq C\|\theta_{0}\|_{L^{\infty}}\|\Lambda^{s+\frac{1}{2}}\theta\|_{L^{2}}^{2}
\end{align}
where we have used the sharp interpolation inequality (see \cite[Theorem 2.42]{BCD})
\begin{eqnarray}
\|f\|_{L^{p}}\leq C\|f\|_{\dot{B}_{\infty,\infty}^{-\gamma}}^{\frac{p-q}{p}}
 \|f\|_{\dot{B}_{q,q}^{\gamma(\frac{p}{q}-1)}}^{\frac{q}{p}}\nonumber
 \end{eqnarray}
with $\gamma>0$ and $1\leq q<p<\infty$. According to Plancherel's theorem, one gets
\begin{align} \label{uchaa05}
\|\Lambda^{s}\Lambda_{x_{1}}^{\frac{1}{2}}{\theta}\|_{L^{2}}^{2}
+\|\Lambda^{s}\Lambda_{x_{2}}^{\frac{1}{2}}{\theta}\|_{L^{2}}^{2}\approx \|\Lambda^{s+\frac{1}{2}}\theta\|_{L^{2}}^{2}.
\end{align}
Combining \eqref{uchaa03}, \eqref{uchaa04} and \eqref{uchaa05} yields
\begin{align}
\frac{1}{2}\frac{d}{dt}\|\Lambda^{s}{\theta}(t)\|_{L^{2}}^{2}
+\|\Lambda^{s+\frac{1}{2}}{\theta}\|_{L^{2}}^{2}
\leq C\|\theta_{0}\|_{L^{\infty}}\|\Lambda^{s+\frac{1}{2}}\theta\|_{L^{2}}^{2},\nonumber
\end{align}
which along with the assumption \eqref{bncond11} implies
\begin{align}\label{uchaa06}
 \frac{d}{dt}\|\Lambda^{s}{\theta}(t)\|_{L^{2}}^{2}
+\|\Lambda^{s+\frac{1}{2}}{\theta}\|_{L^{2}}^{2}
\leq 0.
\end{align}
Integrating \eqref{uchaa06} in time yields \eqref{bncond12}. It is quite easy to show the uniqueness as $s>1$. To show the decay estimate \eqref{uchaa02}, we still have the following estimate due to Proposition \ref{Spr1}
\begin{align}\label{uchaa09}
\|\theta(t)\|_{L^p}\leq \|\theta_0\|_{L^p}.
\end{align}
Thanks to $p\in [1,\,2]$, it follows from the direct interpolation that
$$\|\Lambda^{s}{\theta}\|_{L^{2}}\leq C\|{\theta}\|_{L^{p}}^{\frac{p}{(2s-1)p+4}}
\|\Lambda^{s+\frac{1}{2}}{\theta}\|_{L^{2}}^{\frac{2(s-1)p+4}{(2s-1)p+4}},$$
which together with \eqref{uchaa09} gives
$$\|\Lambda^{s}{\theta}\|_{L^{2}}\leq C\|{\theta_{0}}\|_{L^{p}}^{\frac{p}{(2s-1)p+4}}
\|\Lambda^{s+\frac{1}{2}}{\theta}\|_{L^{2}}^{\frac{2(s-1)p+4}{(2s-1)p+4}}\leq C
\|\Lambda^{s+\frac{1}{2}}{\theta}\|_{L^{2}}^{\frac{2(s-1)p+4}{(2s-1)p+4}}.$$
This further yields
$$C^{-1}\|\Lambda^{s}{\theta}\|_{L^{2}}^{\frac{(2s-1)p+4}{(s-1)p+2}}\leq
\|\Lambda^{s+\frac{1}{2}}{\theta}\|_{L^{2}}^{2},$$
which along with \eqref{uchaa06} shows
\begin{align}\label{uchaa10}
 \frac{d}{dt}\|\Lambda^{s}{\theta}(t)\|_{L^{2}}^{2}
+C^{-1}\|\Lambda^{s}{\theta}\|_{L^{2}}^{\frac{(2s-1)p+4}{(s-1)p+2}}
\leq 0.
\end{align}
Applying Lemma \ref{lem111} to \eqref{uchaa10}, we deduce
$$\|\Lambda^{s} \theta(t)\|_{L^2}\leq C_{0}(1+t)^{-\frac{2+(s-1)p}{ p}}.$$
Consequently, this ends the proof of Theorem \ref{apedth1}.
\end{proof}

\vskip .3in
\section{The decay estimates of linear part}
\label{appdexB}
It is not hard to check that the upper bounds of \eqref{linaddtq1} and \eqref{linaddtq3} follow from the proof of Proposition \ref{dfdsfp37}.
As a result, it suffices to show the lower bound of \eqref{linaddtq3} under the assumption \eqref{linaddtq2}. In fact, it follows from \eqref{linaddtq2} that there exists two constants $\delta>0$, $\widetilde{c}>0$ such that for any $0<\rho\leq\delta$
\begin{align}\label{linadaa1}
 \rho^{-\frac{(\alpha+\beta)(2-p)+2\min\{\alpha,\,\beta\}sp}{2\alpha\beta p}}\int_{E(\rho)} |\xi|^{2s}|\widehat{\theta_{0}}(\xi)|^{2}\;d\xi\geq \widetilde{c}.
\end{align}
Let
$$E(\rho(t))\triangleq\{\xi\in \mathbb{R}^2: \  |\xi_{1}|^{2\alpha}+|\xi_{2}|^{2\beta}\leq \rho(t)\},$$
where $\rho(t)$ is nonincreasing, continuous radius to be
determined later.
Now we write the solution $\widetilde{\theta}$ as
$$\widetilde{\theta}(x,t)=e^{-(\Lambda_{x_{1}}^{2\alpha}
+\Lambda_{x_{2}}^{2\beta})t}\theta_{0},$$
which allows us to derive
\begin{align} \label{linadaa2}
 \|\Lambda^{s}\widetilde{\theta}\|_{L^{2}}^{2}
&= \int_{\mathbb{R}^{2}}|\xi|^{2s}|\widetilde{\theta}(\xi)|^{2}\,d\xi
\nonumber\\ &=\int_{\mathbb{R}^{2}}|\xi|^{2s}e^{-2(|\xi_{1}|^{2\alpha}
+|\xi_{2}|^{2\beta})t}| \widehat{\theta_{0}}(\xi)|^{2}\,d\xi
\nonumber\\ &\geq \int_{E(\rho(t))}|\xi|^{2s}e^{-2(|\xi_{1}|^{2\alpha}
+|\xi_{2}|^{2\beta})t}|\widehat{\theta_{0}}(\xi)|^{2}\,d\xi
\nonumber\\ &\geq \int_{E(\rho(t))}|\xi|^{2s}e^{-2 t\rho(t)}|\widehat{\theta_{0}}(\xi)|^{2}\,d\xi
 \nonumber\\ &=e^{-2 t\rho(t)}\int_{E(\rho(t))}|\xi|^{2s}|\widehat{\theta_{0}}(\xi)|^{2}\,d\xi.
\end{align}
Now we take $\rho(t)=\delta(1+t)^{-1}$, then it follows from \eqref{linadaa2} that
\begin{align}
 \|\Lambda^{s}\widetilde{\theta}\|_{L^{2}}^{2}
&\geq e^{-2 t\rho(t)}\rho(t)^{\frac{(\alpha+\beta)(2-p)+2\min\{\alpha,\,\beta\}sp}{2\alpha\beta p}}\rho(t)^{-\frac{(\alpha+\beta)(2-p)+2\min\{\alpha,\,\beta\}sp}{2\alpha\beta p}}\int_{E(\rho(t))}|\xi|^{2s}|\widehat{\theta_{0}}(\xi)|^{2}\,d\xi\nonumber\\&\geq \widetilde{c}e^{-2 t\rho(t)}\rho(t)^{\frac{(\alpha+\beta)(2-p)+2\min\{\alpha,\,\beta\}sp}{2\alpha\beta p}}\nonumber\\&\geq C_{1}(1+t)^{-\frac{(\alpha+\beta)(2-p)+2\min\{\alpha,\,\beta\}sp}{2\alpha\beta p}},\nonumber
\end{align}
where we have used \eqref{linadaa1}. This immediately implies
the lower bound
$$\|\Lambda^{s}\widetilde{\theta}(t)\|_{L^2}\geq C_{1}(1+t)^{-\frac{(\alpha+\beta)(2-p)+2\min\{\alpha,\,\beta\}sp}{4\alpha\beta p}}.$$
\vskip .3in

\section*{Acknowledgements}
This work is supported by the National Natural Science Foundation of China (No. 11701232), the Natural Science Foundation of Jiangsu Province (No. BK20170224) and the Qing Lan Project of Jiangsu Province.

\end{document}